\documentclass[12pt]{article}
\usepackage{amsthm, amsmath, amssymb} 
\usepackage{graphicx}
\usepackage{enumerate}
\usepackage{natbib}
\usepackage{url} 
\usepackage{float}
\usepackage[colorlinks=true, allcolors=blue]{hyperref}
\usepackage{multirow}
\usepackage{makecell}
\usepackage{rotating}
\usepackage{subcaption} 
\usepackage{enumitem} 
\usepackage{xcolor} 
\usepackage{caption}

\usepackage{bibunits}

\defaultbibliographystyle{agsm} 
\defaultbibliography{reference} 

\newcommand{\blind}{0}
\newcommand{\TPS}{\textcolor{black}}
\usepackage[margin=1in]{geometry}
\numberwithin{equation}{section}
\newtheorem{theorem}{Theorem}[section]
\newtheorem{proposition}[theorem]{Proposition}
\newtheorem{lemma}[theorem]{Lemma}
\newtheorem{corollary}[theorem]{Corollary}
\newtheorem{definition}{Definition}[section]
\newtheorem{example}{Example}[section]
\newtheorem{remark}{Remark}[section]

\makeatletter
\renewenvironment{proof}[1][\proofname]{\par
  \pushQED{\qed}%
  \normalfont \topsep6\p@\@plus6\p@\relax
  \trivlist
  \item[\hskip\labelsep
        \itshape
    #1\@addpunct{.}]%
}{%
  \popQED\endtrivlist\@endpefalse
}
\makeatother

\usepackage{algorithm}
\usepackage[compatible]{algpseudocode} 
\makeatletter
\newenvironment{breakablealgorithm}
{
	\begin{center}
		\refstepcounter{algorithm}
		\hrule height0.8pt depth0pt \kern0pt
		\renewcommand{\caption}[2][\relax]{
			{\raggedright\textbf{\ALG@name~\thealgorithm} ##2\par}%
			\ifx\relax##1\relax 
			\addcontentsline{loa}{algorithm}{\protect\numberline{\thealgorithm}##2}%
			\else 
			\addcontentsline{loa}{algorithm}{\protect\numberline{\thealgorithm}##1}%
			\fi
			\kern2pt\hrule\kern2pt
		}
	}{
		\kern0pt\hrule\relax
	\end{center}
}
\makeatother

\begin{document}

\def\spacingset#1{\renewcommand{\baselinestretch}%
{#1}\small\normalsize} \spacingset{1}


\if0\blind
{
  \title{\bf Simultaneous Inference for Nonlinear Time Series, a Sieve M-regression Approach}
  
  \author{Tianpai Luo\footnotemark[1]\thanks{E-mail address: ltp21@mails.tsinghua.edu.cn;} \qquad Zhou Zhou\thanks{E-mail address: zhou.zhou@utoronto.ca. Zhou Zhou is the corresponding author and acknowledges Funding number 489079 from NSERC of Canada.}
  \vspace{0.2cm}\\ 
  \footnotemark[1] Department of Statistics and Data Science, Tsinghua University
  \vspace{0.2cm}\\ 
  \footnotemark[2] Department of Statistical Sciences, University of Toronto}
  
  \date{}
  \maketitle
} \fi

\if1\blind
{
  \bigskip
  \bigskip
  \bigskip
  \begin{center}
    {\LARGE\bf Simultaneous Inference for Nonlinear Time Series, a Sieve M-regression Approach}
\end{center}
  \medskip
} \fi

\bigskip
\begin{abstract}
This paper studies simultaneous inference of conditional distributions in nonlinear time series from a sieve M-regression perspective. 
Existing literature on sieve M-regression has primarily focused on pointwise asymptotics, leaving the development of uncertainty quantification over the entire predictor space unexplored. We address this gap by establishing a uniform Bahadur representation for the sieve M-estimator, accommodating dependent data and a growing number of sieve basis functions. A novel high-dimensional empirical process theory is developed for temporally dependent data, and a specifically designed M-decomposition method is utilized to control high-dimensional complexities. Building on this representation, we develop a convex Gaussian approximation to characterize the asymptotic behavior of the estimator and construct valid simultaneous confidence regions (SCRs). To facilitate practical implementation, we introduce a self-convolved bootstrap algorithm that accurately approximates the distribution of the maximal deviation. Our inferential framework is supported by rigorous error bounds and validated through numerical simulations and real data applications. 
\end{abstract}

\noindent%
{\it Keywords:} Bahadur representation, empirical process, self-convolved bootstrap, method of sieves, convex Gaussian approximation 
\vfill

\newpage

\begin{bibunit}
\newcommand{\mf}{\mathbf}
\newcommand{\mr}{\mathrm}
\newcommand{\mb}{\mathbb}

\section{Introduction}
\label{sec: introduction}
Given a random variable $Y$ and a set of explanatory random variables ${\bf X}\in\mathbb{R}^d$, studying the conditional distribution of $Y$ given $\mf X$, along with conducting valid inference, is a fundamental problem in modern statistics and machine learning. In this paper,  we consider stationary time series $(\mf{X}_i,Y_i),i=1,\dots,n$, where the primary object of interest is a  conditional quantity $Q(Y|\mf X)$, which characterizes certain aspects of the conditional distribution of the response $Y$ given the explanatory variables $\mf X$. Typically, $Q(Y|\mf X)$ can be estimated via M-regression with appropriate choices of loss functions $\rho(\cdot)$.
For example, the least squares loss is used when  $Q(Y|\mf X)=\mb E(Y|\mf X)$ representing the conditional mean;  see \cite{fan2008nonlinear}, \cite{liu2010simultaneous}, and \cite{vogt2012nonparametric}. For quantile regression in which the least absolute deviation or the check function loss is used, $Q(Y|\mf X)$ can be the $\tau$-th quantile of $Y$ given $\mf X$, $Q(Y|\mf X)=\inf\{s:\mr P(Y\leq s|\mf X)>\tau\}$; see \cite{cai2002regression}, \cite{zhou2010Non}, and \cite{tu2022quantile}. Moreover, certain $Q(Y|\mf X)$ can be estimated by robust or other types of M-regression, as discussed in \cite{kong2010uniform}, \cite{hardle2010confidence}, and \cite{wu2018gradient}. After specifying the quantity $Q(Y|\mf X)$ of interest, the relationship between ${\bf X}$ and $Y$ is written in the following nonlinear time series M-regression form
\begin{equation}
    Y_i=Q_0(\mf X_i)+\varepsilon_i,\quad i=1,\dots,n,\label{eq:basic model}
\end{equation}
where $\{\varepsilon_i\}$ is a stationary error process which may be dependent on $\{\bf X_i\}$ and 
\begin{align}
   \hat{Q}&=: \arg\min\limits_{Q\in\mathfrak{Q}_n} \sum_{i=1}^n \rho(Y_i-Q(\mf X_i)),\label{eq:loss function}
\end{align}
with $\mathfrak{Q}_n$ denoting a nonparametric function space which expands as $n$ increases. Here we abbreviate $Q(Y|\mf X)$ as $Q_0(\mf X)$ to simplify notation. We consider a broad family of loss functions $\rho(\cdot)$ with left derivative $\psi(\cdot)$, including quantile regression with $\rho(x)=\tau x^{+}+(1-$ $\tau)(-x)^{+}, 0<\tau<1$, where $x^{+}=\max (x, 0)$; Huber's loss with $\rho(x)=\left(x^2 \mathbf{1}_{|x| \leq c}\right) / 2+\left(c|x|-c^2 / 2\right) \mathbf{1}_{|x|>c}, c>0$; and the $\mathcal{L}^q$ regression loss with $\rho(x)=|x|^q, 1 < q \leq 2$, among many others. This paper concerns statistical inference of $Q_0(\mf x)$ based on $\hat Q(\mf x)$ uniformly for all $\mf x$ in a given region $\mathcal{X}\subset \mb R^d$.



\subsection{Overview of our results and novelties}

The first and foremost contribution of this paper is the development of an empirical process framework for nonlinear time series sieve M-estimation, in which high dimensionality serves as the main challenge. The sieve approach necessarily introduces a growing number of basis functions as the sample size increases, leading to moderately high-dimensional parameter spaces. 
For time series linear M-regression $y_i=\theta^\top\mf x_i+\epsilon_i$ with fixed-dimensional parameter $\theta$, \cite{wu2007m} developed a martingale decomposition technique for empirical processes under weak dependence. \cite{wu2007m} obtained a nearly optimal rate for the Bahadur representation of the regression estimator. This martingale decomposition method has since become a fundamental tool in the time series literature and has been applied in various settings, including \cite{lee2019martingale}, \cite{zhang2021high}, and \cite{liu2024self}.
However, the martingale decomposition technique of \cite{wu2007m} (see Lemmas 4 and 5 therein for details) includes a Taylor expansion step that generates a number of terms growing exponentially with the dimension of $\theta$. When the dimension diverges, this leads to a severe curse of dimensionality. 
Consequently, the empirical process theory developed in \cite{wu2007m} is not applicable to our high-dimensional setting. To address this difficulty, we extend the empirical process theory for independent data developed in \cite{van1996weak} and \cite{gine2016mathematical}
by introducing a specifically designed M-dependent decomposition technique for the time series $(Y_i,\mf X_i), i=1,\dots,n$. This new approach yields sharp maximal inequalities for a broad class of time series regression models with a diverging number of parameters. Moreover, the high-dimensional empirical process framework established in this paper is expected to be applicable to a substantially wider range of time series problems.

As a direct consequence, we obtain the first uniform Bahadur representation for sieve estimators in nonlinear time series, thereby overcoming the technical barrier that has long hindered statistical inference on $Q_0(\mf x)$. The sharpness of our Bahadur representation can be assessed relative to the best known i.i.d. benchmark for M-estimation. In particular, for a representative non-smooth M-estimation, namely quantile regression, our Bahadur representation yields the order $(K/n)^{3/4}$ up to logarithmic factors under sufficiently weak temporal dependence and regular conditions, where $K\to \infty$ is the number of sieve bases used in the regression. This matches the best known rate for partitioning-based sieve M-estimators under i.i.d. data as established in \cite{cattaneo2025uniform}.
Based on the Bahadur representation, we further develop a time series convex Gaussian approximation theory to capture the uniform stochastic behavior of the moderately high-dimensional sieve M-estimators, which finally provides the foundation for constructing simultaneous confidence regions (SCRs) and conducting uniform statistical inference of $\hat Q(\mf x)$ at an asymptotically correct significance level.

The second major contribution of this paper is the implementation of a self-convolved bootstrap algorithm for simultaneous inference on $Q_0(\mf x)$. Developing fully data-driven resampling methods for nonparametric inference in nonlinear time series regression has long been challenging. Existing approaches typically rely on closed-form limiting distributions and are therefore not directly driven by the observed data; see, for example, the simulation-based procedure in \cite{liu2010simultaneous}. As a result, such methods may be slow in approximating the finite-sample behavior of the statistics of interest and often require the estimation of additional nuisance parameters, which can substantially affect inferential accuracy in finite samples. 

To the best of our knowledge, this is the first work to develop a general data-driven inferential framework for nonlinear time series under a nonparametric M-regression setting. The self-convolved bootstrap was originally proposed by \cite{liu2024self} for time-varying coefficient linear models. However, the technical development in the present paper is fundamentally different for two main reasons. First, the bootstrap theory in \cite{liu2024self} relies on fixed-dimensional empirical process theory for nonstationary time series, whereas our bootstrap procedure is established under a high-dimensional time series empirical process framework based on substantially different techniques. Second, \cite{liu2024self} focuses on inference for the cumulative regression function, obtained by integrating the M-estimators over time. It remains unclear whether their bootstrap can be directly applied to the M-estimators themselves without integration or extended to nonlinear regression settings. In this paper, we show that the self-convolved bootstrap can be successfully adapted to conduct direct inference on nonparametric M-estimators in nonlinear time series regression.

The self-convolved bootstrap offers several notable advantages. First, it requires only the convolution of local M-estimators with i.i.d. auxiliary standard normal variables, thereby avoiding the need to estimate additional nuisance parameters, such as conditional density functions. Second, it involves the selection of only a single tuning parameter, which reduces the burden and uncertainty in statistical inference associated with multiple tuning choices. Third, it fully leverages our nonasymptotic approximation rates derived from Bahadur representations and Gaussian approximations and is shown to be relatively fast and accurate in approximating the uniform probabilistic behavior of the M-estimators. 
Consequently, the self-convolved bootstrap provides a parsimonious, theoretically grounded, and broadly applicable tool for statistical inference.

\subsection{Overview of related literature}

The literature on nonparametric estimation and inference for time series M-regression has primarily focused on kernel-based methods. Fundamental results, including Bahadur representations and asymptotic normality for kernel nonparametric M-regression, have been established in \cite{kong2010uniform}, \cite{zhao2014asymptotics}, \cite{wu2018gradient}, \cite{karmakar2022simultaneous}, and \cite{liu2024self}, among others. In particular, \cite{wu2018gradient}, \cite{karmakar2022simultaneous}, and \cite{liu2024self} investigated simultaneous inference for kernel M-estimators under dependent observations; however, their frameworks are restricted to time-varying coefficient models of the form $Y_i=\mu(\mf X_i,\theta(t_i))+\varepsilon_i$, $i=1,2,\cdots, n$, where the function $\mu(\cdot,\cdot)$ is known, $\theta(t_i)$ is an unknown time-varying parameter, and $t_i=i/n$ are fixed design points. For model \eqref{eq:basic model} with random design, \cite{kong2010uniform} derived a uniform Bahadur representation for the kernel M-estimator, while \cite{alman2022expectile} established pointwise asymptotic normality for kernel expectile regression. To the best of our knowledge, however, simultaneous nonparametric inference for a general class of nonlinear time series M-regression models has not yet been fully developed.

On the other hand, the sieve method \citep{chen2007large} is a widely used basis approximation approach that optimizes an empirical criterion over a sequence of approximating spaces formed by linear or nonlinear combinations of basis functions. Compared with the kernel method, sieve estimation of \eqref{eq:basic model} offers a significant computational advantage, as it requires fitting the model only once, whereas the kernel method requires refitting the model for every $\mf x \in \mathbb{R}^d$. The construction of the sieve estimator will be described in detail in Section \ref{sec:Preliminaries}. 
The literature on sieve-based nonparametric estimation and inference for time series M-regression remains relatively limited. Among the existing works, \cite{CHEN2015uniformsieve} established uniform convergence rates and asymptotic normality for certain functionals of least squares sieve estimators under weak dependence. Under the least squares loss, \cite{ding2021simultaneous} constructed simultaneous confidence regions for sieve estimators in time-inhomogeneous additive time series models.
For completeness, we also note some related developments in the i.i.d. setting. 
Specifically, \cite{chen2014seiveM} developed an asymptotic theory for inference on irregular functionals of sieve M-estimators, while \cite{SU2016sieveIV} studied sieve instrumental variable quantile regression and proposed a specification test. Recently, \cite{cattaneo2025uniform} established uniform Bahadur representations and inference theory for partitioning-based sieve M-estimators under i.i.d. data. 
Empirical process theory constitutes an essential tool for analyzing the asymptotic behavior of M-estimators. In the classical i.i.d. setting, the systematic treatment of empirical process theory in \cite{van1996weak}, together with earlier developments in \cite{pollard1990empirical,pollard1991asymptotics}, shows that local convexity and stochastic equicontinuity lead to a local linearization of the empirical criterion,
thereby yielding Bahadur representation and asymptotic normality even under nonsmooth loss functions. Extensions of empirical process techniques have been investigated in various directions, yet the joint treatment of dependence and high dimensionality remains largely underexplored. For dependent data, empirical process results and maximal inequalities have been developed \citep{wu2005bahadur,doukhan2007probability,dehling2009new,phandoidaen2022empirical}; 
however, these studies focus primarily on empirical processes indexed by function classes of fixed complexity, and therefore are not directly applicable to sieve M-estimation.
In high-dimensional settings, empirical process methods have been investigated mainly for M-estimators for i.i.d. data \citep{welsh1989m,he2000parameters,Wang2012Ultra-high-quantreg,sun2020adaptive,cattaneo2025uniform}. 
To date, however, a unified empirical process framework that simultaneously accommodates temporal dependence and diverging dimensionality has not been established for sieve M-estimators.

The rest of the paper is organized as follows. Section \ref{sec:Preliminaries} introduces the sieve M-estimator and the nonlinear time series model. Section \ref{sec:simultaneous inference} develops a simultaneous inference framework for the sieve M-estimator, establishing key theoretical results including the extended empirical process theory, uniform Bahadur representation, and Gaussian approximation tailored for dependent data and high dimensionality. Based on these theoretical results, Section \ref{sec:SCR construction} discusses the construction of asymptotically correct SCRs, where the self-convolved bootstrap algorithm and tuning parameter selection are discussed. The theoretical validation of the SCR construction, including the scale of the critical value and asymptotic correctness, is presented in Section \ref{sec:theorems for asymptotically correct SCRs}. Section \ref{sec:numerical results} reports simulation studies and real data analysis. All technical assumptions are deferred to Section \ref{sec:def and assump} for clarity.

\section{Preliminaries}
\label{sec:Preliminaries}
We first introduce some notation. For a vector $\mathbf{v}=:\left(v_1, v_2,\dots, v_p\right) \in \mathbb{R}^p$, let $|\mathbf{v}|=\left(\sum_{i=1}^p v_i^2\right)^{1 / 2}$. For a random vector $\mathbf{V}$ and probability measure $\mr P$, denote $\|\mathbf{V}\|_{\mr P, q}=:\left[\mb{E}_{\mr P}\left(|\mathbf{V}|^q\right)\right]^{1 / q}$, $q>0$ where $\mb E_{\mr P}(\cdot)$ is the expectation with respect to probability $\mr P$. For simplicity, we shall use $\mb E(\cdot)$, $\|\cdot\|_q$, $\|\cdot\|$ instead of $\mb E_{\mr P}(\cdot)$, $\|\cdot\|_{\mr P,q}$, $\|\cdot\|_{\mr P,2}$, respectively if no confusion arises. For a matrix $\mf A$, its determinant is denoted as $\operatorname{det}(\mf A)$. If the matrix $\mf A$ is real and symmetric, we use $\lambda_{min}(\mf A)$ to denote the smallest eigenvalue of $\mf A$. For any two positive real sequences $a_n$ and $b_n$, write $a_n \asymp b_n$ if there exists $0<c<C<\infty$ such that $c \leq \liminf _{n \rightarrow \infty} a_n / b_n \leq$ $\limsup _{n \rightarrow \infty} a_n / b_n \leq C$. We write $a_n \lesssim$ $b_n\left(a_n \gtrsim b_n\right)$ to mean that there exists a universal constant $C>0$ such that $a_n \leq C b_n\left(C a_n \geq\right.$ $\left.b_n\right)$.


To obtain the nonparametric estimator in \eqref{eq:loss function}, we adopt the sieve M-regression approach. For the collection of $K=K_n\rightarrow\infty$ basis functions $\mf b(\mf x)=(b_1(\mf x),b_2(\mf x),\dots,b_K(\mf x))^\top$, we propose to approximate the unknown function $Q_0(\mf x)$ based on a linear combination of basis functions, 
\begin{equation}
    Q_{0,n}(\mf x)=: \theta_{0,n}^\top \mf b_\omega(\mf x),\quad \theta_{0,n}=: \operatorname{arg}\min_{\theta\in\Theta}\mb E\left[\rho(Y_i-\theta^\top\mf b_\omega(\mf X_i))\right],\label{eq:sieve basis}
\end{equation}
where $\Theta\subset \mb R^K$ is the feasible set of the optimization problem and $\omega_n(\mf x)=\mathbf{1}_{\mf x\in\mathcal{D}_n}$. The sets $\mathcal{D}_n$ are compact, convex subsets satisfying $\mathcal{D}_n\subset\mathcal{D}_{n+1}\subset\dots\subset\mathcal{X}$ with nonempty interior where $\mathcal{X}\subset \mb R^d$ denotes the support of $\mf X_i$. 
The approximation error $\sup_{\mf x}|Q_{0,n}(\mf x)-Q_{0}(\mf x)|$ is determined by factors such as $K$ and the smoothness of $Q_0$. We defer detailed discussions to Assumption \hyperref[(A5)]{(A5)} in Section \ref{sec:def and assump} and Section \ref{sec:supp-verify A5} in the supplement.

Based on the approximation \eqref{eq:sieve basis}, M-estimation in \eqref{eq:loss function} is performed as follows
\begin{equation}
    \hat{Q}(\mf x)=:\hat{\theta}^\top \mf b_\omega(\mf x),\quad \hat{\theta}=:\arg\min_{\theta\in\Theta}\sum_{i=1}^n\rho(Y_i-\theta^\top \mf b_{\omega}(\mf X_i)),\label{eq:sieve estimator}
\end{equation}
where $\hat{\theta}$ in \eqref{eq:sieve estimator} is a $K$-dimensional random vector with $K \to \infty$, introducing the high-dimensionality challenge in our subsequent simultaneous inference framework. The sets $\mathcal{D}_n$ can be considered as an expanding sequence of compact sets, which avoids estimating $Q_0(\mf x)$ at places far away from the observed data region. 

\subsection{Time series models}
Suppose the stationary time series $\{\mf X_i\},\{\varepsilon_i\}$ in \eqref{eq:basic model} can be represented as
\begin{equation}
    \mf X_i=\mf H(\mathcal{H}_i),\quad \varepsilon_i=G(\mathcal{H}_i,\mathcal{G}_i),\label{eq:causal representation}
\end{equation}
where $\mathcal{H}_i=(\dots,\zeta_{i-1},\zeta_i)$, $\mathcal{G}_i=(\dots,\eta_{i-1},\eta_i)$, and $\{\zeta_i\}_{i\in\mb Z},\{\eta_i\}_{i\in\mb Z}$ are i.i.d. random elements taking values in $\mathcal{S}_1$, $\mathcal{S}_2$, respectively. The functions $\mf H:\mathcal{S}_1^{\mb Z}\rightarrow \mb R^d$, and $G:\mathcal{S}_1^{\mb Z}\times \mathcal{S}_2^{\mb Z}\rightarrow \mb R$ are measurable functions. Representation \eqref{eq:causal representation} covers a wide range of stochastic processes and is widely adopted in time series studies \citep{wu2005nonlinear,dahlhaus2019towards}. Observe that $\varepsilon_i$ can be probabilistically dependent on $\mf X_i$ since both are generated by ${\cal H}_i$. Meanwhile, ${\cal G}_i$ represents exogenous factors which influence the error process. Let $\{\zeta_i^{\prime}\}$ be an i.i.d. copy of $\{\zeta_i\}$, denote $\mathcal{H}_{i,k}=(\mathcal{H}_{i-k-1},\zeta_{i-k}^\prime,\zeta_{i-k+1},\dots,\zeta_i)$. Similarly, denote $\mathcal{G}_{i,k}=(\mathcal{G}_{i-k-1},\eta_{i-k}^\prime,\eta_{i-k+1},\dots,\eta_i)$ where $\{\eta_i^{\prime}\}$ is an i.i.d. copy of $\{\eta_i\}$. 
For $q>0$, we define dependence measures of $\{\mf X_i\}$ and $\{\varepsilon_i\}$ as \citep{wu2005nonlinear} :
\begin{equation}
    \delta_{\mf H}(k,q) =: \|\mf H(\mathcal{H}_i)-\mf H(\mathcal{H}_{i,k})\|_q,\label{eq:phd of x}
\end{equation}
\begin{equation}
    \delta_{G}(k,q) =: \| G(\mathcal{H}_i,\mathcal{G}_i)- G(\mathcal{H}_{i,k},\mathcal{G}_{i,k})\|_q.\label{eq:phd of e}
\end{equation}
The $ \delta_{\mf H}(k,q)$ and $\delta_{G}(k,q)$ are invariant with respect to $i$ in \eqref{eq:phd of x} and \eqref{eq:phd of e} since $\mf H(\mathcal{H}_i)$ and $G(\mathcal{H}_i,\mathcal{G}_i)$ are both stationary. To help understand the generality of our representation in \eqref{eq:causal representation} and the dependence measure, we list commonly used stationary time series models in Section \ref{sec:example of time series models} of the supplementary material. All technical assumptions related to the time series \eqref{eq:causal representation} and sieve method \eqref{eq:sieve basis} are deferred to Section \ref{sec:def and assump} for clarity.

\section{Investigating the Simultaneous Behavior of the Sieve M-estimators}
\label{sec:simultaneous inference}
In this section, we establish the theoretical foundation for simultaneous inference 
in sieve M-estimation for nonlinear time series. The main challenge arises from the joint presence of temporal dependence and a diverging sieve dimension, which renders classical empirical process tools for i.i.d. or fixed-dimensional settings inapplicable.

To address this difficulty, our analysis proceeds in three steps.
First, we establish a high-dimensional empirical process bound for dependent data, which provides uniform stochastic control of the empirical criterion in local neighborhoods of the true parameter.
Second, leveraging this control and the convexity of the loss function, we derive a uniform Bahadur representation for the sieve M-estimator with an explicit remainder rate depending on the sieve dimension.
Third, we develop a convex Gaussian approximation for the resulting high-dimensional statistic, which enables valid simultaneous inference.

\subsection{Empirical process framework under dependence}
We first establish an empirical process framework suitable for temporally dependent data in sieve M-estimation where high dimensionality is involved. Define
\begin{equation}
    L_i(\beta)=:\rho\left( \bar\varepsilon_i - \beta^\top \mf b_\omega(\mf X_i) \right) - \rho(\bar \varepsilon_i),\quad \bar{\varepsilon}_i=:Y_i-Q_{0,n}(\mf X_i), \label{eq:loss L}
\end{equation}
where $\beta=\theta-\theta_{0,n}$ and $\theta_{0,n}$ is the population sieve target in \eqref{eq:sieve basis}. Then \eqref{eq:sieve estimator} implies that the estimator $\hat{\beta}=\hat{\theta}-\theta_{0,n}$ satisfies
\begin{equation}
    \hat{\beta}=\arg\min_{\beta\in\Theta^\Delta}\sum_{i=1}^n\mb P_n(\beta)\label{eq:hat_beta},\quad \mb P_n(\beta)=:\frac{1}{n}\sum_{i=1}^n L_i(\beta).
\end{equation}
where $\Theta^\Delta=\{\theta-\theta_{0,n}:\theta\in\Theta\}$ and $\Theta^\Delta$ is bounded by Assumption \hyperref[(A3)]{(A3)}. In other words, $\mb P_n(\hat{\beta})=\min_{\beta\in\Theta^\Delta} \mb P_n(\beta)$ and $\mb P_n(\hat{\beta})\leq 0$ since $\mb P_n(0)=0$. Suppose $(\bar\varepsilon_i^\prime,\mf X_i^\prime)$ is a copy of $(\bar\varepsilon_i,\mf X_i)$, we further define
\begin{equation}
    \mb P(\beta)=:\mb E\left\{\frac{1}{n}\sum_{i=1}^n \left[\rho(\bar\varepsilon_i^\prime-\beta^\top \mf b_\omega(\mf X_i^\prime))-\rho(\bar\varepsilon_i^\prime)\right] \right\}.\label{eq:def P(beta)}
\end{equation}
$\mb P_n(\beta) - \mb P(\beta)$ is the key quantity of our interest.
Unlike the i.i.d.\ setting in classical empirical process theory,
the sequence $\{L_i(\beta)\}_{i=1}^n$ defined in \eqref{eq:loss L} exhibits both dependence
and high-dimensional complexity due to the growing sieve basis ($\beta\in\mb R^K,K\to \infty$).
To control such dependence, we introduce a specifically designed
M-decomposition technique that allows us to obtain a maximal inequality for $\mb P_n(\beta) - \mb P(\beta)$. In the present setting, the empirical process is indexed by a sieve parameter whose dimension increases with the sample size under temporal dependence, which prevents a direct application of standard empirical process arguments and motivates the development of new probabilistic tools.
\begin{theorem}
    \label{thm:maximal inequality}
    If Assumptions \hyperref[(A1)]{(A1)}-\hyperref[(A4)]{(A4)} and \hyperref[(B1)]{(B1)}-\hyperref[(B3)]{(B3)} in Section \ref{sec:def and assump} hold, for positive real-valued sequence $\{r_n\}$, $r_n\gtrsim n^{-1/2}$ and  $r_n\xi_{K,n}=o(1)$, there exists a universal constant $C>0$ such that 
        \begin{equation}
        \mb E\sup_{|\beta|\leq r_n} | \mb P_n(\beta)-\mb P(\beta) |\leq Cr_n (\xi_{K,n}\vee \sqrt{K}) \sqrt{\frac{\log n}{n}}.    \label{eq:maximal ineq for Pn-P}
        \end{equation}
\end{theorem}
The $\xi_{K,n}$ is specified in Assumption \hyperref[(A1)]{(A1)} and typical $\xi_{K,n}\lesssim \sqrt{K}$ for most commonly used sieve bases.
Theorem \ref{thm:maximal inequality} forms the foundation of our empirical process framework, providing uniform stochastic control of $\mb P_n(\beta)-\mb P(\beta)$ over local neighborhoods of the true parameter $\beta=0$ ($\theta=\theta_{0,n}$). This result extends Pollard’s classical empirical-process approach for M-estimation \citep{pollard1990empirical,van1996weak} from the i.i.d. setting to dependent data, accommodating the sieve structure where the basis dimension $K$ grows with the sample size. The obtained uniform bound implies the stochastic equicontinuity condition required for local linearization based on convexity and directional differentiability, ensuring that the empirical criterion admits a valid expansion around the true parameter and that its minimizer can be expressed as an asymptotically linear functional of the empirical process. Consequently, Theorem \ref{thm:maximal inequality} serves as the key probabilistic ingredient for deriving the Bahadur representation of the sieve M-estimator under dependence and high dimensionality. When 
$K$ is fixed, our inequality reduces to $O(r_nn^{-1/2}\sqrt{\log n}) $, coinciding with the classical maximal inequality in \citet{phandoidaen2022empirical} for the finite-dimensional parameter case with dependent data.

\subsection{Bahadur representation}
In this subsection we develop a new Bahadur representation whose remainder rate depends jointly on the sample size $n$ and sieve dimension $K$. 
Classical results on Bahadur representations for M-estimators typically assume fixed dimension or independence; see, for example, \cite{wu2007m}, \cite{kong2010uniform}, and \cite{wu2018gradient}. 
In contrast, our setting allows $K\to\infty$ and dependent observations. To proceed from probabilistic control to inference, we express the sieve M-estimator through a uniform Bahadur representation with an explicit asymptotic linear form and a controlled remainder.
\begin{proposition}
\label{prop:Bahadur representation} Write $\bar{Q}_n=: \mb E\left( \bar \Xi_n^{(1)}(0|\mf X_i)\mf b_\omega(\mf X_i)\mf b_\omega(\mf X_i)^\top \right)$ where $\overline\Xi_n^{(1)}(0|\mf X_i)$ is defined in \eqref{eq:bar Xi}. 
If Assumptions \hyperref[(A1)]{(A1)}-\hyperref[(A4)]{(A4)} and \hyperref[(B1)]{(B1)}-\hyperref[(B3)]{(B3)} hold, suppose $\xi_{K,n} K^{\eta/4}n^{-\eta/4}\log^2 n=o(1)$, then 
        \begin{equation} 
        \label{eq:bahadur beta}
        \left|\hat\beta -\frac{\bar Q_n^{-1}\mf Z_n}{\sqrt{n}} \right|= O_p\left( (K\vee \xi_{K,n}^2)^{(\eta+2)/4}n^{-(\eta+2)/4}\log^{\eta+2} n\right),
         \end{equation}
    where
    \begin{equation}
    \mf Z_n=:\frac{1}{\sqrt{n}}\sum_{i=1}^n \mf z_i,\quad \mf z_i=:\psi(\bar \varepsilon_i)\mf b_\omega(\mf X_i).\label{eq:Z_n}
\end{equation}
\end{proposition}
The constant $\eta$ in \eqref{eq:bahadur beta} is defined through Assumption \hyperref[(B3)]{(B3)} and depends on the choice of loss function $\rho(\cdot)$. 
Specifically, $\eta=1$ for the quantile loss, while $\eta=2$ for the Huber, least-squares,
and expectile losses.
For the $L_q$ loss with $1<q\le 2$, it holds that $\eta\le 2$ when $3/2<q\le 2$ and
$\eta=2q-1$ when $1<q\le 3/2$; the derivations are deferred to
Section~\ref{sec:def and assump}.
Proposition \ref{prop:Bahadur representation} shows that $\eta$ governs the order of the
remainder term in the Bahadur representation. To illustrate the implied dimensional requirements, consider the typical $\xi_{K,n}\lesssim \sqrt{K}$. For quantile regression, under standard regularity conditions ensuring $\eta=1$,
the remainder in \eqref{eq:bahadur beta} is of order
$K^{3/4} n^{-3/4}\log^3 n$. Ignoring logarithmic factors, this coincides with the best known rate for sieve M-estimators under independent data, as established in \cite{cattaneo2025uniform}.
For the Huber, $L_2$, and expectile regressions with typical $\xi_{K,n}\lesssim \sqrt{K}$, we have $\eta=2$ by certain regular conditions (see \cite{wu2018gradient}). Consequently, \eqref{eq:bahadur beta} achieves an order of $Kn^{-1}\log^4 n$,  which implies $K$ cannot increase faster than $n^{1/2}/\log^4 n$ when \eqref{eq:bahadur beta} is simplified as $o_p(n^{-1/2})$.


\subsection{Gaussian approximation}
\label{sec:convex Gaussian}
Based on the Bahadur representation in Proposition \ref{prop:Bahadur representation}, we shall establish a convex Gaussian approximation theory for approximating the distribution behavior of $\hat \beta$ over Euclidean convex sets in $\mb R^K$. 
The convex Gaussian approximation theory developed in this subsection is motivated by the geometry
of the simultaneous confidence regions (SCRs) in Section~\ref{sec:SCR construction}.
In particular, the acceptance region of the SCR can be written as a set of the form
$\hat\beta \in A(C)$, where $A(C)$ is specified through uniform linear inequalities
$|\beta^\top \mf b_\omega(\mf x)| \le C n^{-1/2} h(\mf x)$ over $\mf x \in \mathcal D_n$;
the precise definition of $A(C)$ is given in \eqref{eq: set A(C)}.
Such a set is an intersection of uncountably many slabs in $\mb R^K$ and hence convex.
As a result, controlling coverage probabilities for the SCR naturally reduces to
approximating probabilities of convex events for the distribution of $\hat\beta$.

Introduce the following convex variational distance to measure the distributional difference between $K$-dimensional random vectors $\mf X$ and $\mf Y$,
\begin{equation}
    \mathcal{K}(\mf X,\mf Y)=:\sup_{A\in\mathcal{A}}|\mr P(\mf X\in A)-\mr P(\mf Y\in A)|,\label{eq:K-dist def}
\end{equation}
where  $\mathcal{A}$ is the collection of all the convex sets in $\mb R^K$. The following Theorem \ref{thm:Gaussian approximation} shows that there exists a $K$-dimensional Gaussian random vector $\mf G_{n}$ which preserves the same covariance structure of $\mf Z_n$, such that any probability of $\hat{\beta}$ over a convex set can be approximated by that of the Gaussian random vector $\mf G_n$. 
\begin{theorem}
            \label{thm:Gaussian approximation} 
        Under Assumptions \hyperref[(A1)]{(A1)}-\hyperref[(A4)]{(A4)} and \hyperref[(B1)]{(B1)}-\hyperref[(B3)]{(B3)}, suppose $\Sigma_{\mf Z}$ is the covariance matrix of $\mf Z_n$ in \eqref{eq:Z_n} and the smallest eigenvalue of covariance matrix $\Sigma_{\mf Z}$ is bounded away from $0$ and for some $q>4$, $\|\mf z_i\|_q=O(\xi_{K,n})$. Then we have
        \begin{itemize}
            \item[(i).] There exists a Gaussian random vector $\mf G_n\sim N(\mf 0,\Sigma_{\mf Z})$ such that
            \begin{align}
                \mathcal{K}\left(\mf Z_n,\mf G_n\right)&=O\left(\gamma_{n,q,K}\log^3 n\right),\label{eq:GA bound Z}\\
                \gamma_{n,q,K}&=:\min\left\{K^{\frac{q-1}{4(q+5)}}\xi_{K,n}^{\frac{3q}{q+5}}n^{\frac{4-q}{2(q+5)}}, K^{\frac{1}{2}+\frac{4}{3q}(\frac{1}{q}-1)}\xi^{\frac{2}{3}}_{K,n}n^{-\frac{1}{3}+\frac{2}{q}(1-\frac{1}{3q})}\right\}.\notag
            \end{align}
            \item[(ii).] 
           Under the event 
           $$
           \mathcal{B}_n^{\epsilon}=\left\{
        \left|\hat\beta -\frac{\bar Q_n^{-1}\mf Z_n}{\sqrt{n}} \right|\leq g_nK^{(\eta+2)/4}n^{-(\eta+2)/4}\log^{\eta+2} n\right\},
           $$
           where $g_n$ diverges to infinity at an arbitrarily slow rate, we further have
    \begin{equation}
        \sup_{A\in\mathcal{A}} \left| \mr P(\sqrt{n}\hat{\beta}\in A| \mathcal{B}_n^\epsilon )-\mr P\left(\bar{Q}_n^{-1}\mf G_n\in A\right)\right|=O\left(\left(K^{\frac{\eta+3}{6}}n^{-\frac{\eta}{6}}+\gamma_{n,q,K}\right)\log^3 n\right).\label{eq:GA bound}
    \end{equation}
        \end{itemize}
\end{theorem}
The condition on the smallest eigenvalue of $\Sigma_{\mf Z}$ ensures $\sum_{i=1}^n \psi(\bar\varepsilon_i)\mf b_\omega(\mf X_i)/\sqrt{n}$ is not degenerate, which is a mild requirement. $\|\mf z_i\|_q=O(\xi_{K,n})$ can be easily checked for most regressions and sieve basis families. For the Huber and quantile regressions, $\|\mf z_i\|_q=O(\xi_{K,n})$ for any given $q>0$ since $\psi(x)$ is bounded. For expectile and robust $L_q$ regressions, $\|\mf z_i\|_q=O(\xi_{K,n})$ is reduced to the moment condition $\|\psi(\varepsilon_i)\|_q<\infty$. 

Consider typical $\xi_{K,n}\lesssim \sqrt{K}$ for tensor-products of univariate polynomial spline, trigonometric polynomial, or wavelet bases, the probability $\mr P(\mathcal{B}_n^\epsilon)=1-o(1)$ by Proposition \ref{prop:Bahadur representation}, and the rate in \eqref{eq:GA bound} can 
yield valid convex Gaussian approximation towards $\hat{\beta}$ when $K=O(n^{\min\{2/5,\eta/(\eta+3)\}-\alpha})$ for some small $\alpha>0$ when $q$ is sufficiently large. More specifically, among quantile, expectile, least squares, the Huber, and robust $L_q$ regressions, the upper bound of $K$ for practically useful convex Gaussian approximation \eqref{eq:GA bound} is in the range from $n^{1/4-\alpha}$ to $n^{2/5-\alpha}$ with some sufficiently small $\alpha>0$.


\section{SCR construction}
\label{sec:SCR construction}
The high dimensional empirical process framework, together with the convex Gaussian approximation scheme can handle many simultaneous inference problems in time series sieve M-estimation. A representative example is the construction of simultaneous confidence regions (SCRs) for $Q_{0,n}(\mf x)=\theta_{0,n}^\top \mf b_\omega(\mf x)$. 
Construction of an asymptotically correct SCR at significance level $1-\alpha$ for $Q_{0,n}(\mf x)$ can be formulated as finding the critical value $C_{\alpha,h}$ (which might depend on $n$) such that the probability 
\begin{equation}
    \label{eq:SCR on Q}
    \lim_{n\rightarrow\infty}\mr P\left(\hat{Q}(\mf x)- \frac{C_{\alpha,h} }{\sqrt{n}} h(\mf x)\leq Q_{0,n}(\mf x)\leq \hat{Q}(\mf x)+ \frac{C_{\alpha,h} }{\sqrt{n}} h(\mf x),\forall \mf x\in\mathcal{D}_n \right)= 1-\alpha,
\end{equation}
where $h(\mf x)$ serves as a localized scaling factor. 
It is worth emphasizing that our convex Gaussian approximation framework 
accommodates a broad choice of positive scaling functions $h(\mf x)$ 
to construct valid SCRs. 
For example, the convenient choice $h(\mf x)=1$ yields a constant-width 
SCR that directly controls the unscaled maximal deviation 
over the entire predictor space. Alternatively, the standardized choice 
$h(\mf x)=\sqrt{\mf b_\omega(\mf x)^\top\bar{Q}_n^{-1}\Sigma_{\mf Z}\bar{Q}_n^{-1}\mf b_\omega(\mf x)}$ 
is proportional to the standard deviation of $\hat Q(\mf x)$ and provides a direct bridge between global and pointwise inference, showing that the width of our SCR differs from a pointwise confidence interval only by the scale of the global critical value $C_{\alpha,h}$. As we demonstrate in Proposition \ref{prop:critical value}, 
the cost of expanding from pointwise to simultaneous coverage 
is merely an additional logarithmic factor $C_{\alpha,h}\asymp \sqrt{\log n}$.

Define the set $A(C)$ with critical value $C$ as
 \begin{equation}
    A(C)=:\left\{ \mf S\in\mathbb{R}^{K}: \left|\mf S^\top \mf b_\omega(\mf x)\right|\leq \frac{C}{\sqrt{n}} h(\mf x),\forall \mf x\in\mathcal{D}_n \right\},\quad C\in \mb R,\label{eq: set A(C)}
\end{equation}
then by \eqref{eq:sieve estimator}, \eqref{eq:SCR on Q} can be rewritten as $\lim_{n\rightarrow\infty}\mr P\left(\hat{\beta}\in A(C_{\alpha,h})\right)= 1-\alpha.$
Recall that $A(C)$ is a convex subset of $\mb R^K$ as discussed in
Section~\ref{sec:convex Gaussian}. By applying Theorem \ref{thm:Gaussian approximation}, we can find $C_{\alpha,h}$ in \eqref{eq:SCR on Q} from the $1-\alpha$ quantile of a rescaled inner product between a standard Gaussian random vector and the sieve basis $\mf b_\omega(\mf x)$, as displayed in the following Corollary.
\begin{corollary}
    \label{cor:SCR Gaussian approx}
    Under the same conditions in Theorem \ref{thm:Gaussian approximation}, if $(K^{\frac{\eta+3}{6}}n^{-\frac{\eta}{6}}+\gamma_{n,q,K})\log^3 n=o(1)$ then 
            \begin{equation}
            \lim_{n\to\infty}\mr P\left(\frac{\sqrt{n}|\hat{Q}(\mf x)-Q_{0,n}(\mf x)|}{h(\mf x)}\leq C_{\alpha,h} ,\forall \mf x\in\mathcal{D}_n\right)=1-\alpha,\label{eq:SCR approximation}
        \end{equation}
        where the critical value $C_{\alpha,h}$ satisfies
    \begin{equation}
        \mr P\left( \frac{\left|\mf b_\omega(\mf x )^\top\bar Q_n^{-1}\mf G_n\right|}{h(\mf x)} \leq C_{\alpha,h}  ,\forall \mf x\in\mathcal{D}_n\right)=1-\alpha,\label{eq:target critical value}
    \end{equation}
    for a Gaussian random vector $\mf G_n\sim N_K(0,\Sigma_{\mf Z})$.
\end{corollary}
Corollary \ref{cor:SCR Gaussian approx} indicates that we can well approximate the simultaneous distributional behavior of $\sqrt{n}(\hat{Q}(\mf x)- Q_{0,n}(\mf x))$ on ${\cal D}_n$ by that of the Gaussian process $\mf b_\omega(\mf x)^\top\bar{Q}_n^{-1} \mf G_n$. 
Critical values of the latter Gaussian process can be evaluated via bootstrap which will be discussed in detail in the next subsection. Hence Corollary \ref{cor:SCR Gaussian approx} serves as a theoretical foundation for the construction of the SCR. 

\subsection{Self-convolved bootstrap}
\label{sec:bootstrap}
Identifying the target critical value $C_{\alpha,h}$ in \eqref{eq:target critical value} for the asymptotically correct SCR is challenging. On the one hand, the approximation convergence rate of $C_{\alpha,h}$ is very slow \citep{Sun&Loader}. On the other hand, $\bar Q_n^{-1}\mf G_n$ involves unknown term $\bar{Q}_n^{-1}$. The $\bar{Q}_n^{-1}$ is complex and hard to estimate because it is determined by the expectation of $\bar\Xi_n^{(1)}(0|\mf X_i)\mf b_\omega(\mf X_i)\mf b_\omega(\mf X_i)^\top$.

To overcome these difficulties in a fully data-driven way and exploit the theoretical results developed in Sections \ref{sec:simultaneous inference}–\ref{sec:SCR construction}, we propose a self-convolved bootstrap method. This method leverages the uniform Bahadur representation and the Gaussian approximation of the linearized sieve estimator to mimic the distribution of $\hat{\beta}$ via block sums of the estimated coefficients. Consequently, it provides a feasible and asymptotically correct SCR without requiring the direct estimation of $\bar Q_n^{-1}$ or slow analytic approximations of $C_{\alpha,h}$.


\begin{breakablealgorithm}
	\caption{Self-convolved bootstrap for SCR}
	\label{alg:self-convolved bootstrap}
    \begin{algorithmic}[0]
    \setlength{\baselineskip}{0.7\baselineskip}
        \STATE{\textbf{Input:}} Data $\{(\mf X_i,Y_i)\}_{i=1}^n$, sieve basis $\mf b_\omega(\mf x)$, and associated $\hat{\theta}$ in \eqref{eq:sieve estimator}.
        \STATE{\textbf{Step 1:}} Select window size $M$ satisfying $M\rightarrow\infty$ but $M=o(n)$. (The selection strategy is introduced in Section \ref{sec:tuning parameter selection}.)
        \STATE{\textbf{Step 2:}} Partition the data into $n-M+1$ blocks of length $M$ and compute the estimated linear coefficient $\hat{\theta}(k)$ using each block of data $\left\{\left(\mf X_i,Y_i\right)\right\}_{i=k}^{k+M-1}$, i.e. 
\begin{equation}
\hat{\theta}(k)=: \arg \min _{\theta \in \Theta} \frac{1}{M} \sum_{i=k}^{k+M-1} \rho\left(Y_i- \theta^\top\mf b_\omega\left(\mf X_i\right)\right),\quad k=1,2,\dots,n-M+1.\label{eq:block theta}    
\end{equation}
        \STATE{\textbf{Step 3:}} Generate i.i.d. standard Gaussian random variables $\{ V_j\}_{j=1}^{n-M+1}$ and compute
        \begin{equation}
    \Phi_M =:\sqrt{\frac{M}{n-M+1}}\sum_{j=1}^{n-M+1} \left(  \hat{\theta}(j)-\hat{\theta}   \right)V_j,\label{eq:bootstrap sample}
\end{equation}
        \STATE{\textbf{Step 4:}} Repeat Step 3 for $B$ times and document the outcomes
        $\Phi_M^{(b)}$, $b=1,\dots,B$.
        \STATE{\textbf{Step 5:}} For a given level $\alpha\in(0,1)$ and a chosen positive function $h(\mf x)$ on $\mathcal{D}_n$, denote $\hat{C}_{\alpha,h}$ as the $(1-\alpha)$-th sample quantile of 
        $$
        \left\{ \sup_{\mf x\in\mathcal{D}_n}\frac{|\mf b_\omega(\mf x)^\top\Phi_M^{(b)}|}{h(\mf x)}\right\}_{b=1}^B
        $$
	    \STATE{\textbf{Output:}} $1-\alpha$-level SCR 
        \begin{equation}
            \hat{Q}(\mf x)\pm \hat{C}_{\alpha,h}\frac{h(\mf x)}{\sqrt{n}}.\label{eq:SCR boot}
        \end{equation}  
 \end{algorithmic}
\end{breakablealgorithm}

The implementation of Algorithm \ref{alg:self-convolved bootstrap} only requires the choice of one tuning parameter $M$ and the convolution of block sieve M-estimators and i.i.d. standard Gaussian random variables. The intuition behind Algorithm \ref{alg:self-convolved bootstrap} is that, under stationarity, each block estimator behaves approximately like a replicate of the full-sample estimator. Building on this insight, the Bahadur representation implies that the centered block estimates $\hat{\theta}(j)-\hat{\theta}$ naturally encode both the intricate matrix $\bar{Q}_n^{-1}$ and the local temporal dependence of the sequence $\{\mf z_i\}$. By further convolving these data-driven block proxies with Gaussian variables $V_j$, we obtain the simulated $\Phi_M^{(b)}$ that asymptotically reproduces the covariance structure of the target Gaussian vector $\bar{Q}_n^{-1}\mf G_n$ given by the convex Gaussian approximation.

The tuning parameter $M$ can be automatically selected using the Flat-Top Lag-Windows strategy described in Section \ref{sec:tuning parameter selection}. Theoretical validation in Theorem \ref{thm:bootstrap} confirms that Algorithm \ref{alg:self-convolved bootstrap} produces asymptotically correct SCRs.


\subsection{Tuning parameter selection}
\label{sec:tuning parameter selection}
To implement our sieve estimation and obtain its simultaneous inference, here we discuss how to choose the appropriate sieve number $K$ for the M-estimation and block size $M$ for the bootstrap.

To select the sieve dimension \(K\), we choose a validation length \(\ell_n\), taken as 
\(\ell_n=\lfloor 3\log_2 n\rfloor\) in our implementation, and divide the time series into two parts: the training part \(\{(\mf X_i,Y_i)\}_{i=1}^{n-\ell_n}\) and the validation part 
\(\{(\mf X_i,Y_i)\}_{i=n-\ell_n+1}^{n}\). For a sequence of choices $K^{(i)},i=1,2,\dots,L$, we estimate $\hat{Q}^{(i)}(\mf x)$ by sieve number $K^{(i)}$ on training data. Then using the fitted model, we forecast the time series in the validation part of the time series. Let $\widehat{Y}_{n-l+1}^{(i)}, \cdots, \widehat{Y}_{n}^{(i)}$ be the forecast of $Y_{n-l+1}, \ldots, Y_n$, respectively using the estimator $\hat{Q}^{(i)}(\mf x)$. Then we choose $K^{(i)}$ with the minimum empirical risk of the forecast, i.e.
$$
K^*= \arg\min_{K^{(i)}:1\leq i\leq L} \frac{1}{l}\sum_{j=n-l+1}^n\rho(Y_j-\hat{Y}_j^{(i)}).
$$
For the product tensor sieve basis, the method can be specified in the choice of sieve number for each tensor. Consider the tensor product linear sieve space $\mathcal{B}_n$, which is constructed as a tensor product space of some commonly used univariate linear approximating spaces $\mathcal{B}_{n1},\dots, \mathcal{B}_{nd}$. Suppose the  orthonormal bases of $\mathcal{B}_{nj}$ is $\mf b^{(j)}(x)=:(b_{1}^{(j)}(x),\dots,b_{k_j}^{(j)}(x))^\top$ where the $k_j$ is the dimension of $\mathcal{B}_{nj}$ and $\prod_{j=1}^d k_j=K$. The sieve basis $\mf b(\mf x)=(b_1(\mf x),\dots,b_K(\mf x))$ for $\mf x=(x_1,\dots,x_d)^\top$ is aligned as the product of spaces $\mathcal{B}_{n1},\dots, \mathcal{B}_{nd}$, i.e. 
\begin{equation}
    \mf b(\mf x)=: \mf b^{(1)}(x_1)\otimes \dots \otimes \mf b^{(d)}(x_d),\label{eq:tensor product}
\end{equation}
where `$\otimes$' is the Kronecker product and each $b_j(\mf x)$ can be expressed as the product $b_j(\mf x)=\prod_{i=1}^{d}b_{j_i}^{(i)}(x_i)$ with $1\leq j_i\leq k_j$. In this circumstance, we can choose a sieve number vector $\mf k = (k_1,\dots,k_d)^\top$. For a sequence of candidate vector $\mf k^{(i)}=(k_1^{(i)},\dots,k_d^{(i)}),i=1,\dots,L$, we estimate $\hat{Q}^{(i)}(x)$ by the product of sieve spaces $\mathcal{B}_{n1}^{(i)},\dots,\mathcal{B}_{nd}^{(i)}$ where $\operatorname{dim}(\mathcal{B}_{nj}^{(i)})=k_j^{(i)}$. Similarly, we obtain $\hat{Y}_{j}^{(i)}=\hat{Q}^{(i)}(\mf X_j)$ and choose the vector $\mf k^{*}$ such that
$\mf k^*= \arg\min_{\mf k^{(i)}:1\leq i\leq L} \frac{1}{l}\sum_{j=n-l+1}^n\rho(Y_j-\hat{Y}_j^{(i)}).$

Secondly, to choose the block size $M$ for our self-convolved bootstrap, we introduce the block selection via Flat-Top Lag-Windows in \cite{politis2004automatic}. Denote residuals $\hat{\varepsilon}_i=Y_i-\hat{Q}(\mf X_i)$. For each $j=1,\dots,K$, we first estimate 
\begin{equation}
        \label{eq:correlogram} \hat{R}_j(k)=n^{-1}\sum_{i=1}^{n-|k|}\left(\hat{z}_{i,j} -\sum_{i=1}^n\hat{z}_{i,j}/n \right)\left(\hat{z}_{i+|k|,j}-\sum_{i=1}^n\hat{z}_{i,j}/n\right),
\end{equation}
where $\hat{z}_{i,j}= \psi(\hat \varepsilon_i) b_j(\mf X_i)\omega(\mf X_i)$. Then we apply the following steps for each $j=1,\dots,K$ and finally choose $\hat{M}=\sum_{j=1}^{K}\hat{M}_j/K$ for our moving block bootstrap.
\begin{breakablealgorithm}
	\caption{Block selection via Flat-Top Lag-Windows}
	\label{alg:block selection}
    \begin{algorithmic}[0]
    \setlength{\baselineskip}{0.7\baselineskip}
        \STATE{\textbf{Input:}} $\hat{R}_j(k)$ defined in \eqref{eq:correlogram}.
        \STATE{\textbf{Step 1:}} Identify the smallest integer, say $\bar{m}_j$, after which the correlogram appears negligible, i.e., $|\widehat{R}_j(k)|\leq \epsilon$ for some sufficiently small $\epsilon>0$ when $k>\bar{m}_j$.
        \STATE{\textbf{Step 2:}} Using the value $\bar{M}=2 \bar{m}_j$, obtain following estimator
    \begin{equation*}
        \hat{G}_j=\sum_{k=-\bar{M}}^{\bar{M}} \lambda(k/\bar{M})|k|\hat{R}_j(k),\quad\hat{D}_j=\frac{4}{3}\left(\sum_{k=-\bar{M}}^{\bar{M}} \lambda(k/\bar{M})\hat{R}_j(k)\right)^2,
    \end{equation*}
    where $\lambda(t)=\mathbf{1}_{0<|t|\leq 1/2}+2(1-|t|)\mathbf{1}_{1/2<|t|\leq 1}$.
        \STATE{\textbf{Output:}} The optimal block size $\hat{M}_j=(2\hat{G}_j^2/\hat{D}_j)^{1/3}n^{1/3}$.
	\end{algorithmic}
\end{breakablealgorithm}
In practice, Step 1 in Algorithm \ref{alg:block selection} is intuitive and feasible by visually inspecting the correlogram plot $\hat{R}_j(k)$ vs. $k$.  However, in the multi-dimensional case, drawing and inspecting correlogram plots for each $j=1,\dots,K$ becomes increasingly complex and inefficient. To determine the appropriate $\bar{m}_j$ automatically, we introduce a precise formulation given in \cite{politis2001taper}. Let $\rho_j(k)=R_j(k) / R_j(0), \hat{\rho}_j(k)=\widehat{R}_j(k) / \widehat{R}_j(0)$, and choose the smallest positive integer $\bar{m}_j$ such that $|\hat{\rho}_j(\bar{m}_j+k)|<c \sqrt{\log n / n}$, for $k=1, \ldots, L_n$, where $c>0$ is a fixed constant, and $L_n$ is a positive, nondecreasing integer-valued function of $n$ such that $L_n=$ $o(\log n)$. As discussed in \cite{politis2001taper} and \cite{politis2004automatic}, $c=2$ and $L_n=\max (5, \sqrt{\log n})$ are recommended. Equipped with the procedure introduced in \cite{politis2001taper}, the block selection algorithm is fully automatic and can be implemented easily.

\section{Theoretical Justification of the SCRs}
\label{sec:theorems for asymptotically correct SCRs}
This section establishes the theoretical validity of our proposed SCR for the sieve M-estimator in Section \ref{sec:SCR construction}. 
Building upon the Gaussian approximation in Theorem~\ref{thm:Gaussian approximation}, we first derive the scale of $C_{\alpha,h}$, which determines the scale of the SCR. 
We then prove that the SCR constructed by the self-convolved bootstrap algorithm is asymptotically correct under general dependence and growing sieve dimension. 

\subsection{Asymptotic critical value}
The maximum deviation of the Gaussian process $|\mf b_\omega(\mf x)^\top\bar{Q}_n^{-1}\mf G_n|/h(\mf x)$ is characterized by the supremum of functional linear combinations of Gaussian variables. 
The volume-of-tube formula (c.f. \cite{Sun&Loader}) enables us to approximate the corresponding critical value $C_{\alpha,h}$ that defines the asymptotic confidence level in~\eqref{eq:SCR on Q}. 
The following Proposition specifies the asymptotic order of $C_{\alpha,h}$ under mild regularity conditions on the geometry of the sieve basis and the predictor domain $\mathcal{X}$.
\begin{proposition}
    \label{prop:critical value}  Under the same conditions in Corollary \ref{cor:SCR Gaussian approx}, denote matrix $\mf M_j(\mf x)=:\frac{\partial }{\partial x_j}\left( \frac{\mf b(\mf x)}{|\mf b(\mf x)|}\right)$ and $\mf M(\mf x)=:( \mf M_1(\mf x),\dots, \mf M_d(\mf x))$. Suppose space $\mathcal{D}_n=\mathcal{X}$ is compact and there exist constants $c_0, c_1, c_2, \underline{c} \geq 0$ such that
\begin{equation}
  \sup_{\mf x\in\mathcal{X}}|\nabla\mf b(\mf x)|\lesssim n^{c_1},\sup_{\mf x\in\mathcal{X}}|\nabla^2\mf b(\mf x)|\lesssim n^{c_2},\inf_{\mf x\in\mathcal{X}}|\mf b(\mf x)|\gtrsim n^{c_0},\int_{\mathcal{X}}\lambda_{min}(\mf M(\mf x)^\top\mf M(\mf x))\mr d\mf x\gtrsim n^{\underline{c}}.\label{eq:critical value condition}
\end{equation}
Let the target critical value $C_{\alpha,h}$ according to flexible weight $h(\mf x)>0$ satisfy
    \begin{equation}
        \lim_{n\to \infty}\mr P\left\{\sup_{\mf x\in\mathcal{X}}\frac{\sqrt{n}|\hat Q(\mf x)-Q_{0,n}(\mf x) |}{h(\mf x)}\leq C_{\alpha,h} \right\}=1-\alpha,\label{eq:critical value rate}
    \end{equation}
    where $\alpha$ is the given significance level and $\alpha\in(0,1)$. Then for $\sigma_n^2(\mf x)=:\mf b(\mf x)^\top\bar{Q}_n^{-1}\Sigma_{\mf Z}\bar{Q}_n^{-1}\mf b(\mf x)$,
    $$
        \inf_{\mf x\in \mathcal{X}}\frac{\sigma_n(\mf x)}{h(\mf x)}\sqrt{\log n} \lesssim C_{\alpha,h} \lesssim\sup_{\mf x\in \mathcal{X}}\frac{\sigma_n(\mf x)}{h(\mf x)}\sqrt{\log n}.
    $$
    In particular, $C_{\alpha,h}\asymp \sqrt{\log n}$ if choosing $h(\mf x)=\sigma_n(\mf x)$.
\end{proposition}
Conditions in \eqref{eq:critical value condition} are simplified requirements on the sieve basis which will yield a polynomial rate $n^c$ ($c\geq 0$) for the geometric quantities, including volume, curvature, and boundary of the manifold $\{\mf b(\mf x)/|\mf b(\mf x)|: \mf x\in\mathcal{X}\}$. The rates $n^{c_1},n^{c_2},n^{c_0}$ in condition \eqref{eq:critical value condition} are mild assumptions which have been frequently used in the literature of sieve nonparametric estimation and inference; see Assumption 4 of \cite{CHEN2015uniformsieve} and Example 1-2 in \cite{quan2024JASA} for more details. The rate $n^{\underline{c}}$ in condition \eqref{eq:critical value condition} can be directly calculated in practice given basis $\mf b(\mf x)$, which will be verified in Section \ref{sec:examples of sieve basis} of the supplementary material with commonly used sieve bases. 
\subsection{Asymptotic validity of the bootstrap SCR}
We now show that the self-convolved bootstrap algorithm proposed in Section~\ref{sec:bootstrap} produces an asymptotically correct SCR. 
The following theorem provides a bound on the bootstrap approximation error and guarantees the asymptotic correctness of the resulting SCRs.
\begin{theorem}
    \label{thm:bootstrap}  Under the same conditions in Theorem \ref{thm:Gaussian approximation}, suppose $\xi_{K,n}\lesssim \sqrt{K}$ and $\max\{K^{\frac{\eta+3}{4}} M^{-\frac{\eta}{4}},K^{\frac{3}{4}}(M^{\frac{1}{2}}n^{-\frac{1}{2}}+M^{-1})\}=O(n^{-\kappa})$ for some $\kappa>0$.
    Then the SCR obtained in \eqref{eq:SCR boot} is asymptotically correct, i.e.,
\begin{equation}
    \lim_{n\rightarrow\infty}\lim_{B\rightarrow\infty}\mr P\left( |\hat{Q}(\mf x)-Q_{0,n}(\mf x)|\leq \hat{C}_{\alpha,h}\frac{h(\mf x)}{\sqrt{n}},\forall \mf x\in\mathcal{D}_n \right)=1-\alpha.\label{eq:valid boot strap}
\end{equation}
\end{theorem}
Theorem \ref{thm:bootstrap} guarantees the asymptotic validity of the self-convolved bootstrap under a generic scaling function $h(x)$. In practice, setting $h(x) = 1$ provides a computationally efficient SCR. Alternatively, we also suggest a data-driven scaling factor $h(x) = \sqrt{\mf b_\omega(\mf x)^\top \hat{\Sigma}_{\mf Z} \mf b_\omega(\mf x)}$ where $\hat{\Sigma}_{\mf Z}$ is estimated via block sums
$$
\hat{\Sigma}_{\mf Z} = \frac{1}{n-M+1}\sum_{j=1}^{n-M+1} \hat W(j)\hat W(j)^\top,\quad \hat W(j) = \frac{1}{\sqrt{M}}\sum_{i=j}^{j+M-1} (\hat{\mf z}_i-\bar{\mf z}),
$$
with $\hat{\mf z}_i=\psi(\hat\varepsilon_i)\mf b_\omega(\mf X_i)$, $\hat\varepsilon_i=Y_i-\hat{\theta}^\top \mf b_\omega(\mf X_i)$, and $\bar{\mf z}=\sum_{i=1}^n \hat{\mf z}_i/n$. In this way, the resulting SCR reflects local uncertainty with varying width.

\subsection{Inference for the exact target \(Q_0(\mf x)\)}
Finally, we extend the simultaneous inference from \(Q_{0,n}(\mf x)\) to the true function \(Q_0(\mf x)\). Proposition \ref{prop:undersmoothing}
establishes that the same SCR, \(\widehat Q(\mf x)\pm C_{\alpha,h}h(\mf x)/\sqrt{n}\),
which is valid for \(Q_{0,n}(\mf x)\), also yields asymptotically correct inference for \(Q_0(\mf x)\) under a suitable undersmoothing condition.
\begin{proposition}
    \label{prop:undersmoothing}
Suppose Assumption \hyperref[(A5)]{(A5)} and the same conditions in Proposition \ref{prop:critical value} hold. If
\begin{equation}
\label{eq:undersmoothing condition}
K^{-\varsigma}\,\sqrt{n\log n}\,
\frac{
\sup_{\mf x\in\mathcal D_n} h(\mf x)^{-1}
}{
\inf_{\mf x\in\mathcal D_n}\sigma_n(\mf x)/h(\mf x)
}
=o(1),
\end{equation}
where $\varsigma$ is defined in Assumption \hyperref[(A5)]{(A5)}, then the same critical value \(C_{\alpha,h}\) satisfying \eqref{eq:SCR on Q}
also ensures
\begin{equation}
\label{eq:SCR on Q0}
\lim_{n\rightarrow\infty}
\mr P\left(
\widehat Q(\mf x)- \frac{C_{\alpha,h}}{\sqrt{n}}h(\mf x)
\leq Q_{0}(\mf x)\leq
\widehat Q(\mf x)+ \frac{C_{\alpha,h}}{\sqrt{n}}h(\mf x),
\ \forall \mf x\in\mathcal{D}_n
\right)
= 1-\alpha.
\end{equation}
\end{proposition}
Consequently, in conjunction with Proposition \ref{prop:undersmoothing}, the SCR validity statements for \(Q_{0,n}\) in \eqref{eq:target critical value}, \eqref{eq:critical value rate}, and \eqref{eq:valid boot strap} carry over to the true target \(Q_0\) under condition \eqref{eq:undersmoothing condition}.
\begin{remark}
Condition \eqref{eq:undersmoothing condition} can be verified once the
smoothness of \(Q_0\) and the scaling \(h(\mf x)\) are specified.
For instance, when \(Q_0\) belongs to a \(p\)-smooth H\"older class, it can
be shown that one may take \(\varsigma=p/d\); see Section
\ref{sec:supp-verify A5} in the supplement. Together with the choice \(h(\mf x)=\sigma_n(\mf x)\) (or $h(\mf x)=1$) and typical $\xi_{K,n}\lesssim \sqrt{K}$, condition
\eqref{eq:undersmoothing condition} reduces to $K \gg (n\log n)^{d/(2p+d)}$. This indicates that
\eqref{eq:undersmoothing condition} is essentially an undersmoothing
condition, requiring the sieve dimension \(K\) to be sufficiently large so
that the deterministic approximation is negligible relative to the
simultaneous stochastic scale.
Moreover, such a lower bound on \(K\) is compatible with Theorem \ref{thm:Gaussian approximation}. For example, when \(d=1\) and \(p=2\),
one may choose \(K=n^{\bar\kappa}\) with
$1/5<\bar\kappa<\min\left\{2/5,\eta/(\eta+3)\right\}$,
which satisfies both the undersmoothing requirement and the conditions in our limiting theorems.
\end{remark}

\section{Numerical results}
\label{sec:numerical results}
In this section, we conduct extensive numerical simulations to illustrate the effectiveness of our results in moderate samples. 
Using the selected tuning parameters and the proposed theorems, we conduct simulation studies and a real-data analysis of the S\&P 500 daily returns to show the merits of our SCR for the sieve M-estimator. 

\subsection{Simulation study}
Our simulation studies are focused on nonlinear time series models \eqref{eq:basic model} with dependent error $\{\varepsilon_i\}$ and multi-dimensional $\{\mf X_i\}$ for M-estimation including quantile, $L_q$, and Huber robust estimates. The simulation setup considers model \eqref{eq:basic model} with $d=2$,
\begin{equation}
    Y_i=Q(X_{i,1},X_{i,2})+\varepsilon_i,\label{eq:simul model}
\end{equation}
where $\mf X_i=(X_{i,1},X_{i,2})^\top$ is a 2-dimensional random vector. We consider the $Q(x_1,x_2)$ in \eqref{eq:simul model} taking the following forms: $Q_1(x_1,x_2)=\sin(2\pi x_1)+x_2$, $Q_2(x_1,x_2)=(\sin(2\pi x_1) + 1)\exp(-x_2/2)$. The covariate process $\{(X_{i,1},X_{i,2})\}$ is generated by an MA model $\mf X_i=0.5 \mf U_i + 0.3 \mf U_{i-1} + 0.2 \mf U_{i-2}$ 
where $\mf U_i=(U_{i,1},U_{i,2})^\top$ and $\{\mf U_i\}$ are i.i.d. from a bivariate uniform distribution on $[0,1]^2$ with $\operatorname{cov}(U_{i,1},U_{i,2})=0.2876$. The error process $\{\varepsilon_i\}$ is generated by an AR model $\varepsilon_i=0.5 \varepsilon_{i-1}+\sigma\eta_i$ where $\{\eta_i\}$ are i.i.d. $N(0,1)$ distribution and $\sigma=1/2$. 
We consider the following loss functions for our sieve estimators: the quantile loss $\rho(x)=\tau x^{+}+(1-$ $\tau)(-x)^{+}$ with $\tau=0.02,0.5,0.98$; $L_q$ loss $\rho(x)=|x|^q$ with $q=2$; Huber loss $\rho(x)=\left(x^2 \mathbf{1}_{|x| \leq c}\right) / 2+\left(c|x|-c^2 / 2\right) \mathbf{1}_{|x|>c}$ with $c=1.5$. 

For model \eqref{eq:simul model}, our sieve estimator is based on the tensor product sieve basis introduced in \eqref{eq:tensor product} where the sieve basis $\mf b(x_1,x_2)=\mf b^{(1)}(x_1) \otimes \mf b^{(2)}(x_2).$
For simplicity, we choose $\mf b^{(1)}(x_1)$ and $\mf b^{(2)}(x_2)$ from the same type of sieve basis family, including the Fourier series, the Legendre polynomials, and the Daubechies wavelet. The detailed formulations of the sieve basis are listed in Section \ref{sec:examples of sieve basis} of the supplement materials. The dimensions of $\mf b^{(1)}(x_1)$ and $\mf b^{(2)}(x_2)$ are determined by the method in Section \ref{sec:tuning parameter selection}. All simulation results are based on 1000 simulation runs with sample size $n = 500, 800$ and bootstrap size $B = 2000$.

Table~\ref{tab:coverage} presents the simulated coverage probabilities for model \eqref{eq:simul model} with two types of target functions and five regression cases. Our simulation results conclude that the performance of joint SCRs generated by Algorithm \ref{alg:self-convolved bootstrap} is reasonably accurate, and the coverage result becomes more robust as sample size $n$ grows. We compare our sieve simultaneous inference results with the SCRs built on kernel estimators. For the multivariate model \eqref{eq:basic model} with i.i.d. observations, \cite{chao2017confidence} proposed the construction of SCRs on kernel-based estimators for quantile and expectile regressions. For the mean ($L_2$) regression with random design, \cite{chen2021multinon} also considered the construction of SCRs for the kernel method in nonstationary time series. However, the predictor in \cite{chen2021multinon} is univariate; therefore, we here take the SCRs in \cite{chao2017confidence} for quantile and $L_2$ regressions as comparison.  
To our knowledge, the SCRs for Huber's multivariate regression model based on kernel estimator have not been proposed yet. Both the bandwidth selection strategy and the associated bootstrap approach align with \cite{chao2017confidence}. 

\begin{table}[h]
\centering
  \setlength{\belowcaptionskip}{-0.5cm}
  \captionsetup{font={small}, skip=-3pt}
  \caption{Simulated coverage probabilities of SCRs.}
  \scalebox{0.8}
  {\renewcommand{\arraystretch}{0.5}
    \begin{tabular}{lrrrrrrrrr}
    \hline
          & \multicolumn{4}{c}{90\% SCR}     &       & \multicolumn{4}{c}{95\% SCR} \\
\cline{2-5}\cline{7-10}          & \multicolumn{2}{c}{$n=500$} & \multicolumn{2}{c}{$n=800$} &       & \multicolumn{2}{c}{$n=500$} & \multicolumn{2}{c}{$n=800$} \\
          & \multicolumn{1}{c}{$Q_1$} & \multicolumn{1}{c}{$Q_2$} & \multicolumn{1}{c}{$Q_1$} & \multicolumn{1}{c}{$Q_2$} &       & \multicolumn{1}{c}{$Q_1$} & \multicolumn{1}{c}{$Q_2$} & \multicolumn{1}{c}{$Q_1$} & \multicolumn{1}{c}{$Q_2$} \\
\cline{2-5}\cline{7-10}    \multicolumn{4}{l}{Sieve estimator (Fourier basis)}  & &     & & & & \\
    \hline
    Quantile (2\%)   & \TPS{0.929} & \TPS{0.935} & \TPS{0.905} & \TPS{0.914} &       & \TPS{0.952} & \TPS{0.951} & \TPS{0.929} & \TPS{0.946} \\
    Quantile (50\%)   & \TPS{0.874} & \TPS{0.918}  & \TPS{0.911} & \TPS{0.910} &       & \TPS{0.893} & \TPS{0.936} & \TPS{0.914} & \TPS{0.933} \\
    Quantile (98\%)   &  0.889     & 0.924 &  0.909  & 0.913 &       & 0.916 & 0.940 & 0.925 & 0.939 \\
    $L_q$ ($q=2$)    & \TPS{0.921} & 0.935 & \TPS{0.920} & 0.908 &       & \TPS{0.926} & 0.959 & \TPS{0.927} & 0.924 \\
    Huber & 0.870 & 0.889 &  0.916  & 0.914 &       & 0.888 & 0.920 & 0.922 & 0.926 \\
    \hline
       \multicolumn{4}{l}{Sieve estimator (Legendre basis)}  & &     & & & & \\
    \hline
    Quantile (2\%)   & 0.913  & 0.909 & 0.904 &  \TPS{0.913}  &       & 0.921  & 0.921 &  0.923 & \TPS{0.930} \\
    Quantile (50\%)   & 0.919 & 0.917 & 0.916 & \TPS{0.921} &       & 0.927 & 0.930 & 0.929 & \TPS{0.930} \\
    Quantile (98\%)   & 0.863 & 0.905 & \TPS{0.929} & 0.914 &       & 0.886  & 0.914 & \TPS{0.940} & 0.930 \\
    $L_q$ ($q=2$)    & \TPS{0.932}  & \TPS{0.911} &  \TPS{0.930}  & \TPS{0.906} &       &   \TPS{0.941} & \TPS{0.918} & \TPS{0.942} & \TPS{0.918} \\
    Huber & \TPS{0.892} & \TPS{0.910} & \TPS{0.906} & \TPS{0.912} &  & \TPS{0.901} & \TPS{0.919} & \TPS{0.917} & \TPS{0.922} \\
    \hline
     \multicolumn{4}{l}{Sieve estimator (Wavelet basis)}  & &     & & & & \\
    \hline
    Quantile (2\%)  & 0.911 & 0.908 & 0.937 & 0.929 &       & 0.932 & 0.930  & 0.955 & 0.960 \\
    Quantile (50\%) & 0.909 & 0.892 & 0.924 & 0.906 &       &    0.934 & 0.922 & 0.940  & 0.931 \\
    Quantile (98\%) & 0.888 & 0.925 & 0.892 & 0.927 &       & 0.911 & 0.954 &  0.920 & 0.954 \\
    $L_q$ ($q=2$)   & 0.901 & 0.888 & 0.909  & 0.903 &       & 0.924 & 0.920 & 0.923 & 0.925 \\
    Huber & 0.928 & 0.930  & 0.919 & 0.928 &       & 0.950  & 0.954 & 0.951 & 0.960 \\
    \hline
    \multicolumn{3}{l}{Kernel estimator}  & & &     & & & & \\
    \hline
    Quantile (2\%)   & 0.669 & 0.297 & 0.645 & 0.169 &       & 0.750 & 0.395&   0.708 & 0.267 \\
    Quantile (50\%)   & 0.852  & 0.838 & 0.841 &  0.822 &       & 0.944 & 0.928 & 0.947 & 0.929 \\
    Quantile (98\%)   &  0.670   & 0.377 & 0.658 &  0.313  &       & 0.751 &  0.478   & 0.728 & 0.427 \\
    $L_q$ ($q=2$)  & 0.604  & 0.724 & 0.550 & 0.760 &       & 0.639 & 0.750 & 0.578 & 0.780 \\
    \hline
    \end{tabular}}%
    \label{tab:coverage}%
\end{table}%

We conclude that our SCRs achieve reasonably high accuracy and outperform the kernel estimators for all the commonly used sieve basis functions. 
The undercoverage of the kernel-based method \citep{chao2017confidence} is likely attributable to multivariate boundary effects inherent in kernel estimators \citep{liu2010simultaneous,chapman2000short}, as well as to its standard wild bootstrap implementation, which ignores temporal dependence. In contrast, our proposed sieve M-regression, combined with the self-convolved bootstrap, explicitly accommodates dependence and mitigates boundary issues, highlighting our non-trivial contribution in establishing SCRs for nonlinear time series models. We note that \citep{chao2017confidence} is designed for kernel-based inference with i.i.d. data in the interior region; thus, it is not surprising that their methodology does not directly extend to time series settings with boundary considerations. We emphasize this point to avoid any misunderstanding regarding the scope of their contributions.

 Additionally, sieve estimation possesses significant advantages in terms of computational efficiency. As discussed in \cite{chao2017confidence}, the kernel M-estimation is obtained by
\begin{equation}
    \hat{Q}(\mf{x})=\underset{Q \in \mathbb{R}}{\arg \min } \frac{1}{n} \sum_{i=1}^n K_h\left(\mf x-\mf X_i\right) \rho\left(Y_i-Q\right),\label{eq:kernel M-estimation}
\end{equation}
where $K_h(\mf x)$ is a multivariate kernel density function with bandwidth parameter $h$. \eqref{eq:kernel M-estimation} indicates that every new kernel estimation or prediction made on $\mf x$ requires solving the nonlinear optimization \eqref{eq:kernel M-estimation} one more time. On the other hand, the sieve estimation only requires one more inner product computation between the obtained $\hat{\theta}$ and the known basis function $\mf b(\mf x)$. Moreover, SCRs for the kernel method proposed in \cite{chao2017confidence} and \cite{chen2021multinon} rely on a wild bootstrap procedure generating new bootstrap observations $(\mf X_i^*,Y_i^*)$. 
This implies that the nonlinear optimization problem \eqref {eq:kernel M-estimation} needs to be solved repeatedly for $B$ bootstrap replicates, which introduces a heavy computational burden in practice. In contrast, our self-convolved bootstrap only requires repeating sampling Gaussian random variables for $B$ times, which relaxes the burden of computation.

\subsection{Real data application}
In this section, we analyze the daily S\&P500 return data based on the quantile regression model and apply our simultaneous inference framework. The quantile regression yields an entire conditional distribution relationship between variables, which has been increasingly used in economic studies. The estimation $Q_\tau(Y|\mathbf{X})=\inf\{ s:\mr P(Y\leq s|\mf X)>\tau\}$ for extreme $\tau$, say $\tau=0.98$ or $\tau=0.02$, represents the value at risk, or VaR, an important risk measure in finance and insurance. We consider the log-returns of the S\&P 500 index daily data from July 2, 2012, to January 5, 2015. The data can be downloaded from https://seekingalpha.com/symbol/SP500/historical-price-quotes. Let $Y_i=\log(S_{i+1})-\log(S_i)$, where the $S_i, i=1,\dots,n$ are the records of S\&P 500 daily index. As displayed in Figure \ref{fig:sp500}, the log-return series $Y_i$ appears to be stable and stationary during the aforementioned period. 
\begin{figure}[h]
    \centering
    \includegraphics[width=0.7\linewidth,height=3cm]{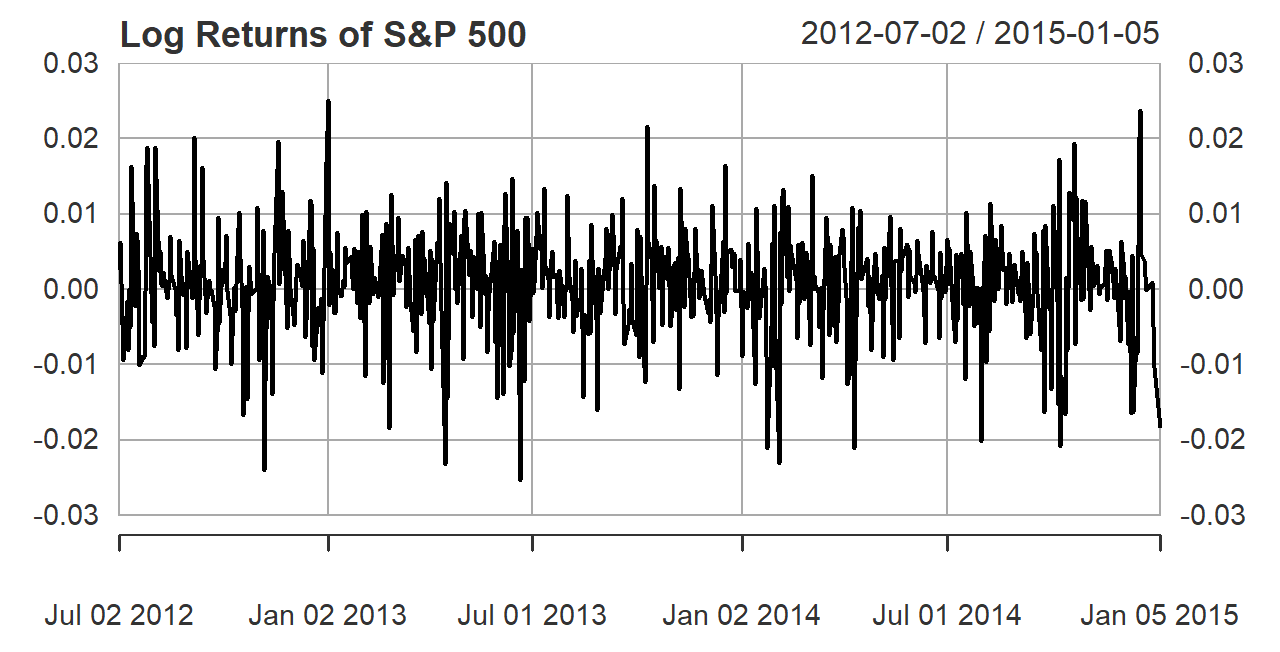}
    \captionsetup{font={small, stretch=1},skip=-0cm}
    \setlength{\belowcaptionskip}{-20pt} 
    \caption{The daily log returns of S\&P 500 index data.}
    \label{fig:sp500}
\end{figure}

In financial time series analysis, the quantile autoregression (QAR) model is widely used for modelling $Q_\tau(Y|\mathbf{X})$ \citep{xiao2004uqar,baur2012stock,baur2019quantile}. 
The standard linear QAR specification \citep{koenker2006quantile} can be written as
\begin{equation}
   Q_\tau(Y_i|Y_{i-1},\dots,Y_{i-d})=\theta_0(\tau)+\theta_1(\tau)Y_{i-1}+\dots+\theta_d(\tau)Y_{i-d},\label{eq:qar(d)}
\end{equation}
where $Y_i$ denotes the observed time series, $d$ is the lag order, and $\theta_j(\tau)$ is the autoregressive coefficient.
Our first application is not to develop a parametric model, but to use our SCR as a diagnostic tool to assess whether such a standard linear specification is adequate.
According to the BIC criterion of \cite{schwarz1978estimating} and \cite{rissanen1978modeling}, the optimal lag length $d$ in model \eqref{eq:qar(d)} is selected to be $d=1$ for different quantiles $\tau=0.02,0.5,0.98$. Based on the selected QAR(1) model, we conduct the linearity test for the QAR(1) hypothesis
\begin{equation}
    \label{eq:h0 QAR(1)}
    H_0:Q_{\tau}\left(Y_i|Y_{i-1}\right)=\theta_0(\tau)+\theta_1(\tau) Y_{i-1},
\end{equation}
against the alternative hypothesis that $Q_{\tau}(Y_i|Y_{i-1})$ is nonlinear in $Y_{i-1}$. 
As described in Section \ref{sec: introduction}, we can view $Q_\tau(Y_i|Y_{i-1})$ as a function of $Y_{i-1}$, namely $Q(x)=\inf\{s:\mr P(Y_i\leq s|Y_{i-1}=x)\}$, and estimate it by \eqref{eq:loss function} using $\mf X_i = Y_{i-1}$ and loss function $\rho(x) = \tau x^+ + (1 - \tau)(-x)^+$ via the sieve method. Following Algorithm \ref{alg:self-convolved bootstrap} with bootstrap sample size $B=2000$, we construct $(1 - \alpha)$-level SCR \eqref{eq:SCR boot} for the quantile $Q_\tau(Y_i|Y_{i-1})$. For the null hypothesis \eqref{eq:h0 QAR(1)}, we accept $H_0$ at significance level $\alpha$ if the fitted line $\hat\theta_0(\tau)+\hat\theta_1(\tau) Y_{i-1}$ is fully covered by the $(1 - \alpha)$-level SCR obtained by Algorithm \ref{alg:self-convolved bootstrap}. 
 The results show that the null hypothesis \eqref{eq:h0 QAR(1)} is accepted with p-values $0.167,0.317,0.384$ for $\tau=0.98,0.5,0.02$, respectively. Thus the linear QAR(1) model provides a statistically reasonable specification for modelling the conditional quantile relationship between $Y_i$ and $Y_{i-1}$ across tested quantiles.

Furthermore, beyond validating the QAR(1) specification, our simultaneous inference framework can examine whether an additional lag $Y_{i-2}$ contributes incremental explanatory power for the conditional quantiles of $Y_i$, possibly in a nonlinear way. Specifically, first we test 
$H_0: Q_\tau(Y_i \mid Y_{i-1}, Y_{i-2}) = Q_\tau(Y_i \mid Y_{i-1})$, corresponding to the case where $Y_{i-2}$ is redundant.
We construct the SCR for $Q_\tau(Y_i \mid Y_{i-1}, Y_{i-2})$ and assess whether the fitted univariate estimator $\hat Q_\tau(Y_i \mid Y_{i-1})$ is covered by it. 
The resulting p-values are $0.165$, $0.670$, and $0.030$ for $\tau=0.02, 0.5$, and $0.98$, respectively. 
Thus, the null cannot be rejected at the lower and median quantiles, but is rejected at $\tau=0.98$, indicating that $Y_{i-2}$ is relevant for modeling upper tail risk. To further examine the bivariate relationship, we test the linear QAR(2) specification for $\tau=0.98$, i.e., $H_0:Q_{0.98}(Y_{i}|Y_{i-1},Y_{i-2})=\theta_0+\theta_1Y_{i-1}+\theta_2Y_{i-2}$. The test yields a p-value of $0.018$, decisively rejecting the linear null hypothesis. These findings demonstrate that nonlinear modeling becomes essential when considering the multivariate target $Q_\tau(Y_i|Y_{i-1},Y_{i-2})$ for the upper tail risk. 
In addition to the testings, the SCR itself provides a statistical uncertainty quantification for the conditional VaR estimation, which is an important aspect of risk management (displayed in Figure \ref{fig:SCR 2-dim}). 
\begin{figure}[h]
\vspace{-0.4cm} 
    \centering
    \includegraphics[width=0.8\linewidth,height=5.5cm]{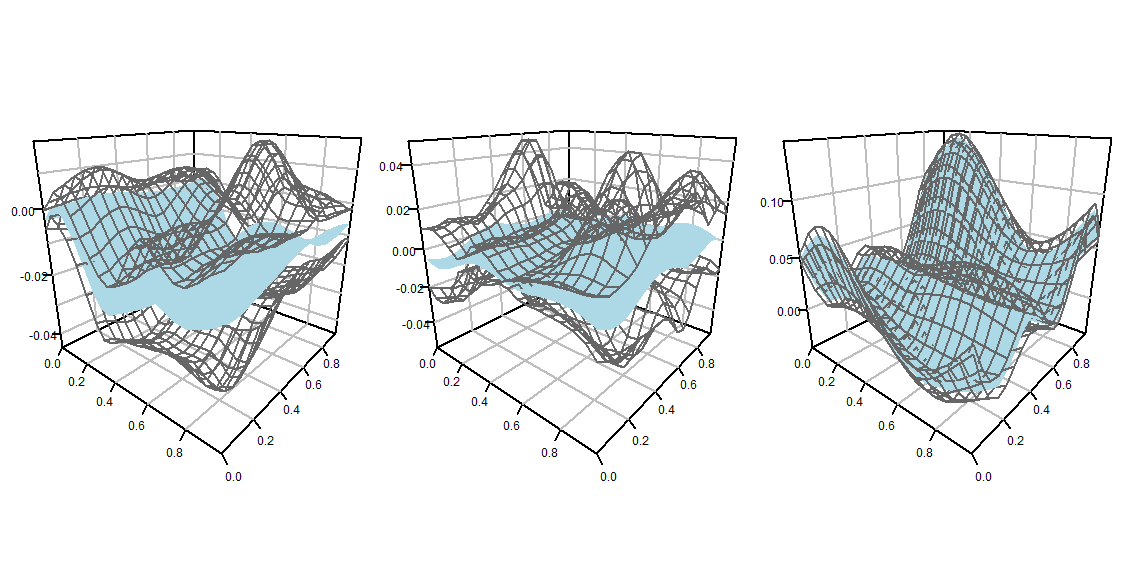}
    \captionsetup{font={small, stretch=1},skip=-0.7cm}
    \setlength{\belowcaptionskip}{-0.5pt} 
    \caption{95\% SCRs based on sieve estimator $\hat{Q}_{\tau}(Y_i|Y_{i-1},Y_{i-2})$ for quantile $\tau=0.02$ (left), $\tau=0.5$ (middle), and $\tau=0.98$ (right). The transparent surfaces represent the SCRs, and the opaque blue surface is the estimation $\hat{Q}_{\tau}(Y_i|Y_{i-1},Y_{i-2})$.}
    \label{fig:SCR 2-dim}
\end{figure}

\section{Definitions and assumptions}
\label{sec:def and assump}
We now introduce the necessary definitions and regularity assumptions. These assumptions are standard in nonlinear time series models and sieve-based M-estimation. 
According to the sieve estimator in \eqref{eq:sieve estimator} with sieve bases $\mf b_\omega(\mf x)$, we introduce the following mild regularity conditions that are satisfied by most linear sieve basis:
\begin{itemize}[itemsep=0pt,parsep=0pt,topsep=0pt,partopsep=0pt]
\setlength{\baselineskip}{0.7\baselineskip}
    \item[(A1)] \phantomsection\label{(A1)} Define $\xi_{K,n}=:\sup_{\mf x\in\mathcal{X}}|\mf b_\omega(\mf x)|$ and $\Delta_{K,n}=:\sup_{x \in \mathcal{D}_n}\left|\nabla \mf b_\omega(\mf x)\right|$. Suppose there exists $\omega_0,\omega_1,\omega_1^\prime\geq 0$ s.t. $K\lesssim n^{\omega_0}$, $\xi_{K,n}\lesssim n^{\omega_1}$, and $\Delta_{K,n}\lesssim n^{\omega_1'}$.
\item[(A2)] \phantomsection\label{(A2)} Define $\mf B_{K}=:\mb E(\mf b_\omega(\mf X_i) \mf b_\omega(\mf X_i)^\top )$. We denote the smallest and largest eigenvalues of $\mf B_{K}$ as $\lambda_{min}(\mf B_{K} ),\lambda_{max}(\mf B_{K} )$.  Assume there exist constants $c,C>0$ such that $c\leq\lambda_{min}(\mf B_{K} )\leq\lambda_{max}(\mf B_{K} )\leq C$ for any $K$.
    \item[(A3)] \phantomsection\label{(A3)} Assume the sieve coefficient $\theta$ in \eqref{eq:sieve estimator} lies in 
    $\Theta=:\left\{\theta\in\mb R^K: \sup_{\mf x\in\mathcal{X}}|\theta^\top \mf b_\omega(\mf x)|\leq C\right\}$, for some constant $C>0$.
    \item[(A4)] \phantomsection\label{(A4)} Define the function family
    $\mathcal{F}_\rho=:\{ \rho(\varepsilon-\theta^\top\mf b_\omega(\mf x)): \theta\in \mb R^K\}.$ For any probability measure $\mr P$, any $r>1$, and $\varepsilon\in(0,1)$, we assume the covering number of $\mathcal{F}_\rho$ satisfies $N(\varepsilon,\mathcal{F}_\rho,\|\cdot\|_{\mr P,r})\leq C K \left(A/\varepsilon\right)^{rK},$
    with universal constant $C,A>0$.
    \item[(A5)] \phantomsection\label{(A5)} There exists a constant \(\varsigma>0\) such that $\sup_{\mf x\in\mathcal D_n}
|Q_{0,n}(\mf x)-Q_0(\mf x)|
=O(K^{-\varsigma}).$
\end{itemize}
Assumption \hyperref[(A1)]{(A1)} is a mild regularity condition on the sieve basis functions and can be checked by calculation on the specific basis functions. If $\mathcal{X}$ is compact and rectangular, Assumption \hyperref[(A1)]{(A1)} is satisfied by commonly used sieve basis with typical $\xi_{K,n}\lesssim \sqrt{K}$; see \cite{newey1997convergence}, \cite{quan2024JASA}. 
Assumptions \hyperref[(A2)]{(A2)} and \hyperref[(A3)]{(A3)} trivially nest typical regression models and the H\"older function space, which aligns with many nonparametric sieve estimates such as \cite{CHEN2015uniformsieve} and \cite{quan2024JASA}. For orthonormal basis $\mf b(\mf x)$ on $\mf x\in\mathcal{D}_n$ in the sense that $\int_{\mathcal{D}_n} b_i(\mf x)b_j(\mf x)\mr d \mf x=0$ and $\int_{\mathcal{D}_n} b_i^2(\mf x)\mr d \mf x=1$, the $\lambda_{max}(\mf B_{K} )\leq C$ in Assumption \hyperref[(A2)]{(A2)} easily holds by $\beta^\top \mf B_K\beta=\|\beta^\top\mf b_\omega(\mf X_i)\|^2\leq  C\int_{\mathcal{D}_n}|\beta^\top \mf b(\mf x)|^2\mr d \mf x=C|\beta|^2.$
Similarly, the $\lambda_{min}(\mf B_{K,n} )>0$ in Assumption \hyperref[(A2)]{(A2)} holds if the density function of $\mf X_i$ has a positive lower bound $c>0$ on $\mathcal{D}_n$ such that $\beta^\top \mf B_K\beta\geq c\int_{\mathcal{D}_n}|\beta^\top \mf b(\mf x)|^2\mr d\mf x=c|\beta|^2.$
The boundedness condition \hyperref[(A3)]{(A3)} is essential in practice, considering the optimization algorithm for \eqref{eq:sieve estimator}. \hyperref[(A3)]{(A3)} is a mild condition requiring that the approximation sieve space does not deviate too far from $Q_0(\mf x)$. 
Section \ref{sec:example of loss} in the supplement verifies Assumption \hyperref[(A4)]{(A4)} for various loss functions $\rho(\cdot)$, including Huber's, expectile, and $\mathcal{L}^q$ regressions. Assumption \hyperref[(A5)]{(A5)} can be verified by imposing smoothness conditions on \(Q_0(\cdot)\). For example, a \(p\)-smooth H\"older class  takes \(\varsigma=p/d\); see detailed verification in Section \ref{sec:supp-verify A5} in the supplement.

Our assumptions on the time series are connected with the loss function in the M-estimation problem. For the loss function in \eqref{eq:sieve estimator}, we require
\begin{equation}
|\psi(x)-\psi(y)| \leq M_1+M_2|x-y|,\label{eq:condition psi}
\end{equation}
for all $x, y \in \mathbb{R}$ and some positive constants $M_1$ and $M_2$. The condition \eqref{eq:condition psi} holds for various loss functions of quantile, expectile, $\mathcal{L}_q$ for $1<q<2$, least squares, and the Huber regressions. For integer $q\geq0$, write
\begin{equation}
    \overline{\Xi}_n^{(q)}\left(x \mid \mf X_i\right)=:\frac{\partial^q}{\partial x^q}\mb E\left[ \psi(\bar\varepsilon_i+x)-\psi(\bar\varepsilon_i) \mid \mf X_i\right] ,\label{eq:bar Xi}
\end{equation}
where the superscript $(q)$ is omitted if $q=0$. We assume the time series $\{\mf X_i\}$ and $\{\bar\varepsilon_i\}$ are written as $\mf X_i=\mf H(\mathcal{H}_i)$, $\bar\varepsilon_i=\bar G(\mathcal{H}_i,\mathcal{G}_i)$ and satisfy the following mild conditions:
\begin{itemize}[itemsep=0pt,parsep=0pt,topsep=0pt,partopsep=0pt]
\setlength{\baselineskip}{0.7\baselineskip}
    \item[(B1)] \phantomsection\label{(B1)} $\delta_{\mf H}(k,1)= O(\chi^{k})$ and 
     $\mathbb{E}\left(\exp \left(t_0\left|\mathbf{X}_i\right|\right)\right) < \infty$ for some $\chi\in(0,1)$ and $t_0>0$.
    \item[(B2)] \phantomsection\label{(B2)} Define $\delta_\psi(k,q) =:\|\psi(\bar G(\mathcal{H}_i,\mathcal{G}_i))-\psi(\bar G(\mathcal{H}_{i,k},\mathcal{G}_{i,k}))\|_q$.
    Assume $\delta_\psi(k,4)=O(\chi^k)$. 
     \item[(B3)] \phantomsection\label{(B3)} For $t(\theta)=\sup_\mf x|\theta^\top \mf b_\omega(\mf x)|$ and the modulus function $\omega(\cdot)$ satisfying $\omega(u)\to 0,u\to 0$, assume there exists a constant $\delta_0>0$ such that for any $\theta,h\in \mb R^K$ and $|\theta|,|h|\leq \delta_0$,
     \begin{align} \text{(i). }&\mb E\left[ \big(\overline{\Xi}_n(\theta^\top \mf b_\omega(\mf X_i)|\mf X_i)-\overline{\Xi}_n^{(1)}(0|\mf X_i)\theta^\top\mf b_\omega(\mf X_i) \big)h^\top \mf b_\omega(\mf X_i) \right]=O(t(\theta) |\theta||h| ),\notag \\ &\mb E\left[ \overline{\Xi}_n( (\theta+h)^\top \mf b_\omega(\mf X_i) |\mf X_i)-\overline{\Xi}_n( \theta^\top \mf b_\omega(\mf X_i) |\mf X_i)-\overline{\Xi}_n^{(1)}( \theta^\top \mf b_\omega(\mf X_i)|\mf X_i)h^\top \mf b_\omega(\mf X_i) \right]=O(h^2),\notag\\ &\mb E\left[ \big(\overline{\Xi}_n^{(1)}( \theta^\top \mf b_\omega(\mf X_i)|\mf X_i)-\overline{\Xi}_n^{(1)}( 0|\mf X_i) \big)(h^\top \mf b_\omega(\mf X_i))^2\right]=O(\omega(t(\theta)) h^2).\notag \end{align}
     (ii). For $\overline{\Xi}_n^{(1)}(0\mid \mathcal{H}_i)$, we require that $\overline{\Xi}_n^{(1)}\left(0 \mid \mathcal{H}_i\right)|> 0$ a.s.\\
     (iii). Define 
        $\nu(\delta)=\mb E\left\{[\psi(\bar\varepsilon_i+\delta^\top\mf b_\omega(\mf X_i))-\psi(\bar\varepsilon_i-\delta^\top \mf b_\omega(\mf X_i))]^2|\delta^\top \mf b_\omega(\mf X_i)|^2/|\delta|^2\right\}$. Assume that $\nu(\delta)$ is continuous at $\delta=0$ and $\nu(\delta)=O(|t(\delta)|^\eta)$ for some $\eta>0$.   
\end{itemize}

 Assumptions \hyperref[(B1)]{(B1)}, \hyperref[(B2)]{(B2)}, and \hyperref[(B3)]{(B3)} are in line with those in many classic M-estimation literature such as \cite{wu2007m}, \cite{wu2018gradient}. 
 Assumptions \hyperref[(B1)]{(B1)} and \hyperref[(B2)]{(B2)} require $\{\mf X_i\}$ and $\{\psi(\bar\varepsilon_i)\}$ enjoy a geometrically decaying dependence measure, 
 which can be satisfied in most commonly used time series models. We refer to \cite{wu2007m} for representative examples showing Assumptions \hyperref[(B1)]{(B1)} and \hyperref[(B2)]{(B2)}.
 In the supplement, we also verify this using the examples \eqref{eq:VARMA} and \eqref{eq:GARCH}. The geometric decay $O(\chi^k)$ can be relaxed to a polynomial decay. However, this requires more complicated conditions and substantially more intricate mathematical arguments. For simplicity, we stick to the geometric decay assumption. 

 Assumption \hyperref[(B3)]{(B3)} is essential for obtaining a consistent M-estimator and a Bahadur-type representation. 
 This condition is in the same spirit as, and can be implied by, a twice continuously differentiable condition in Assumption D(iv) of \cite{cattaneo2025uniform}.
 Section \ref{sec:supp-verify B3} in the supplement verifies (i) in Assumption \hyperref[(B3)]{(B3)} for several commonly used loss functions. (ii) gives boundedness conditions at $x=0$ so that a stochastic expansion can be applied. 
 The condition on $\nu(\delta)$ in (iii) 
 plays an important role in determining the convergence rate of our Bahadur representation and can be satisfied by various loss functions. For quantile regression, we have $\eta=1$ if mild conditions are assumed such that
 the conditional density function of $\varepsilon_i$ is uniformly bounded away from 0 and $\infty$ in a neighborhood around $0$ \citep{Wang2012Ultra-high-quantreg, Wu2017quantile}. For the Huber, least squares, and expectile regressions, as in the discussion of Example 1 of \cite{wu2007m}, we have $\eta=2$. For $L_q$ regression with $1< q\leq 2$, it can be shown that $\eta=2$ if $3/2<q\leq 2$ and $\eta=2q-1$ if $1<q\leq 3/2$; 
 see \cite{wu2018gradient}.
\bigskip
\begin{center}
{\large\bf SUPPLEMENTARY MATERIAL}
\end{center}
In the supplement, we provide
\begin{itemize}[itemsep=0pt,parsep=0pt,topsep=0pt,partopsep=0pt]
\setlength{\baselineskip}{0.7\baselineskip}
    \item[A.] Detailed proofs of the theorems with auxiliary lemmas.
    \item[B.] Assumption verifications and specific examples of the time series models, loss functions, and sieve spaces including Fourier series, Legendre polynomials, and wavelets.
\end{itemize}



    \linespread{0.949}\selectfont 
    \setlength{\bibsep}{5pt}     
    \putbib
\end{bibunit}

\clearpage
\appendix
\begin{center}
{\large\bf SUPPLEMENTARY MATERIAL for ``Simultaneous Inference for Nonlinear Time Series, a Sieve M-regression Approach"}
\end{center}

\begin{bibunit}
\newcommand{\mf}{\mathbf}
\newcommand{\mr}{\mathrm}
\newcommand{\mb}{\mathbb}
\renewcommand{\thesection}{\Alph{section}}
\renewcommand{\thetable}{\thesection.\arabic{table}}
\renewcommand{\thefigure}{\thesection.\arabic{figure}}
\setcounter{section}{0} 
This supplementary material will provide detailed proofs in Section \ref{sec:proof} and a further appendix in Section \ref{sec:appendix}. Section \ref{sec:proof} contains proofs of the theorems and propositions in the main paper with essential lemmas and auxiliary results. Section \ref{sec:appendix} displays specific examples and associated verification to show the generality of our methodology framework. 

In particular, Section \ref{sec:example of time series models} introduces a series of time series models in the formulation \eqref{eq:causal representation}, including AR, MA, and GARCH, and verifies their dependence structure. Sections \ref{sec:supp-verify B3}, \ref{sec:example of loss}, and \ref{sec:supp-verify A5} verify Assumption \hyperref[(B3)]{(B3)}, \hyperref[(A4)]{(A4)}, and \hyperref[(A5)]{(A5)} for different loss functions, including quantile, expectile, $L_q$, and Huber's. Commonly used sieve basis functions and verifications of their geometric quantities in Proposition \ref{prop:critical value} can be seen in Section \ref{sec:examples of sieve basis}.

\section{Proofs}
\label{sec:proof}
\subsection{Proof of Theorem \ref{thm:maximal inequality}}
\label{sec:proof of thm:maximal inequality}
\begin{proof}
Notice that our data is not independent, we firstly decompose the process $\mb P_n(\beta)-\mr P(\beta) $ into $m$-independent blocks. Recall $\mathcal{H}_i$ and $\mathcal{G}_i$ in \eqref{eq:causal representation}, we denote $\mathcal{H}_i^{(m)}=:(\zeta_{i-m+1},\dots,\zeta_{i-1},\zeta_i)$, $\mathcal{G}_i^{(m)}=:(\eta_{i-m+1},\dots,\eta_{i-1},\eta_i)$, and
$$
L_i^{(m)}(\beta)=:\mb E\left(L_i(\beta)|\mathcal{H}_i^{(m)},\mathcal{G}_i^{(m)}\right).
$$
Without loss of generality, suppose $n/m\in\mb Z$ then $\{L_{m(i-1)+j}^{(m)}(\beta)\}_{i=1}^{n/m}$ is i.i.d. process for each $j=1,\dots,m$, on which maximal inequality in Empirical process can be applied. Following arguments show that the error for such m-dependent decomposition can be bounded. For simplicity, we call $\Upsilon_i=(\dots,\gamma_0,\gamma_1,\dots,\gamma_i)$ with $\gamma_i=(\zeta_i,\eta_i)$ so that $L_i^{(m)}(\beta)=\mb E\left(L_i(\beta)|\Upsilon_i^{(m)}\right)$. 

Denote $\{\gamma_i^\prime\}$ as an i.i.d. copy of $\{\gamma_i\}$ and $\Upsilon_i^\prime = (\dots,\gamma_{i-1}^\prime,\gamma_i^\prime)$, then we have
    \begin{align}
        |L_i(\beta)-L_i^{(m)}(\beta)|& = \left|\mb E\left[ L_i(\beta)-\mb E\left( L_i(\beta)|\Upsilon_{i-m}^\prime,\gamma_{i-m+1},\dots,\gamma_i \right)\Big|\Upsilon_i \right]\right|\notag\\
        &\leq \sum_{k\geq m}^\infty \left|\mb E\left[ \mb E\left(L_i(\beta)|\Upsilon_{i-k-1}^\prime,\gamma_{i-k},\dots,\gamma_{i}  \right) -\mb E\left( L_i(\beta) |\Upsilon_{i-k}^\prime,
        \gamma_{i-k+1},\dots,\gamma_{i}\right) |\Upsilon_i\right] \right |\notag.
    \end{align}
    Thus for $\mathcal{H}_{i,k}=(\mathcal{H}_{i-k-1},\zeta_{i-k}^\prime,\zeta_{i-k+1},\dots,\zeta_i)$ and $\mathcal{G}_{i,k}=(\mathcal{G}_{i-k-1},\eta_{i-k}^\prime,\eta_{i-k+1},\dots,\eta_i)$,
    \begin{equation}
        \mb E \sup_{|\beta|\leq r_n}|L_i(\beta)-L_i^{(m)}(\beta)| \leq \sum_{k\geq m}^\infty \mb E\sup_{|\beta|\leq r_n}|L_i(\beta)-L_{i,k}(\beta)|  , \label{eq:sum up over m}    
    \end{equation}
    where $\mf X_{i,k}=\mf H(\mathcal{H}_{i,k})$, $\bar\varepsilon_{i,k}=a_n(\mf X_{i,k})+G(\mathcal{H}_{i,k},\mathcal{G}_{i,k})$, and
   $L_{i,k}(\beta)=:\rho(\bar\varepsilon_{i,k}- \beta^\top\mf b_\omega(\mf X_{i,k}))-\rho(\bar\varepsilon_{i,k})$.
    By Assumptions \hyperref[(A3)]{(A3)} and Lemma \ref{lem:phd on rho}, using the triangular inequality, there exists constant $\chi\in(0,1)$, for any given $\alpha\in(0,1)$,
    \begin{align}
    \mb E\sup_{|\beta|\leq r_n}|L_i(\beta)-L_{i,k}(\beta)|&=O\left(\chi^k+r_n\xi_{K,n}\Delta_{K,n}^\alpha\chi^{\alpha k}\right).\label{eq:bound L}
    \end{align} 
    Combining \eqref{eq:bound L} and \eqref{eq:sum up over m}, then we have
\begin{equation}
        \mb E\left(\sup_{|\beta|\leq r_n} |L_i(\beta)-L_i^{(m)}(\beta)|\right) =O(\chi^m+r_n\xi_{K,n}\Delta_{K,n}^{\alpha} \chi^{\alpha m}).\label{eq:loss error for M-decomp}
\end{equation}

Note that $\sup_{|\beta|\leq r_n}\sup_{\mf x\in\mathcal{D}_n}|\beta^\top \mf b_{\omega}(\mf x)|\leq r_n\xi_{K,n}$, we consider function family
\begin{align}
    \mathcal{R}_L&=:\left\{f_\beta(\varepsilon,\mf x)= \rho\left(\varepsilon-\beta^\top\mf b_\omega(\mf x)\right)-\rho(\varepsilon): |\beta|\leq r_n\right\}.\label{eq: R2 function family}
\end{align}
Note that the sequence $\{L_{m(i-1)+j}^{(m)}(\beta)\}_{i=1}^{n/m}$ is i.i.d., by the standard symmetrization \citep{van1996weak}, we have
\begin{align}
    &\mb E\sup_{|\beta|\leq r_n}\left|\frac{1}{\sqrt{n/m} }\sum_{i=1}^{n/m}(L_{m(i-1)+j}^{(m)}-\mb EL_{m(i-1)+j}^{(m)})\right| \notag\\
    &\leq 2\mb E\sup_{|\beta|\leq r_n}\left|\frac{1}{\sqrt{n/m} }\sum_{i=1}^{n/m}\epsilon_iL_{m(i-1)+j}^{(m)}(\beta)\right|, \label{eq:symmetrization}
\end{align}
where $\{\epsilon_i\}$ are i.i.d. Rademacher variables.

By condition \eqref{eq:condition psi} and the fact $\rho(\nu-u)-\rho(u)=-\nu \int \psi(u-t\nu)\mr dt$, there exists constant $c_0,c_1>0$ such that for any $f_\beta\in\mathcal{R}_L$,
\begin{align}
    \left| f_\beta(\varepsilon,\mf X) \right|&\leq \left| \beta^\top\mf b_\omega(\mf X) \right|\int_0^1|\psi(\varepsilon-t\beta^\top\mf b_\omega(\mf X))\mr dt|,\notag\\
    &\leq |\beta^\top\mf b_\omega(\mf X)| (c_0+\psi(\varepsilon) )+c_1|\beta^\top \mf b_\omega(\mf X)|^2.\label{eq:multiplier reduction}
\end{align}
Hence, the symmetrized supremum is bounded by a sum of two terms
\begin{equation}
    \mb E\sup_{|\beta|\leq r_n}\left|\frac{1}{\sqrt{n/m} }\sum_{i=1}^{n/m}\epsilon_iL_{m(i-1)+j}^{(m)}(\beta)\right|\leq T_{1,j}+T_{2,j},
\end{equation}
where
\begin{align}
    T_{1,j}&=:\mb E\sup_{|\beta|\leq r_n}\left| \frac{1}{\sqrt{n/m} }\sum_{i=1}^{n/m} \epsilon_i\mb E\left[\beta^\top \mf b_\omega(\mf X_{i(m-1)+j}))(c_0+\psi(\bar\varepsilon_{m(i-1)+j}) )\mid \mathcal{H}_{i(m-1)+j}^{(m)},\mathcal{G}_{i(m-1)+j}^{(m)}\right] \right|,\notag\\
    T_{2,j}&=:c_1\mb E\sup_{|\beta|\leq r_n}\left| \frac{1}{\sqrt{n/m} }\sum_{i=1}^{n/m} \epsilon_i\mb E\left[(\beta^\top \mf b_\omega(\mf X_{i(m-1)+j}^{(m)}))^2\mid \mathcal{H}_{i(m-1)+j}^{(m)},\mathcal{G}_{i(m-1)+j}^{(m)}\right]\right|.
\end{align}
Let $\mf V_{i,j}^{(m)}=:\mb E\left[\beta^\top \mf b_\omega(\mf X_{i(m-1)+j}))(c_0+\psi(\bar\varepsilon_{m(i-1)+j}) )\mid \mathcal{H}_{i(m-1)+j}^{(m)},\mathcal{G}_{i(m-1)+j}^{(m)}\right]$, and notice that
$$
\mb E_\epsilon\left|\sum_{i=1}^{n/m} \epsilon_i \mf V_{i,j}^{(m)}\right|\leq \left(\mb E_\epsilon\left|\sum_{i=1}^{n/m} \epsilon_i \mf V_{i,j}^{(m)}\right|^2\right)^{1/2}=\left(\sum_{i=1}^{n/m} |\mf V_{i,j}^{(m)}|^2\right)^{1/2},
$$
where $\mb E_\epsilon(\cdot)$ denotes expectation with respect to the Rademacher variables $\{\epsilon_i\}$. Therefore, by $\sup_{|\beta|\leq r_n}|\beta^\top \nu|\leq r_n|\nu|$ and Assumption \hyperref[(B2)]{(B2)},
\begin{align}
    T_{1,j}&\leq r_n\mb E\left| \frac{1}{\sqrt{n/m}}\sum_{i=1}^{n/m} \epsilon_i \mf V_{i,j}^{(m)} \right|\leq  r_n\left\| \frac{1}{\sqrt{n/m}}\sum_{i=1}^{n/m}  \mf V_{i,j}^{(m)} \right\| =O(r_n\xi_{K,n}), \label{eq:bound for T_1,j}
\end{align}
uniformly over $j=1,\dots,m$. Similarly, on the other hand,
\begin{align}
    T_{2,j}&\leq c_1 r_n^2\mb E\left| \frac{1}{\sqrt{n/m}}\sum_{i=1}^{n/m} \epsilon_i \mf |b_\omega(\mf X_{m(i-1)+j})|^2 \right|= r_n \|\mf b_\omega(\mf X_i)\|_4^2 =O(r_n^2\xi_{K,n}\sqrt{K}).\label{eq:bound for T_2,j}
\end{align}


Combining \eqref{eq:bound for T_1,j}, \eqref{eq:bound for T_2,j}, \eqref{eq:symmetrization}, and \eqref{eq:loss error for M-decomp}, with appropriate setting $m\asymp \log n$ (e.g., $m= \frac{\log n}{\alpha\log(1/\chi)}$ ) and the fact $\xi_{K,n}r_n=o(1)$, we have
\begin{equation}
        \mb E\sup_{|\beta|\leq r_n} | \mb P_n(\beta)- \mb P(\beta)|=O\left(n^{-1}+r_n\xi_{K,n}\Delta_{K,n}^\alpha n^{-1}+r_n (\sqrt{K}\vee\xi_{K,n}) \sqrt{\frac{\log n}{n}}\right). 
\end{equation}
Then setting $0<\alpha\leq  \min\{(\omega_0+1)/2\omega_1',1\}$, the \eqref{eq:maximal ineq for Pn-P} holds.
\end{proof}

\subsection{Bahadur representations}

\begin{proof}[Proof of Proposition \ref{prop:Bahadur representation}.]
Recall $\mf Z_n$ in \eqref{eq:Z_n}, and define 
$$
\beta_{\mf Z}=:\frac{1}{\sqrt{n}}\bar Q_n^{-1}\mf Z_n,\quad e_n=:\hat{\beta}-\beta_{\mf Z},
$$
where $\bar Q_n=\mb E[\overline \Xi_n^{(1)}(0|\mf X_i) \mf b_\omega(\mf X_i)\mf b_\omega(\mf X_i)^\top ]$.
Hence, to establish the Bahadur representation in Proposition \ref{prop:Bahadur representation}, it suffices to show that $e_n$ satisfies the desired remainder bound.

Define
\begin{equation}
\widetilde R_i(\theta,h)
=:
L_i(\theta+h)-L_i(\theta)+\psi(\bar\varepsilon_i)h^\top\mf b_\omega(\mf X_i),\label{eq:remainder}    
\end{equation}
and $\widetilde{\mb R}_n(\theta,h)=:\frac1n\sum_{i=1}^n \widetilde R_i(\theta,h)$,
$\widetilde{\mb R}(\theta,h)=:\mb E\widetilde{\mb R}_n(\theta,h)$. Then we have
\begin{align}
        \mb P_n(\beta_{\mf Z}+e_n)-\mb P_n(\beta_{\mf Z})&=-\frac{1}{\sqrt{n}}\mf Z_n^\top e_n+\big(\tilde{\mb R}_n(\beta_\mf Z,e_n)-\tilde{\mb R}(\beta_\mf Z,e_n)\big)+\tilde{\mb R}(\beta_\mf Z,e_n).\label{eq:localization-exact decomposition}
\end{align}
Applying the deterministic expansion in Lemma \ref{lem:deterministic expansion}, we have
\begin{equation}
    \tilde{\mb R}(\beta_{\mf Z},e_n)=e_n^\top \bar{Q}_n\beta_{\mf Z}+\frac{1}{2}e_n^\top \bar{Q}_ne_n+r_P(\beta_{\mf Z},e_n),\label{eq:deterministic expansion beta_z e_n}
\end{equation}
where $r_P(\theta,h)=O(t(\theta)|\theta| |h|+t(h)|h|^2+t(\theta) |h|^2 )$.

Together with (ii) in Proposition \ref{prop:quadra approx} and (i) in Lemma \ref{lem:sup_bZ}, we have $t(e_n)=o_p(1)$ and
$r_P(\beta_{\mf Z},e_n)=O_p(\bar r_n^2|e_n|)+o_p(|e_n|^2)$, where $\bar r_n= (\xi_{K,n}\vee K^{1/2})n^{-1/2}\log n$. Since
$$
\frac{1}{\sqrt{n}}\mf Z_n^\top e_n=e_n^\top\bar{Q}_n\beta_\mf Z,\quad \mb P(\beta_\mf Z+e_n)=\mb P_n(\hat{\beta})\leq \mb P(\beta_{\mf Z}),
$$
combining \eqref{eq:localization-exact decomposition} and \eqref{eq:deterministic expansion beta_z e_n} yield
\begin{equation}
    0\geq -|e_n|\Gamma_n(t(\beta_{\mf Z}),t(e_n)) + \frac{1+o_p(1)}{2}e_n^\top \bar Q_ne_n+O_p(\bar r_n^2 |e_n|),\label{eq:ineq iteration bahadur bound}
\end{equation}
where
\begin{equation}
    \Gamma_n(a_\theta,a_h) = : \sup \left\{ \left|\frac{\tilde{\mb R}_n(\theta,h)-\tilde{\mb R}(\theta,h) }{|h|}\right|:|\theta|\leq \bar r_n,0<|h|\leq 2\bar r_n,t(\theta)\leq a_\theta,t(h)\leq a_h \right\}
\end{equation}
By Assumption \hyperref[(A2)]{(A2)} and \hyperref[(B3)]{(B3)}, there exists constant $c_0>0$ such that $e_n^\top \bar Q_ne_n\geq c_0|e_n|^2$. Therefore, \eqref{eq:ineq iteration bahadur bound} yields
\begin{equation}
    |e_n|=O_p(\Gamma_n(t(\beta_{\mf Z}),t(e_n))+\bar r_n^2),\label{eq:main iteration bahadur bound}
\end{equation}
where the $(a_\theta,a_h)$ will be specified in each step below. 
In the following steps, we proceed with a finite iterative approach to show that \eqref{eq:main iteration bahadur bound} can achieve \eqref{eq:bahadur beta}.

\paragraph{Beginning step.}
Let 
$$
A_{n,0}=:\{|\hat{\beta}|\leq \bar r_n,|\beta_{\mf Z}|\leq \bar r_n, t(\beta_{\mf Z})\leq \bar r_n\}.
$$
By Proposition \ref{prop:quadra approx}, \eqref{eq:bound for Z_n}, and Lemma \ref{lem:sup_bZ}, we have $\mr P(A_{n,0})=1-o(1)$. Denote event
$$
B_{n,0}=:\left\{ \Gamma_n(t(\beta_{\mf Z}),t(e_n) )\leq g_n K^{1/2}n^{-1/2}(\bar r_n+\nu_{n,0})^{\eta/2}\log^2 n \right\},
$$
where $\nu_{n,0}=\xi_{K,n}\bar r_n$ and $g_n\to \infty$ arbitrarily slow. Condition on $A_{n,0}$, we have $t(e_n)\leq \xi_{K,n}|\hat\beta|+t(\beta_{\mf Z})\leq \xi_{K,n}\bar r_n$ at the starting rate. Applying Lemma \ref{lem:bound centered remainder} with $r_\theta=\bar r_n$, $r_h=2\bar r_n$, $a_\theta=\bar r_n$, and $a_h=\nu_{n,0}$, we have $\mr P(B_{n,0}\cap A_{n,0})=1-o(1)$. Substituting the bound in event $B_{n,0}$ into \eqref{eq:main iteration bahadur bound} gives
\begin{equation}
    |e_n|=O_p(r_{n,0}),\quad r_{n,0}=\Phi_n(\bar r_n,\nu_{n,0}),\label{eq:starting bahadur rate}
\end{equation}
where $\Phi_n(a,b)=:g_n(K^{1/2}n^{-1/2}(a+b)^{\eta/2}\log^2 n)+\bar r_n^2$. Consequently, $t(\hat \beta)\leq \bar r_n+\xi_{K,n}|e_n|=O_p(\nu_{n,1})$ where $\nu_{n,1}=\bar r_n+\xi_{K,n}r_{n,0}$.

\paragraph{First self-refinement.}
On the event $A_{n,0}\cap B_{n,0}$, we can re-localize the stochastic remainder in Lemma \ref{lem:bound centered remainder} with
$r_\theta=\bar r_n$, $r_h=2\bar r_n$, $a_\theta = \bar r_n$, $a_h=\nu_{n,1}$. In this way we can replace event $B_{n,0}$ by $B_{n,1}$ with $\mr P(B_{n,1})=1-o(1)$ where
$$
B_{n,1}=:\left\{ \Gamma_n(t(\beta_{\mf Z}),t(e_n) )\leq g_n K^{1/2}n^{-1/2}(\bar r_n+\nu_{n,1})^{\eta/2}\log^2 n \right\}.
$$
Therefore, \eqref{eq:main iteration bahadur bound} yields
\begin{equation}
    |e_n|=O_p(r_{n,1}),\quad r_{n,1}=\Phi_n(\bar r_n,\nu_{n,1}),\label{eq:starting bahadur rate}
\end{equation}
Consequently, $t(\hat \beta)\leq \bar r_n+\xi_{K,n}|e_n|=O_p(\nu_{n,2})$ where $\nu_{n,2}=\bar r_n+\xi_{K,n}r_{n,1}$.

\paragraph{Second self-refinement.} A third application of Lemma \ref{lem:bound centered remainder} with $r_\theta=\bar r_n$, $r_h=2\bar r_n$, $a_\theta = \bar r_n$, $a_h=\nu_{n,2}$, gives the event $B_{n,2}$ with $\mr P(B_{n,2})=1-o(1)$ such that
\begin{equation}
B_{n,2}=:\left\{ \Gamma_n(t(\beta_{\mf Z}),t(e_n) )\leq g_n K^{1/2}n^{-1/2}(\bar r_n+\nu_{n,2})^{\eta/2}\log^2 n \right\}.
\end{equation}
Substituting again into \eqref{eq:main iteration bahadur bound} yields
$$
|e_n|=O_p(r_{n,2}),\quad r_{n,2}=\Phi_n(\bar r_n,\nu_{n,0}).
$$
It remains to evaluate whether the final recursive $r_{n,2}$ achieves the target Bahadur rate \eqref{eq:bahadur beta}. Basic calculation yields
$$
r_{n,2}=g_n(K^{1/2}n^{-1/2}(\bar r_n+\nu_{n,2})^{\eta/2}\log^2 n)+(K\vee \xi_{K,n}^2)n^{-1}\log^2 n.
$$
Using the fact $\eta\in[1,2]$ in Assumption \hyperref[(B3)]{(B3)}, it suffices to show $\nu_{n,2}\lesssim \bar r_n$, which means $\xi_{K,n}r_{n,1}\lesssim \bar r_n$. Notice that $\xi_{K,n}\bar r_n=o(1)$, we only need to verify $(\xi_{K,n}r_{n,0})^{\eta/2}\log n=o(1)$, which can be implied by condition $\xi_{K,n} K^{\eta/4}n^{-\eta/4}\log^2 n=o(1)$.

\end{proof}

\begin{proposition}
        \label{prop:quadra approx}
    If Assumptions \hyperref[(A1)]{(A1)}-\hyperref[(A4)]{(A4)} and \hyperref[(B1)]{(B1)}-\hyperref[(B3)]{(B3)} hold, we have the following assertions:
    \begin{description}
         \item[(i)] For $|\beta|\to0$, 
        \begin{equation}
        \mb P(\beta)=\frac{1}{2}\beta^\top \bar{Q}_n\beta+ o(|\beta|^2).\label{eq:quadra approx}    
    \end{equation}
        \item[(ii)] If $r_n=(\xi_{K,n}\vee K^{1/2})n^{-1/2}\log n$ and $\xi_{K,n}K^{1/2}n^{-1/2}\log n=o(1)$, then $\hat{\beta}=o_p(r_n)$.
     \end{description}
\end{proposition}
\begin{proof}
Using the fact $\rho(x-y)-\rho(x) = \int_0^1 \psi(x-yt)y\mr d t  $, by Assumption \hyperref[(B3)]{(B3)} and the sieve first-order condition $\mb E[\psi(\bar{\varepsilon}_i)\mf b_\omega(\mf X_i)]=0$, we have
\begin{align}
    \mb P(\beta)&=\mb E[\rho(\bar{\varepsilon}_i-\beta^\top \mf b_\omega(\mf X_i))-\rho(\bar{\varepsilon}_i)],\notag\\
    &=-\int_0^1 \mb E\left[ \psi\left(\bar\varepsilon_i-t\beta^\top\mf b_\omega(\mf X_i) \right)\beta^\top\mf b_\omega(\mf  X_i) \right]\mr d t,\notag\\
    &=- \int_0^1 \mb E\left[ \bar \Xi_n\left(-t\beta^\top \mf b_\omega(\mf X_i)\Big|\mf X_i \right)\beta^\top\mf b_\omega(\mf  X_i) \right]\mr d t,\notag\\
    &=\int_0^1 \mb E\left[ \bar \Xi_n^{(1)}\left(0\Big|\mf X_i \right)|\beta^\top\mf b_\omega(\mf  X_i)|^2\right] t\mr d t + o\left( |\beta|^2 \right),\notag\\
    &=\frac{1}{2}\mb E\left[ \bar \Xi_n^{(1)}\left(0\Big|\mf X_i \right)|\beta^\top\mf b_\omega(\mf  X_i)|^2\right] + o\left( |\beta|^2 \right)=\frac{1}{2}\beta^\top \bar{Q}_n\beta+o(|\beta|^2),\label{eq:quadra on risk}
\end{align}
which yields \eqref{eq:quadra approx}.

Combining (i)-(ii) in Assumption \hyperref[(B3)]{(B3)} and Assumption \hyperref[(A2)]{(A2)}, there exists universal constant $\underline{p},\overline{p}>0$ s.t.
\begin{equation}
    \underline{p}|\beta|^2\leq \frac{1}{2}\beta^\top \bar{Q}_n\beta \leq \overline{p}|\beta|^2.
\end{equation}
therefore, we can denote
\begin{equation}
    p(0)|\beta|^2=: \frac{1}{2}\mb E\left\{ \bar \Xi^{(1)}\left(0\Big|\mf X_i \right)|\beta^\top\mf b_\omega(\mf  X_i)|^2\right\}\label{eq:p(0) def}
\end{equation}
where $p(0)$ is independent of $n$ and $p(0)\in[\underline{p}, \overline{p}]$. Combining \eqref{eq:p(0) def} and \eqref{eq:quadra on risk}, the \eqref{eq:quadra approx} holds.  

Moreover, for any $\beta\in\{\beta:|\beta|=r_n\}$, combining \eqref{eq:quadra approx} and the fact $r_n\xi_{K,n}=o(1)$, we have $\{\beta:|\beta|=r_n\}\subset \Theta^\Delta$ for sufficiently large $n$ and
    \begin{equation}
        \liminf_n\inf_{ |\beta|= r_n }  \mb P(\beta)\geq \frac{1}{2}p(0)r_n^2.\label{eq:lower bound}    
    \end{equation}

By Theorem \ref{thm:maximal inequality},
\begin{align}
    \inf_{|\beta|=r_n} \mb P_n(\beta)&\geq \inf_{|\beta|=r_n}\mb P(\beta)-\sup_{|\beta|\leq r_n}|\mb P_n(\beta)-\mb P(\beta)|,\label{eq:inf lower bound}\\
    &\geq \frac{p(0)+o(1)}{2}r_n^2 + O_p\left(r_n(\xi_{K,n}\vee K^{1/2})n^{-1/2}\log ^{1/2}n\right).
\end{align} 
Note that $r_n(\xi_{K,n}\vee K^{1/2})n^{-1/2}\log ^{1/2}n=o(r_n^2)$, we have
\begin{equation}
\lim_{n\rightarrow\infty}\mr P\left\{\inf_{|\beta|= r_n} \mb P_n(\beta)\geq  \frac{1}{3}p(0)r_n^2\right\}=1\label{eq:positive lower bound p1}    
\end{equation}
By the convexity of $\mb P_n(\beta)$, the event
\begin{align}
    \left\{\inf_{|\beta|\geq r_n} \mb P_n(\beta)\geq  \frac{1}{3}p(0)r_n^2\right\}=\left\{\inf_{|\beta|= r_n} \mb P_n(\beta)\geq  \frac{1}{3}p(0)r_n^2\right\}.
\end{align} 
Then for any $\varepsilon>0$, note that $\mb P_n(\hat{\beta})\leq \mb P_n(0)=0$, we have
\begin{align}
0\leq \mr P\left\{ |\hat \beta|\geq \varepsilon r_n \right\}&= \mr P\left\{\mb P_n(\hat{\beta})\geq \frac{1}{3}p(0)\varepsilon r_n^2,|\hat{\beta}|\geq \varepsilon r_n\right\}+\mr P\left\{\mb P_n(\hat{\beta})< \frac{1}{3}p(0)\varepsilon r_n^2,|\hat{\beta}|\geq \varepsilon r_n\right\}\notag\\
&\leq \mr P\left\{\mb P_n(\hat{\beta})> 0\right\}+\mr P\left\{\inf_{|\beta|\geq \varepsilon r_n}\mb P_n(\beta)< \frac{1}{3}p(0)\varepsilon r_n^2\right\}.\label{eq:consistency op}
\end{align}
Let $n\rightarrow\infty$, \eqref{eq:positive lower bound p1} and \eqref{eq:consistency op} imply that $\hat{\beta}=o_p(r_n)$.

\end{proof}

\begin{lemma}
    \label{lem:deterministic expansion}
    If Assumptions \hyperref[(A1)]{(A1)}-\hyperref[(A3)]{(A3)} and \hyperref[(B3)]{(B3)} hold, then
    \begin{equation}
        \tilde{\mb R}(\theta,h)=h^\top \bar{Q}_n \theta + \frac{1}{2}h^\top \bar{Q}_n h+ r_P(\theta,h),\label{eq:lem deterministic expansion}    
    \end{equation}
    where $$
    r_P(\theta,h)=O(t(\theta)|\theta| |h|+t(h)|h|^2+t(\theta) |h|^2 ),
    $$ 
    and $t(\theta)=\sup_{\mf x\in \mathcal{D}_n}|\theta^\top \mf b_\omega(\mf x)|$, $t(h)=\sup_{\mf x\in \mathcal{D}_n}|h^\top \mf b_\omega(\mf x)|$.
\end{lemma}

\begin{proof}
    Denote $\nu_i=-\theta^\top \mf b_\omega(\mf X_i)$ and $u_i=-h^\top \mf b_{\omega}(\mf X_i)$, then
    \begin{align}
        \tilde{\mb R}(\theta,h)&=\mb E\left[ u_i\int_0^1 \overline{\Xi}_n(\nu_i+tu_i|\mf X_i)\mr dt \right],\notag\\
        &=\mb E\left[ u_i \overline{\Xi}_n(\nu_i|\mf X_i)\mr dt \right]+\frac{1}{2}\mb E\left[ u_i^2 \overline{\Xi}_n^{(1)}(\nu_i|\mf X_i)\mr dt \right]\notag\\
        &
        +\mb E\left[u_i\int_0^1 \left(  \overline{\Xi}_n(\nu_i+tu_i|\mf X_i)- \overline{\Xi}_n(\nu_i|\mf X_i)-\overline{\Xi}_n^{(1)}(\nu_i|\mf X_i)tu_i\right)\mr d t\right]\notag\\
        &=T_1(\theta,h)+T_2(\theta,h)+T_3(\theta,h),\label{eq:decomp deterministic expansion}
    \end{align}
where
\begin{align}
    T_1(\theta,h)&=:\mb E\left[ u_i \overline{\Xi}_n(\nu_i|\mf X_i) \right],\notag\\
    T_2(\theta,h)&=:\frac{1}{2}\mb E\left[ u_i^2 \overline{\Xi}_n^{(1)}(\nu_i|\mf X_i) \right],\notag\\
    T_3(\theta,h)&=:\mb E\left[u_i\int_0^1 \left(  \overline{\Xi}_n(\nu_i+tu_i|\mf X_i)- \overline{\Xi}_n(\nu_i|\mf X_i)-t\overline{\Xi}_n^{(1)}(\nu_i|\mf X_i)u_i\right)\mr d t\right].
\end{align}
For $T_1(\theta,h)$, by Assumption \hyperref[(B3)]{(B3)},
$$
    \mb E\left[ u_i (\overline{\Xi}_n(\nu_i|\mf X_i)-\overline \Xi(0|\mf X_i)-\overline \Xi^{(1)}(0|\mf X_i)\nu_i) \right]+O(|\theta||h|\sup_{\mf x}|\theta^\top \mf b_\omega(\mf x)|),
$$
therefore,
\begin{equation}
 T_1(\theta,h)=h^\top \bar Q_n\theta +  O\left(t(\theta)|\theta||h| \right).\label{eq:bound T1}
\end{equation}
For $T_2(\theta,h)$, 
\begin{align}
&\mb E\left[ u_i^2 \overline{\Xi}_n^{(1)}(\nu_i|\mf X_i)\right]-h^\top \bar Q_0h,\notag\\
=&\mb E\left[ u_i^2\big(\overline{\Xi}_n^{(1)}(\nu_i|\mf X_i)-\overline{\Xi}_n^{(1)}(0|\mf X_i)\big) \right]=O(\omega(t(\theta))h^2).\label{eq:bound T2}
\end{align}
By Assumptions \hyperref[(A2)]{(A2)} and \hyperref[(B3)]{(B3)}
\begin{equation}
    T_3(\theta,h)\lesssim \mb E(u_i^3)=O(t(h)|h|^2).\label{eq:bound T3}
\end{equation}
Combining \eqref{eq:bound T1}, \eqref{eq:bound T2}, \eqref{eq:bound T3}, and \eqref{eq:decomp deterministic expansion}, \eqref{eq:lem deterministic expansion} holds.
\end{proof}

\subsection{Gaussian approximation}
\label{sec:gaussian approximation}
\begin{proof}[Proof of Theorem \ref{thm:Gaussian approximation}.]
Recall $\mf Z_n=\sum_{i=1}^n \mf z_i/\sqrt{n}$ in \eqref{eq:Z_n}, define the truncated $\mf z_i$ as
\begin{equation}
    \bar{\mf z}_i=\begin{cases}
        \mf z_i,&\quad |\mf z_i|\leq \pi_n\\
        0,&\quad \text{otherwise}.
    \end{cases}\label{eq:truncation z}
\end{equation}
Denote $\bar{\mf Z}_n=(\bar{\mf Z}_{n,1},\dots,\bar{\mf Z}_{n,K})^\top=:\sum_{i=1}^n  \bar{\mf z}_i/\sqrt{n}$ and $\bar{\mf Z}_n^*=:\bar{\mf Z}_n-\mb E\bar{\mf Z}_n$. Recall $\Upsilon_i^{(m)}=(\gamma_{i-m+1},\dots,\gamma_i)$ where $\gamma_i=(\zeta_i,\eta_i)$, suppose $n/m\in\mb Z$ w.l.o.g. and define
\begin{equation}
    \bar{\mf z}_i^{(m)}=:\mb E(\bar{\mf z}_i| \mathcal{F}_m(i) ),
\end{equation}
where $\mathcal{F}_m(i)=\sigma(\gamma_{i-m+1},\dots,\gamma_i)$ is the sigma-field of $\Upsilon_i^{(m)}$.
Then $\bar{\mf z}_i^{(m)},\bar{\mf z}_j^{(m)}$ are independent as long as $|i-j|>m$. Further let $\bar{\mf Z}_n^{(m)}=:\sum_{i=1}^n  \bar{\mf z}_i^{(m)}/\sqrt{n}$ and $\tilde{\mf Z}_n^{(m)}=:\bar{\mf Z}_n^{(m)}-\mb E\bar{\mf Z}_n^{(m)}$, then $\bar{\mf Z}_n^*-\tilde{\mf Z}_n^{(m)}=\bar{\mf Z}_n-\bar{\mf Z}_n^{(m)}$.

Denote the covariance matrix of $\bar{\mf{Z}}_n^{(m)}$
$$
\begin{aligned}
&\Sigma_{\mf Z}^{(m)}=:\frac{1}{n}\mb E\left\{\left[\sum_{i=1}^n\left(\overline{\mf z}_{i}^{(m)}-\mb E \overline{\mf{z}}_{i}\right)\right]\left[\sum_{i=1}^n\left(\overline{\boldsymbol{z}}_{ i}^{(m)}-\mb E \overline{\mf{z}}_{ i}\right)^{\top}\right]\right\}.
\end{aligned}
$$
Introduce a $K$-dimensional standard Gaussian random vector $\mf G$ and denote 
$$
\mf{G}_n=\Sigma_{\mf Z}^{1/2}\mf G,\quad \tilde{\mf G}_n^{(m)}=\left(\Sigma_{\mf Z}^{(m)}\right)^{1/2}\mf G
$$
so that $\mf{G}_n$ and $\tilde{\mf G}_n^{(m)}$ preserve the covariance structure $\Sigma_{\mf Z},\Sigma_{\mf Z}^{(m)}$, respectively. 

Introduce a smoothed function
$$
h_{A,\epsilon_1}(\omega)=:h\left(\frac{\inf_{\nu\in A}|\omega-\nu| }{\epsilon_1}\right),
$$
where $\omega\in\mb R^K$, set $A\subset \mb R^K$, and
$$
h(x)= \begin{cases}1, & x<0, \\ 1-2 x^2, & 0 \leqslant x<\frac{1}{2}, \\ 2(1-x)^2, & \frac{1}{2} \leqslant x<1, \\ 0, & x \geqslant 1 .\end{cases}
$$

By Lemma \ref{lem:liu2025}, for any $\epsilon>0$, $\mathcal{K}(\mf Z_n,\mf G_n)$ can be bounded as
\begin{equation}
    \mathcal{K}(\mf Z_n,\mf G_n)\lesssim K^{1/4}\epsilon+\min\left\{\sup _{A \in \mathcal{A}}\left|\mb E\left(h_{A, \epsilon}\left(\mf Z_n\right)-h_{A, \epsilon}\left(\mf G_n\right)\right)\right| + \mr P\left( |\mf Z_n-\mf G_n|\geq \epsilon \right)   \right\}.\label{eq:liu2025 lemma decomp}
\end{equation}
In the following arguments, we will take two approaches to bound $\mathcal{K}(\mf Z_n,\mf G_n)$ based on the two quantities in the minimum part of \eqref{eq:liu2025 lemma decomp}.

\subsubsection{First approach}
The first approach starts from the first result in the minimum part of Lemma \eqref{eq:liu2025 lemma decomp}. Notice that $|\nabla h_{A,\epsilon}|\leq 2\epsilon^{-1}$,
then for any $\epsilon_1>0$, we can decompose the $\mathcal{K}(\mf Z_n,\mf G_n)$ as
\begin{align}
    \mathcal{K}(\mf Z_n,\mf G_n) & \lesssim K^{1/4}\epsilon_1+\sup _{A \in \mathcal{A}}\left|\mb E\left[h_{A, \epsilon_1}\left(\mf Z_n\right)-h_{A, \epsilon_1}\left(\mf G_n\right)\right]\right|,\notag\\
    &\lesssim K^{\frac{1}{4}} \epsilon_1+\sup _{A \in \mathcal{A}}\left|\mb E\left[h_{A, \epsilon_1}\left(\mf Z_n\right)-h_{A, \epsilon_1}\left(\overline{\mf Z}_n^*\right)\right]\right|+\sup_{A \in \mathcal{A}}\left|\mb E\left[h_{A, \epsilon_1}\left(\overline{\mf Z}_n^*\right)-h_{A, \epsilon_1}\left(\tilde{\mf Z}_n^{(m)}\right)\right]\right|\notag\\
    &+\sup_{A \in \mathcal{A}}\left|\mb E\left[h_{A, \epsilon_1}\left(\tilde{\mf G}_n^{(m)}\right)-h_{A, \epsilon_1}\left(\mf G_n\right)\right]\right|+\sup_{A \in \mathcal{A}}\left|\mb E\left[h_{A, \epsilon_1}\left(\tilde{\mf Z}_n^{(m)}\right)-h_{A, \epsilon_1}\left(\tilde{\mf G}_n^{(m)}\right)\right]\right|,\notag\\
    &\lesssim K^{\frac{1}{4}} \epsilon_1+\frac{1}{\epsilon_1} \mb E\left|\mf Z_n-\overline{\mf Z}_n^*\right|+\frac{1}{\epsilon_1} \mb E\left|\overline{\mf Z}_n^*-\tilde{\mf Z}_n^{(m)}\right|+\frac{1}{\epsilon_1} \mb E\left|\tilde{\mf G}_n^{(m)}-\mf{G}_n\right|\notag\\
    &+\sup_{A \in \mathcal{A}}\left|\mb E\left[h_{A, \epsilon_1}\left(\tilde{\mf Z}_n^{(m)}\right)-h_{A, \epsilon_1}\left(\tilde{\mf G}_n^{(m)}\right)\right]\right|.\label{eq:decomp of Kolmogorov dist}
\end{align}

Based on decomposition \eqref{eq:decomp of Kolmogorov dist}, setting appropriate $m\asymp \log n$, we shall prove the following assertions as follows:
\begin{description} 
    \item[(1)] Truncation error
    \begin{equation}
        \mb E\left|\mf Z_n-\overline{\mf Z}_n^*\right|=O(\sqrt{n}\pi_n^{1-q} \xi_{K,n}^q).\label{eq:truncation error}
    \end{equation}
    \item[(2)] M-decomposition error
    \begin{equation}
        \mb E\left|\overline{\mf Z}_n^*-\tilde{\mf Z}_n^{(m)}\right|=o(\sqrt{n}\pi_n^{1-q} \xi_{K,n}^q).  \label{eq:m-decomposition error GA}
    \end{equation}
    \item[(3)] Gaussian comparison
    \begin{equation}
        \mb E\left|\tilde{\mf G}_n^{(m)}-\mf{G}_n\right| =O\left(n\pi_n^{2-2q}\xi_{K,n}^{2q}\right).\label{eq:Gaussian comparison}
    \end{equation}
    \item[(4)] Gaussian approximation
    \begin{equation}
        \sup_{A \in \mathcal{A}}\left|\mb E\left[h_{A, \epsilon_1}\left(\tilde{\mf Z}_n^{(m)}\right)-h_{A, \epsilon_1}\left(\tilde{\mf G}_n^{(m)}\right)\right]\right|=O\left( K^{1/4}n^{-1/2}\pi_n^3\log^2n+\frac{1}{\epsilon_1}K^{1/4}n^{-1}\pi_n^6\log^4 n \right). \label{eq:bounded Gaussian approx}
    \end{equation}
\end{description}
Combining \eqref{eq:truncation error}, \eqref{eq:m-decomposition error GA}, \eqref{eq:Gaussian comparison}, \eqref{eq:bounded Gaussian approx}, and \eqref{eq:decomp of Kolmogorov dist},
\begin{align*}
    \mathcal{K}(\mf Z_n,\mf G_n)&=O\left( K^{\frac{1}{4}}\epsilon_1+\frac{1}{\epsilon_1}\left( n^{\frac{1}{2}}\pi_n^{1-q} \xi_{K,n}^{q}+n\pi_n^{2-2q} \xi_{K,n}^{2q}+K^{\frac{1}{4}}n^{-1}\pi_n^6\log^4 n\right)+ K^{\frac{1}{4}}n^{-\frac{1}{2}}\pi_n^{3}\log^2 n \right).
\end{align*}
Therefore, with setting $\pi_n=K^{-\frac{1}{2(q+5)}}n^{\frac{3}{2(q+5)}}\xi_{K,n}^{\frac{q}{q+5}}\log^{-\frac{4}{q+5}}n$, using the arbitrariness of $\epsilon_1>0$, we can have
\begin{equation}
     \mathcal{K}(\mf Z_n,\mf G_n)=O\left(K^{\frac{q-1}{4(q+5)}}\xi_{K,n}^{\frac{3q}{q+5}}n^{\frac{4-q}{2(q+5)}}\log^{2-\frac{12}{q+5}}n\right).\label{eq:1st approach rate}
\end{equation}


In the following we show the \eqref{eq:truncation error}, \eqref{eq:m-decomposition error GA}, \eqref{eq:Gaussian comparison}, and \eqref{eq:bounded Gaussian approx} hold.
\paragraph*{Truncation error}
In view of $\mb E\mf Z_n=0$ and $\mf z_i-\bar{\mf z}_i=\mf z_i\mf 1_{|\mf z_i|>\pi_n}$, we have for any $q> 1$,
\begin{align}
    \mb E|\mf Z_n-\bar{\mf Z}^*_n|&\leq\mb E|\mf Z_n-\bar{\mf Z}_n|+|\mb E \mf Z_n-\mb E\bar{\mf Z}_n|\notag\\
    &\leq \frac{2}{\sqrt{n}} \mb E\left|\sum_{i=1}^n\mf z_i\mf 1_{|\mf z_i|>\pi_n}\right|\notag\\
    &\leq 2\sqrt{n}\mb E \left[ |\mf z_i|\left(\frac{|\mf z_i|}{\pi_n}\right)^{q-1} \right]\notag\\
    &=2\sqrt{n}\pi_n^{1-q}\mr \|\mf z_i\|_q^q. \label{eq:error bound for trucation}
\end{align}
Then by condition $\|\mf z_i\|_q = O(\xi_{K,n}),q>4$ in Theorem \ref{thm:Gaussian approximation}, we show that \eqref{eq:truncation error} holds. 
Moreover, for the $q>4$ in condition $\|\mf z_i\|_q = O(\xi_{K,n})$, we can also have for $s=q/2$,
\begin{align}
    \|\mf Z_n-\bar{\mf Z}^*_n\|&\leq 2\sqrt{n}\left\{\mb E \left[ |\mf z_i|^2\left(\frac{|\mf z_i|}{\pi_n}\right)^{2s-2} \right]\right\}^{1/2}\notag\\
    &=O(\sqrt{n}\pi_n^{1-s}\mr \|\mf z_i\|_{2s}^{s})=O(T_{n,q}), \label{eq:error bound for trucation-q2}
\end{align}
where $T_{n,q}=:\sqrt{n}\pi_n^{1-q/2}\xi_{K,n}^{q/2}$.
\paragraph*{M-decomposition error}
Denote operator  $\mathcal{P}^{(k)}\bar{\mf z}_i=:\mb E(\bar{\mf z}_i|\mathcal{F}_{k}(i))-\mb E(\bar{\mf z}_i|\mathcal{F}_{k-1}(i) )$ 
using the fact $\bar{\mf z}_i=\lim_{j\rightarrow \infty}\mb E(\bar{\mf z}_i|\mathcal{F}_{i+j}(i))$, we have
\begin{equation}
    \bar{\mf Z}_n-\bar{\mf Z}_n^{(m)}=\frac{1}{\sqrt{n}}\sum_{i=1}^n \sum_{j=m-i+1}^\infty \mathcal{P}^{(i+j)}\bar{\mf z}_i=\frac{1}{\sqrt{n}}\sum_{j=m-n+1}^\infty \mf R_{n,j},\label{eq:decomp for m-dependent-GA}
\end{equation}
where $\mf R_{n,j}=:\sum_{i=(m-j+1)\vee1}^n \mathcal{P}^{(i+j)}\bar{\mf z}_i$. Recall $\Upsilon_i$ in Section \ref{sec:proof of thm:maximal inequality}, by Jensen's inequality, for $q> 1$,
\begin{align}
    \left\|\mathcal{P}^{(k)}\bar{\mf z}_i\right\|_q &\leq \left\|\mathcal{P}^{(k)}\mf z_i\right\|_q=\left\{\mb E\left| \mb E(\mf z_i|\gamma_{i-k+1},\dots,\gamma_i)- \mb E(\mf z_i|\gamma_{i-k+2},\dots,\gamma_i)\right|^q\right\}^{1/q}\notag\\
    &= \left\{\mb E\left| \mb E\left[ \mb E\left(\mf z_i| \Upsilon_{i,k-1}\right)- \mb E\left(\mf z_i| \Upsilon_i\right) \Big| \Upsilon_{i-k+1}\right]\right|^q\right\}^{1/q},\notag\\
    &\leq \left\| \psi(\varepsilon_{i,k-1})\mf b_\omega(\mf X_{i,k-1})- \psi(\varepsilon_{i})\mf b_\omega(\mf X_{i}) \right\|_q=\delta_{\mf z}(k-1,q).\label{eq:bound for operator P(k)}
\end{align}
Note that process $\{\mf R_{n,j},j\geq m-n+1\}$ is martingale difference with respect to filtration $\sigma(\gamma_{-j+1},\gamma_{-j+2},\dots)$. If $q\geq 2$, by Burkholder's inequality, there exists constant $C_q>0$ such that
\begin{align}
    \left\| \sum_{j=m-n+1}^\infty \mf R_{n,j}\right\|_q^2&\leq C_q \sum_{j=m-n+1}^\infty \|\mf R_{n,j}\|_q^2 \notag\\
    &\leq C_q\sum_{j=m-n+1}^\infty \left(\sum_{i=(m-j+1)\vee1}^n \left\|\mathcal{P}^{(i+j)}\bar{\mf z}_i\right\|_q   \right)^2.\label{eq:Burkholder for m-dependent-GA}
\end{align}
By Lemma \ref{lem:phd on z} and \eqref{eq:bound for operator P(k)}, there exists constant $\chi\in(0,1)$, for any given $\alpha\in(0,1)$,
\begin{align}
 \left\|\mathcal{P}^{(k)}\bar{\mf z}_i\right\|=O(\xi_{K,n}\Delta_{K,n}^\alpha\chi^{\alpha k}).
 \label{eq:bound for operator P(k)}
\end{align}
Combining \eqref{eq:decomp for m-dependent-GA}, \eqref{eq:bound for operator P(k)}, and \eqref{eq:Burkholder for m-dependent-GA}, elementary calculation yields
\begin{equation}
\|\bar{\mf Z}_n^*-\tilde{\mf Z}_n^{(m)}\|=\| \bar{\mf Z}_n-\bar{\mf Z}_n^{(m)} \|=\frac{1}{\sqrt
n}\left\| \sum_{j=m-n+1}^\infty \mf R_{n,j}\right\|=O(\xi_{K,n}\Delta_{K,n}^\alpha\chi^{\alpha m} ).\label{eq:error bound for m-dependent-GA}
\end{equation}
Setting appropriate m-decomposition $m\asymp \log n$ (e.g., $m= \frac{(q-4)|\omega_1+\omega_0/12-1/6|\log n}{\alpha\log(1/\chi)}$ ), we have for $q>4$,
\begin{equation}
    \xi_{K,n}\Delta_{K,n}^\alpha\chi^{\alpha m} \ll T_{n,q} \label{eq:m-decomp design}
\end{equation} 
with $\alpha<\min\{1,\frac{q-4}{2\omega_1^\prime}|\omega_1+\omega_0/12-1/6|\}$.




\paragraph*{Gaussian comparison} 
For a matrix $\mf A$, denote $\|\mf A\|_F$ as the Frobenius norm of $\mf A$ i.e. $\|\mf A\|_F=\left(\operatorname{tr}(\mf A^\top \mf A)\right)^{1/2}$. Recall $\Sigma_{\mf Z}=\mb E\left(\sum_{i=1}^n \mf{z}_{i}\right)\left(\sum_{i=1}^n \mf{z}_{i}^{\top}\right) / n$ and denote the covariance matrix of $\bar{\mf{Z}}_n^{(m)}$
$$
\begin{aligned}
&\Sigma_{\mf Z}^{(m)}=:\frac{1}{n}\mb E\left\{\left[\sum_{i=1}^n\left(\overline{\mf z}_{i}^{(m)}-\mb E \overline{\mf{z}}_{i}\right)\right]\left[\sum_{i=1}^n\left(\overline{\boldsymbol{z}}_{ i}^{(m)}-\mb E \overline{\mf{z}}_{ i}\right)^{\top}\right]\right\}.
\end{aligned}
$$

Introduce $K$-dimensional Gaussian random vectors
$$
\mf{G}_n=\Sigma_{\mf Z}^{1/2}\mf G,\quad\tilde{\mf G}_n^{(m)}=\left(\Sigma_{\mf Z}^{(m)}\right)^{1/2}\mf G,
$$ 
where $\mf G$ is a $K$-dimensional standard Gaussian random vector so that $\mf{G}_n$ and $\tilde{\mf G}_n^{(m)}$ preserve the same covariance structure of $\mf Z_n$ and $\bar{\mf Z}_n^{(m)}$, respectively.

Consider the difference of covariance matrix between $\bar{\mf{Z}}_n^{(m)}$ and $\mf{Z}_n$ based on Frobenius norm, using the fact $\mb E\mf Z_n=0$ and $\mb E\bar{\mf Z}_n^{(m)}=\mb E \bar{\mf Z}_n$,
\begin{align}
\left\|\Sigma_{\mf Z}-\Sigma_{\mf Z}^{(m)}\right\|_F&\leq \left\|\mb E\left(\mf Z_n-\bar{\mf Z}_n^{(m)}\right)\left(\mf Z_n\right)^\top\right\|_F+\left\|\mb E\bar{\mf Z}_n^{(m)}\left(\mf Z_n-\bar{\mf Z}_n^{(m)}\right)^{\top}\right\|_F\notag\\
&+\left\| \mb E\left(\mf Z_n-\bar{\mf Z}_n\right)  \mb E\left(\mf Z_n-\bar{\mf Z}_n\right)^{\top}\right\|_F.    
\end{align}
Denote operator $\mathcal{P}_k\cdot=:\mb E(\cdot|\Upsilon_k)-\mb E(\cdot|\Upsilon_{k-1})$. By Burkholder's inequality and H\"older's inequality, for any given $\alpha\in(0,1)$,
\begin{align}
    \left\| \mf Z_n \right\|&=\frac{1}{\sqrt{n}}\left\| \sum_{i=1}^n\sum_{k=0}^\infty \mathcal{P}_{i-k}\psi(\varepsilon_i+a_n(\mf X_i))\mf b_\omega(\mf X_i) \right\|\leq\frac{1}{\sqrt{n}} \sum_{k=0}^\infty\left\| \sum_{i=1}^n \mathcal{P}_{i-k}\psi(\varepsilon_i+a_n(\mf X_i))\mf b_\omega(\mf X_i) \right\|,\notag\\
    &\lesssim \frac{1}{\sqrt{n}} \sum_{k=0}^\infty \sqrt{\sum_{i=1}^n\left\| \mathcal{P}_{i-k}\psi(\varepsilon_i+a_n(\mf X_i))\mf b_\omega(\mf X_i) \right\|^2},\notag\\
    &=O\left(\sum_{k=0}^\infty \delta_{\mf z}(k,2)\right)=O( \Delta_{K,n}^\alpha\xi_{K,n} ),\label{eq:bound for Z_n}
\end{align}
where the last line is obtained by Assumption \hyperref[(B2)]{(B2)}, \hyperref[(A1)]{(A1)} and Lemma \ref{lem:phd on bx}.

By \eqref{eq:error bound for trucation-q2}, \eqref{eq:bound for Z_n}, and \eqref{eq:error bound for m-dependent-GA},
\begin{align}
    \left\|\mb E\left(\mf Z_n-\bar{\mf Z}_n^{(m)}\right)\left(\mf Z_n\right)^\top\right\|_F\leq \|\mf Z_n-\bar{\mf Z}_n^{(m)}\|\cdot \|\mf Z_n\|=O\left(T_{n,q}\xi_{K,n}\Delta_{K,n}^\alpha\chi^{\alpha m}  \right),\label{eq: cov comparison-1}
\end{align}
and similarly, we also have
\begin{equation}
    \left\|\mb E\bar{\mf Z}_n^{(m)}\left(\mf Z_n-\bar{\mf Z}_n^{(m)}\right)^{\top}\right\|_F\leq \|\mf Z_n-\bar{\mf Z}_n^{(m)}\|\cdot \|\bar{\mf Z}_n\|=O\left( T_{n,q}\xi_{K,n}\Delta_{K,n}\chi^{\alpha m}  \right).\label{eq: cov comparison-2}
\end{equation}
Besides, by \eqref{eq:error bound for trucation},
\begin{equation}
    \left\| \mb E\left(\mf Z_n-\bar{\mf Z}_n\right)  \mb E\left(\mf Z_n-\bar{\mf Z}_n\right)^{\top}\right\|_F \leq   \left|\mb E(\mf Z_n-\bar{\mf Z}_n)\right|^2=O\left(T_{n,q}^2\right).\label{eq: cov comparison-3}
\end{equation}
Combining \eqref{eq: cov comparison-1}, \eqref{eq: cov comparison-2}, \eqref{eq: cov comparison-3}, 
\begin{equation}
    \left\|\Sigma_{\mf Z}-\Sigma_{\mf Z}^{(m)}\right\|_F=O\left(T_{n,q}^2+T_{n,q}\xi_{K,n}\Delta_{K,n}^\alpha\chi^{\alpha m}\right).\label{eq:bound for cov comparison}
\end{equation}
Since we assume the smallest eigenvalue of $\Sigma_{\mf Z}$ is bounded away from $0$, as $n$ large enough, we have
\begin{equation}
\lambda_{min}(\Sigma_{\mf Z}^{(m)})\geq \lambda_{min}(\Sigma_{\mf Z}) + \lambda_{min}(\Sigma_{\mf Z}-\Sigma_{\mf Z}^{(m)}) \gtrsim c-\left\|\Sigma_{\mf Z}-\Sigma_{\mf Z}^{(m)}\right\|_F>0.\label{eq: lower bound for cov-m}    
\end{equation}
Combining \eqref{eq:bound for cov comparison}, \eqref{eq: lower bound for cov-m} and sub-multiplicativity of Frobenius norm, we have
\begin{align}
    \mb E\left|\tilde{\mf G}_n^{(m)}-\mf{G}_n\right| &=\mb E\left| \left(\Sigma_{\mf Z}^{1/2}-\left(\Sigma_{\mf Z}^{(m)}\right)^{1/2}\right) \mf{G}\right|\notag\\
    &\leq \left\|\Sigma_{\mf Z}^{1/2}-\left(\Sigma_{\mf Z}^{(m)}\right)^{1/2}\right\|_F=O\left(T_{n,q}^2+T_{n,q}\xi_{K,n}\Delta_{K,n}^\alpha\chi^{\alpha m}\right). \label{eq:Gaussian vector comparison}
\end{align}
By similar arguments in \eqref{eq:m-decomp design}, $\xi_{K,n}\Delta_{K,n}^\alpha\chi^{\alpha m}=o(T_{n,q})$ by appropriate $m$ and $\alpha$, which yields \eqref{eq:Gaussian comparison}.

\paragraph*{Gaussian approximation}
Plug $n_1=m$, $n_2=2m$, $n_3=3m$, $\kappa=n^{-1/2}\pi_n$ into Lemma \ref{lem:Fang2016} with \eqref{eq: lower bound for cov-m}, we have
\begin{equation}
    \mathcal{K}\left(\tilde{\mf Z}_n^{(m)},\tilde{\mf G}_n^{(m)}\right)=O\left(K^{1/4}n^{-1/2}\pi_n^3\log^2 n\right). \label{eq: application of Fang2016 }
\end{equation}
Moreover, in the proof of Lemma \ref{lem:Fang2016}, equation (4.23) in \cite{Fang2016CLT} yields
$$
\sup_{A \in \mathcal{A}}\left|\mr{P}\left[h_{A, \epsilon_1}\left(\tilde{\mf Z}_n^{(m)}\right)-h_{A, \epsilon_1}\left(\tilde{\mf G}_n^{(m)}\right)\right]\right|\leq C n\kappa^3n_1n_2\epsilon_1^{-1}\left[K^{1/4}(\epsilon_1+n_3\beta)+\mathcal{K}\left(\tilde{\mf Z}_n^{(m)},\tilde{\mf G}_n^{(m)}\right)\right],
$$
by \eqref{eq: application of Fang2016 },
\begin{equation}
    \sup_{A \in \mathcal{A}}\left|\mr{P}\left[h_{A, \epsilon_1}\left(\tilde{\mf Z}_n^{(m)}\right)-h_{A, \epsilon_1}\left(\tilde{\mf G}_n^{(m)}\right)\right]\right|=O\left( K^{1/4}n^{-1/2}\pi_n^3\log^2n+\frac{1}{\epsilon_1}K^{1/4}n^{-1}\pi_n^6\log^4 n \right).\label{eq:mid-proof of Fang 2016}
\end{equation}

\subsubsection{Second approach}
The second approach is to bound the second result in the minimum part of \eqref{eq:liu2025 lemma decomp}. For any $\epsilon_2>0$, we can decompose the $\mathcal{K}(\mf Z_n,\mf G_n)$ as
\begin{align}
    \mathcal{K}(\mf Z_n,\mf G_n)&\lesssim K^{1/4}\epsilon_2+\mr P(|\mf Z_n-\mf G_n|\geq \epsilon_2),\notag\\
    &\lesssim K^{1/4}\epsilon_2+\mr P(|\mf Z_n-\mf G_n^{(m)}|\geq \epsilon_2/2)+\mr P(|\mf G_n-\mf G_n^{(m)}|\geq \epsilon_2/2),\notag\\
    &\lesssim K^{1/4}\epsilon_2+\frac{1}{\epsilon_2^2}\|\mf Z_n-\mf G_n^{(m)}\|^2+\frac{1}{\epsilon_2^2}\|\mf G_n-\mf G_n^{(m)}\|^2,\label{eq:2nd approach GA}
\end{align}
where the last line in \eqref{eq:2nd approach GA} is obtained by Chebyshev's inequality. By 
triangular inequality, we have
\begin{equation}
    \left\|\mf Z_n-\mf G_n^{(m)}\right\|\leq \left\|\mf Z_n-\tilde{\mf Z}_n^{(m)}\right\|+\left\|\tilde{\mf Z}_n^{(m)}-\mf G_n^{(m)}\right\|.\label{eq:2nd GA decomp}
\end{equation}
$\left\|\mf Z_n-\tilde{\mf Z}_n^{(m)}\right\|$ has been bounded by \eqref{eq:truncation error} and \eqref{eq:m-decomposition error GA} so it remains to bound $\left\|\tilde{\mf Z}_n^{(m)}-\mf G_n^{(m)}\right\|$.

Recall $\tilde{\mf Z}_n^{(m)}=\sum_{i=1}^n\tilde{\mf z}_i^{(m)}/\sqrt{n}$ and suppose $n/m=L\in\mb Z$ without loss of generality. Let $s=m+1$ and define the residue classes
\[
\mathcal I_r=:\{\,r+ks:\ k\ge 0,\ r+ks\le n\,\},\qquad r=1,\dots,s,
\]
with $n_r=|\mathcal I_r|$. For each $r$, the collection
$\{\tilde{\mf z}_i^{(m)}: i\in \mathcal I_r\}$ is independent because
$|i-j|\ge s=m+1$ for distinct $i,j\in\mathcal I_r$.
Denote 
\begin{equation}
    \tilde{\mf Z}_{n,r}^{(m)}=:\frac{1}{\sqrt n}\sum_{i\in \mathcal I_r}\tilde{\mf z}_i^{(m)},
    \label{eq:subsequence split}
\end{equation}
then $\tilde{\mf Z}_n^{(m)}=\sum_{r=1}^{s}\tilde{\mf Z}_{n,r}^{(m)}$.
Fix $r\in\{1,\dots,s\}$ and enumerate $\mathcal I_r=\{i_{r,1}<\cdots<i_{r,n_r}\}$.
Let $\mf x_{r,t}^*=:\tilde{\mf z}_{i_{r,t}}^{(m)}$ for $t=1,\dots,n_r$, then $\{\mf x_{r,t}^*\}_{t=1}^{n_r}$ are independent, centered, and satisfy
$|\mf x_{r,t}^*| \le 2\pi_n$. Apply Lemma \ref{lem:strong GA in mies}, on a rich enough probability space, there exist random vectors $\{\tilde{\mf x}_{r,t}\}_{t=1}^{n_r}$ with
$\tilde{\mf x}_{r,t}\stackrel{d}{=}\mf x_{r,t}^*$ and independent Gaussian vectors
$\{\mf g_{r,t}\}_{t=1}^{n_r}$ satisfying $\mf g_{r,t}\sim N(\mathbf 0,\operatorname{Cov}(\mf x_{r,t}^*))$
such that
\begin{equation}
    \left\|
    \frac{1}{\sqrt{n_r}}\sum_{t=1}^{n_r}\big(\tilde{\mf x}_{r,t}-\mf g_{r,t}\big)
    \right\|
    \leq
    C\,\pi_n\,K^{\frac12-\frac1q}\,n_r^{-\frac12+\frac1q}\sqrt{\log n},
    \label{eq:subseq strong approx}
\end{equation}
where $C>0$ is a constant and we used $\log n_r\le \log n$. Define the Gaussian partial sum on the original $\sqrt n$ scale by
\begin{equation}
    \mf G_{n,r}^{(m)}=:\frac{1}{\sqrt n}\sum_{t=1}^{n_r}\mf g_{r,t}.
    \label{eq:Gnr def}
\end{equation}
Then, by \eqref{eq:subseq strong approx},
\begin{align}
    \big\|\tilde{\mf Z}_{n,r}^{(m)}-\mf G_{n,r}^{(m)}\big\|
    &=
    \left\|
    \frac{1}{\sqrt n}\sum_{t=1}^{n_r}\big(\tilde{\mf x}_{r,t}-\mf g_{r,t}\big)
    \right\|
    =
    \sqrt{\frac{n_r}{n}}
    \left\|
    \frac{1}{\sqrt{n_r}}\sum_{t=1}^{n_r}\big(\tilde{\mf x}_{r,t}-\mf g_{r,t}\big)
    \right\| \notag\\
    &\le
    C\,\pi_n\,K^{\frac12-\frac1q}\,n^{-\frac12}\,n_r^{\frac1q}\sqrt{\log n}.
    \label{eq:subseq scaled error}
\end{align}
Note that $\mf G_n^{(m)}=\sum_{r=1}^{s}\mf G_{n,r}^{(m)}$. By the triangle inequality and \eqref{eq:subseq scaled error},
\begin{align}
    \big\|\tilde{\mf Z}_n^{(m)}-\mf G_n^{(m)}\big\|
    &\le
    \sum_{r=1}^{s}\big\|\tilde{\mf Z}_{n,r}^{(m)}-\mf G_{n,r}^{(m)}\big\|
    \le
    C\,\pi_n\,K^{\frac12-\frac1q}\,n^{-\frac12}\sqrt{\log n}\sum_{r=1}^{s}n_r^{\frac1q}.
    \label{eq:aggregate step 1}
\end{align}
Since $\sum_{r=1}^{s} n_r=n$ and $q>2$, Hölder's inequality yields
\begin{equation}
    \sum_{r=1}^{s} n_r^{\frac1q}
    \le
    s^{1-\frac1q}\left(\sum_{r=1}^{s} n_r\right)^{\frac1q}
    =
    s^{1-\frac1q}n^{\frac1q}.
    \label{eq:holder nr}
\end{equation}
Combining \eqref{eq:aggregate step 1}--\eqref{eq:holder nr} and $s=m+1$ gives
\begin{equation}
    \left\|\tilde{\mf Z}_n^{(m)}-\mf G_n^{(m)}\right\|
    \le
    C\,\pi_n\,K^{\frac12-\frac1q}\,(m+1)^{1-\frac1q}\,n^{-\frac12+\frac1q}\sqrt{\log n}
    \lesssim
    \pi_n\,K^{\frac12-\frac1q}\,m^{1-\frac1q}\,n^{-\frac12+\frac1q}\sqrt{\log n}.
    \label{eq:2nd approach main strong approx}
\end{equation}
Plugging \eqref{eq:2nd approach main strong approx} and \eqref{eq:Gaussian vector comparison}, into \eqref{eq:2nd approach GA} and \eqref{eq:2nd GA decomp}, we have
\begin{align}
    \|\mf Z_n-\mf G_n\|
    &\le
    \|\mf Z_n-\tilde{\mf Z}_n^{(m)}\|+\|\tilde{\mf Z}_n^{(m)}-\mf G_n^{(m)}\|
    +\|\mf G_n-\mf G_n^{(m)}\|\notag\\
    &=O\left( n^{-\frac{1}{2}+\frac{3}{q}-\frac{1}{q^2}}\xi_{K,n}K^{\frac{1}{2}+\frac{2}{q^2}-\frac{2}{q}}\log^{2-\frac{4}{q}}n \right),
    \label{eq:2nd approach final}
\end{align}
where the last line is obtained by $\xi_{K,n}\Delta_{K,n}^{\alpha}\chi^{\alpha m}=o(T_{n,q})$ and $T_{n,q}=\sqrt{n}\pi_n^{1-q/2}\xi_{K,n}^{q/2}$ with setting $m\asymp \log n$ and $\pi_n=n^{2/q-1/q^2}\xi_{K,n} K^{2/q^2-1/q}\log^{-4/q}n$. Therefore, by the arbitrariness of $\epsilon_1>0$, the second approach yields
\begin{align}
    \mathcal{K}(\mf Z_n,\mf G_n)& \lesssim     K^{1/4}\epsilon_2
    +\frac{1}{\epsilon_2^2}\|\mf Z_n-\mf G_n\|^2 ,\notag\\
    &\lesssim K^{1/4}\epsilon_2+\frac{1}{\epsilon_2^2}\left( 
    n^{-\frac{1}{2}+\frac{3}{q}-\frac{1}{q^2}}\xi_{K,n}K^{ \frac{1}{2}+\frac{2}{q^2}-\frac{2}{q}}\log^{2-\frac{4}{q}}n \right)^2,\notag\\
    &=O\left(K^{\frac{1}{2}+\frac{4}{3q}(\frac{1}{q}-1)}\xi^{\frac{2}{3}}_{K,n}n^{-\frac{1}{3}+\frac{2}{q}(1-\frac{1}{3q})}\log^{\frac{4}{3}-\frac{8}{3q}} n\right).\label{eq:2nd approach rate}
\end{align}
Combining \eqref{eq:2nd approach rate}, \eqref{eq:1st approach rate}, and \eqref{eq:liu2025 lemma decomp}, we have
$$
\mathcal{K}(\mf Z_n,\mf G_n)=O\left(\min\left\{K^{\frac{q-1}{4(q+5)}}\xi_{K,n}^{\frac{3q}{q+5}}n^{\frac{4-q}{2(q+5)}}\log^{2-\frac{12}{q+5}}n,\quad K^{\frac{1}{2}+\frac{4}{3q}(\frac{1}{q}-1)}\xi^{\frac{2}{3}}_{K,n}n^{-\frac{1}{3}+\frac{2}{q}(1-\frac{1}{3q})}\log^{\frac{4}{3}-\frac{8}{3q}} n  \right\}\right),
$$
which yields \eqref{eq:GA bound Z}. 

Furthermore, by Lemma \ref{lem:basic_smoothing}, we have for any $\epsilon_0>0$,
\begin{align}
   \sup_{A\in\mathcal{A}} \left| \mr P(\sqrt{n}\hat{\beta}\in A| \mathcal{B}_n^\epsilon )-\mr P\left(\bar{Q}_n^{-1}\mf G_n\in A\right)\right| \lesssim     K^{1/4}\epsilon_0
    +\frac{1}{\epsilon_0^2}\mb E\left( |\sqrt{n}\hat{\beta}-\bar{Q}_n^{-1}\mf Z_n|^2 \mid \mathcal{B}_n^{\epsilon} \right)+\mathcal{K}(\mf Z_n,\mf G_n).\label{eq:biginning decomp GA}
\end{align}
Under the event $\mathcal{B}_n^\epsilon$, combining \eqref{eq:biginning decomp GA} and  \eqref{eq:GA bound Z},
\begin{align}
&\sup_{A\in\mathcal{A}} \left| \mr P(\sqrt{n}\hat{\beta}\in A| \mathcal{B}_n^\epsilon )-\mr P\left(\bar{Q}_n^{-1}\mf G_n\in A\right)\right|\notag\\
& \lesssim     K^{1/4}\epsilon_0
    +\frac{1}{\epsilon_0^2}\mb E\left( |\sqrt{n}\hat{\beta}-\bar{Q}_n^{-1}\mf Z_n|^2 \mid \mathcal{B}_n^{\epsilon} \right)+\mathcal{K}(\mf Z_n,\mf G_n),\notag\\
    &\lesssim K^{1/4}\epsilon_0+ \frac{1}{\epsilon_0^2} \left( K^{\frac{\eta+2}{4}}n^{-\frac{\eta}{4}}\log^{\frac{\eta+1}{2}} n \right)^2 \notag\\
    &+\min\left\{K^{\frac{q-1}{4(q+5)}}\xi_{K,n}^{\frac{3q}{q+5}}n^{\frac{4-q}{2(q+5)}}\log^{2-\frac{12}{q+5}}n,\quad K^{\frac{1}{2}+\frac{4}{3q}(\frac{1}{q}-1)}\xi^{\frac{2}{3}}_{K,n}n^{-\frac{1}{3}+\frac{2}{q}(1-\frac{1}{3q})}\log^{\frac{4}{3}-\frac{8}{3q}} n  \right\},
\end{align}
which yields \eqref{eq:GA bound} by the arbitrariness of $\epsilon_0>0$.
\end{proof}

\subsection{Proof of Corollary \ref{cor:SCR Gaussian approx}}
\begin{proof}
Recalling $\mf G_n\sim N_K(0,\Sigma_{\mf Z})$, we can rewrite
\begin{equation}
\frac{1}{2\sqrt{n}p(0)} \mf G_n^\top \mf b_\omega(\mf x)=\mf l(x)^\top \mf G,
\end{equation}
where $\mf G=(G_1,\dots,G_K)^{\top}$ and $\{G_j\}_{j=1}^K$ are i.i.d. standard normal random variables
\begin{equation}
    \mf l(\mf x)=:\frac{1}{2\sqrt{n}p(0)}\Sigma_{\mf Z}^{1/2}\mf b_\omega(\mf x),\notag
\end{equation}
Denote set
\begin{equation}
    A(C) = :\left\{\mf S\in\mb R^K: \left|\mf S^\top \mf b_\omega(\mf x) \right|\leq C|\mf l(\mf x)|,\forall \mf x\in \mathcal{D}_n \right\} .\label{eq:convex set A}
\end{equation}
Note that $A(C)$ is a convex set and $A(C)\in\mathcal{A}$ where $\mathcal{A}$ is the collection of all the convex sets in $\mb R^K$. Using the fact $|\mf l(\mf x)|=h(\mf x)/2\sqrt{n}p(0)$, we have 
\begin{align}
    &\sup_{C\geq 0}\left|\mr P\left(\frac{\sqrt{n}|\hat{Q}(\mf x)-Q_{0,n}(\mf x)|}{h(\mf x)}\leq C,\forall \mf x\in\mathcal{D}_n\right)-\mr P\left(\frac{\left|\bar Q_n^{-1} \mf G_n^\top \mf b_\omega(\mf x )\right|}{h(\mf x)} \leq C,\forall \mf x\in\mathcal{D}_n\right)\right|,\notag\\
    &=\sup_{C\geq 0} \left|\mr P\left(\left|\hat{\beta}^\top\mf b_\omega(\mf x)\right|\leq C|\mf l(\mf x)|,\forall \mf x\in\mathcal{D}_n\right)-\mr P\left(\left|\frac{1}{2\sqrt{n}p(0)} \mf G_n^\top \mf b_\omega(\mf x )\right| \leq C|\mf l(\mf x)|,\forall \mf x\in\mathcal{D}_n\right)\right|,\notag\\
    &=\sup_{C\geq 0}\left|\mr P\left(\hat{\beta}\in A(C)\right)-\mr P\left(\frac{1}{2\sqrt{n}p(0)}\mf G_n\in A(C) \right)\right|,\notag\\
    &\leq \sup_{A\in\mathcal{A}}\left| \mr P\left(\sqrt{n}\hat{\beta}\in A\right)-\mr P\left(\bar Q_n^{-1}\mf G_n\in A \right) \right|,\label{eq:SCR decomp}
\end{align}    
which immediately yields \eqref{eq:SCR approximation} by Theorem \ref{thm:Gaussian approximation}.
\end{proof}

\subsection{Proof of Proposition \ref{prop:critical value}}
\begin{proof}
Suppose $C_\alpha$ satisfies
$$
\mr P\left\{\sup_{\mf x\in\mathcal{X}}\frac{|\mf b(\mf x)^\top\bar Q_n^{-1}\mf G_n |}{\sigma_n(\mf x)}\leq C_{\alpha} \right\}=1-\alpha.
$$
Note that
$$
\mr P\left\{\sup_{\mf x\in\mathcal{X}}\frac{|\mf b(\mf x)^\top\bar Q_n^{-1}\mf G_n |}{h(\mf x)}\leq C_{\alpha,h} \right\}=1-\alpha=\mr P\left\{\sup_{\mf x\in\mathcal{X}}\frac{|\mf b(\mf x)^\top\bar Q_n^{-1}\mf G_n |}{\sigma_n(\mf x)}\leq C_{\alpha} \right\}.
$$
On the one hand,
\begin{align}
    \mr P\left\{\sup_{\mf x\in\mathcal{X}}\frac{|\mf b(\mf x)^\top\bar Q_n^{-1}\mf G_n |}{h(\mf x)}\leq C_{\alpha,h} \right\}&=\mr P\left\{\frac{|\mf b(\mf x)^\top\bar Q_n^{-1}\mf G_n |}{\sigma_n(\mf x)}\leq C_{\alpha,h} \frac{h(\mf x)}{\sigma_n(\mf x)}, \forall \mf x\in \mathcal{X}\right\},\notag\\
    &\leq \mr P\left\{\sup_{\mf x\in\mathcal{X}}\frac{|\mf b(\mf x)^\top\bar Q_n^{-1}\mf G_n |}{\sigma_n(\mf x)}\leq C_{\alpha,h} \sup_{ \mf x\in \mathcal{X}}\frac{h(\mf x)}{\sigma_n(\mf x)}\right\}.\notag
\end{align}
By the monotonicity of the distribution function, we have
\begin{equation}
    C_\alpha \leq C_{\alpha,h} \sup_{ \mf x\in \mathcal{X}}\frac{h(\mf x)}{\sigma_n(\mf x)}.\label{eq:critical value lower bound}
\end{equation}
On the other hand,
\begin{align}
    \mr P\left\{\sup_{\mf x\in\mathcal{X}}\frac{|\mf b(\mf x)^\top\bar Q_n^{-1}\mf G_n |}{\sigma_n(\mf x)}\leq C_{\alpha} \right\}&=\mr P\left\{\frac{|\mf b(\mf x)^\top\bar Q_n^{-1}\mf G_n |}{h(\mf x)}\leq C_{\alpha} \frac{\sigma_n(\mf x)}{h(\mf x)}, \forall \mf x\in \mathcal{X}\right\},\notag\\
    &\leq \mr P\left\{\sup_{\mf x\in\mathcal{X}}\frac{|\mf b(\mf x)^\top\bar Q_n^{-1}\mf G_n |}{\sigma_n(\mf x)}\leq C_{\alpha} \sup_{ \mf x\in \mathcal{X}}\frac{\sigma_n(\mf x)}{h(\mf x)}\right\}.\notag
\end{align}
By the monotonicity of the distribution function, we also have
\begin{equation}
    C_{\alpha,h} \leq C_{\alpha} \sup_{ \mf x\in \mathcal{X}}\frac{\sigma_n(\mf x)}{h(\mf x)}.\label{eq:critical value upper bound}
\end{equation}
Combining \eqref{eq:critical value lower bound} and \eqref{eq:critical value upper bound}, we have
$$
C_\alpha \inf_{ \mf x\in \mathcal{X}}\frac{\sigma_n(\mf x)}{h(\mf x)}\leq C_{\alpha,h} \leq C_{\alpha} \sup_{ \mf x\in \mathcal{X}}\frac{\sigma_n(\mf x)}{h(\mf x)}.
$$
Based on \eqref{eq:SCR decomp} and \eqref{eq:SCR approximation}, it suffices to show $C_\alpha\asymp \sqrt{\log n}$ such that
\begin{equation}
    \mr P\left( \sup_{\mf x\in \mathcal{X}} |\mf T(\mf x)^\top \mf G| \leq  C_\alpha \right)= 1-\alpha,\label{eq:critical value GA}
\end{equation}
where $\mf T(\mf x)=\mf l(\mf x)/|\mf l(\mf x)|=\Sigma_{\mf Z}^{1/2}\mf b(\mf x)/\sqrt{\mf b(\mf x)^\top \Sigma_{\mf Z}\mf b(\mf x)}$ and $\mf G$ is standard $K$-dimensional random vector. Denote manifold
\begin{equation}
    \mathcal{M}=:\{\mf T(\mf x): \mf x \in \mathcal{X}\},
\end{equation}
and let $\kappa_0$ be the volume of the manifold $\mathcal{M}$, and $\zeta_0$ be the area of the boundary $\partial\mathcal{M}$; Let $\kappa_2$ and $\zeta_1$ be measures of the curvature of $\mathcal{M}$ and $\partial\mathcal{M}$ respectively, and $m_0$ measures the rotation angles in the regions $\partial^2\mathcal{M}$.

By Proposition 3 in \cite{Sun&Loader}, for the $\alpha$ in \eqref{eq:critical value GA}, we have 
    \begin{align}
\alpha = & \kappa_0 \frac{\Gamma((d+1) / 2)}{\pi^{(d+1) / 2}}\mr P\left(\chi_{d+1}^2>C_\alpha^2\right)+\frac{\zeta_0}{2} \frac{\Gamma(d / 2)}{\pi^{d / 2}}\mr P\left(\chi_d^2>C_\alpha^2\right) \notag\\
& +\frac{\kappa_2+\zeta_1+m_0}{2 \pi} \frac{\Gamma((d-1) / 2)}{\pi^{(d-1) / 2}}\mr P\left(\chi_{d-1}^2>C_\alpha^2\right) \notag\\
& +O\left(C_\alpha^{d-4} \exp \left(-\frac{C_\alpha^2}{2}\right)\right),\label{eq:alpha compute}
    \end{align}
where $\chi_d^2$ is the chi-square random variable with the degree of freedom $d$.

To bound the positive geometric quantities $\kappa_0,\zeta_0,\kappa_2,\zeta_1,m_0$ appearing in \eqref{eq:alpha compute}, we give the following formulations for numerical computation. For simplicity, we suppose $\mathcal{X}=[0,1]^d$ and the boundary $\partial \mathcal{X}$ consists of those points $\mf x$ with exactly one component 0 or 1. The regions where two faces of $\partial \mathcal{X}$ meet are denoted $\partial^2\mathcal{X}$.
Denote matrix $\mf A=(\mf T_1(\mf x),\dots, \mf T_d(\mf x))$ where $\mf T_j(\mf x)=\partial \mf T(\mf x)/\partial x_j$ with $\mf x=(x_1,\dots, x_d)^\top$ and indicator vector $\mf e_j=(e_{j,1},\dots,e_{j,d})^\top$ such that $e_{j,j}=1$ and $e_{j,k}=0$ if $k\neq j$. By (3.2) and (3.3) in \cite{Sun&Loader}, the $\kappa_0$ and $\kappa_2$ can be computed as
\begin{align}
    \kappa_0&=\int_{\mathcal{X}} \operatorname{det}^{1/2}(\mf A^\top \mf A)  \mr d\mf x, \label{eq:kappa0}\\
    \kappa_2&= \int_{\mathcal{X}} \left\{\frac{S(\mf x)}{2}-\frac{d(d-1)}{2}\right\}\operatorname{det}^{1/2}(\mf A^\top \mf A)  \mr d\mf x,\label{eq:kappa2}
\end{align}
where
\begin{align*}
&S(\mf x)=2\sum_{j=2}^d\sum_{k=1}^{j-1}[\alpha_{j,j}(\mf x)^\top \alpha_{k,k}(\mf x)-\alpha_{j,k}(\mf x)^\top 
\alpha_{k,j}(\mf x) ],\\
&\alpha_{k, j}(\mf x)^\top=\mf e_k^\top\left(\mf A^\top \mf A\right)^{-1} \frac{\partial \mf A^\top}{\partial x_j}\left(I-\mf A\left(\mf A^\top \mf A\right)^{-1} \mf A^\top\right).
\end{align*}

For $\zeta_1,\zeta_0$ measuring the boundary $\partial\mathcal{M}$, by (3.4) in \cite{Sun&Loader} and  the second and third equation on Page 1335 of \cite{Sun&Loader},
\begin{align}
    \zeta_0&=\int_{\partial \mathcal{X}}\operatorname{det}^{1/2}(\mf A_*^\top\mf A_*)\mr d\mf x,\label{eq:zeta0}\\
    \zeta_1&=\int_{\partial\mathcal{X}}\zeta_1(\mf x)\operatorname{det}^{1/2}(\mf A_*^\top\mf A_*)\mr d\mf x,\label{eq:zeta1}
\end{align}
where indicator vector $\mf e_j^*=(e_{j,1},\dots,e_{j,d-1})^\top$ such that $e_{j,j}=1$ and $e_{j,k}=0$ if $k\neq j$.
\begin{align}
    \zeta_1(\mf x)&=-\sum_{j=1}^{d-1}(\mf e_j^{*})^\top (\mf A_*^\top\mf A_*)^{-1}\frac{\partial \mf A_*^\top}{\partial x_j}\mf U_d(\mf x),\notag\\
    \mf U_d(\mf x)&\asymp (I-\mf A_*(\mf A_*^\top\mf A_*)^{-1}\mf A_*^\top)\mf T_d(\mf x),\notag\\
    \mf A_*&=(\mf T_1(\mf x),\dots,\mf T_{d-1}(\mf x)),
\end{align}
on the face $\mf x\in\partial \mathcal{X}$ at which $x_d$ is maximized, with similar definitions for $\zeta_1(\mf x),\mf U_j(\mf x), \mf A_*$ on other faces where $x_j$ is maximized. Moreover, by the fifth and seventh equation on Page 1335 of \cite{Sun&Loader}, 
\begin{equation}
    m_0=\int_{\partial^2\mathcal{X}}m_0(\mf x)\operatorname{det}^{1/2}(\mf A_{**}^\top\mf A_{**})\mr d\mf x,\label{eq:m0}
\end{equation}
where
\begin{align}
    m_0(\mf x)&=\cos^{-1}\left(  \mf U_{d-1}(\mf x)^\top \mf U_d(\mf x) \right),\notag\\
    \mf A_{**}&=(\mf T_1(\mf x),\dots,\mf T_{d-2}(\mf x)),
\end{align}
at a point $\mf x$ at the meeting of the faces $x_{d-1}=1$ and $x_{d}=1$, with similar 
definitions for $m_0(\mf x)$ on other meetings of the two faces of $\partial\mathcal{X}$.

Denote $\widetilde{\mf b}(\mf x)=\Sigma_{\mf Z}^{1/2}\mf b(\mf x)$ and $\partial_{x_j}\widetilde{\mf b}(\mf x)=\partial \widetilde{\mf b}(\mf x)/\partial x_j$, then basic calculation yields
\begin{equation}
        0\leq |\mf T_j(\mf x)|^2= \frac{|\partial_{x_j}\widetilde{\mf b}(\mf x)|^2}{|\widetilde{\mf b}(\mf x)|^2}-\frac{\left|\left(\partial_{x_j}\widetilde{\mf b}(\mf x)\right)^\top \widetilde{\mf b}(\mf x)\right|^2}{|\widetilde{\mf b}(\mf x)|^4}\leq \frac{|\partial_{x_j}\widetilde{\mf b}(\mf x)|^2}{|\widetilde{\mf b}(\mf x)|^2}.\notag
\end{equation}
Using the fact $\operatorname{det}^{1/d}(\mf A^\top \mf A)\leq \operatorname{tr}(\mf A^\top \mf A)/d$ and $\operatorname{tr}(\mf A^\top \mf A)=\sum_{j=1}^d |\mf T_j(\mf x)|^2$, by \eqref{eq:kappa0}, \eqref{eq:zeta0}, \eqref{eq:m0}. Note that $|\tilde{\mf b}(\mf x)|\geq \sqrt{\lambda_{min}(\Sigma_{\mf Z})} |\mf b(\mf x)|\gtrsim n^{c_0}$ by condition \eqref{eq:critical value condition} and Assumption \hyperref[(A1)]{(A1)}, there exists constant $c_1>0$ such that
\begin{align}
    \kappa_0&\leq \int_{\mathcal{X}} \left( \frac{1}{d}\operatorname{tr}(\mf A^\top \mf A)\right)^{d/2}\mr d\mf x\lesssim \int_{\mathcal{X}} \left( \sum_{j=1}^d |\mf T_j(\mf x)|^2\right)^{d/2}\mr d\mf x\lesssim  \int_{\mathcal{X}} \frac{|\nabla\widetilde{\mf b}(\mf x)|^d}{|\widetilde{\mf b}(\mf x)|^d}\mr d\mf x=O(n^{c_1}),\notag
\end{align}
and similarly, $\zeta_0=O\left( n^{c_1} \right),m_0=O(n^{c_1})$ since $\operatorname{tr}(\mf A_{**}^\top \mf A_{**})\leq \operatorname{tr}(\mf A_*^\top \mf A_*)\leq \operatorname{tr}(\mf A^\top \mf A)$.

For $\kappa_2$, note that matrix $\mf A(\mf A^\top \mf A)^{-1}\mf A^\top$ is idempotent, thus 
$$
|\alpha_{k,j}(\mf x)|\leq \left| \frac{\partial \mf A}{\partial x_j}(\mf A^\top \mf A)^{-1}\mf e_k \right|\leq \frac{1}{\lambda_{min}(\mf A^\top \mf A)}  \left|\frac{\partial \mf A}{\partial x_j} \right|
$$
By condition \eqref{eq:critical value condition}, for some constant $c_2>0$,
$$
\kappa_2\lesssim n^{c_1}\int_{\mathcal{X}}\frac{1}{\lambda_{min}(\mf A^\top \mf A)}  \left|\frac{\partial \mf A}{\partial x_j} \right|  \mr d\mf x+O(n^{c_1})=O(n^{c_1+c_2}).
$$
thus $\kappa_2=O(n^{c_3})$ and similarly, we have $\zeta_1=O(n^{c_3})$ for some constant $c_3>0$.

To sum up, there exists constant $\bar c>0$ such that
\begin{equation}
    \max\{\kappa_0,\zeta_0,\kappa_2,\zeta_1,m_0 \}=O(n^{\bar c}). \label{eq:bound geom constants}
\end{equation}
By Theorem 6 in \cite{zhang2020non}, for constants $\tilde{c},\tilde{C},\bar{C}>0$, the tail bounds of $\chi^2_{d}$ 
\begin{equation}
    \tilde{c} \exp \left(-\tilde{C} x \right) \leq \mr{P}(\chi_d^2-d \geq x) \leq \exp \left(-\bar{C} x \right), \forall x>d,\label{eq:tail bound of chi}
\end{equation}
Combining \eqref{eq:bound geom constants}, \eqref{eq:alpha compute}, \eqref{eq:tail bound of chi} and the fact $\alpha$ is a fixed value, we have $n^{\bar c}\exp(-\tilde{C} C_\alpha^2 )\gtrsim 1$, which implies that $C_\alpha=O(\sqrt{\log n})$. On the other hand, by condition \eqref{eq:critical value condition}, there exists constant $\underline{c}>0$, such that 
$$
\kappa_0\geq \int_\mathcal{X} \lambda^{d/2}_{min}(\mf A^\top \mf A)\mr d\mf x\gtrsim n^{\underline{c}}.
$$
Combining \eqref{eq:alpha compute}, \eqref{eq:tail bound of chi} and the fact $\alpha$ is a fixed value, we have
$n^{\underline{c}}\exp(-\tilde{C}C_\alpha^2)\lesssim 1$, which shows $C_\alpha\gtrsim \sqrt{\log n}$. Combining \eqref{eq:critical value GA}, \eqref{eq:SCR decomp}, and Theorem \ref{thm:Gaussian approximation}, we have 
\begin{equation*}
    \lim_{n\rightarrow\infty}\mr P\left( \hat Q(\mf x)- \frac{C_{\alpha,h} }{\sqrt{n}}h(\mf x) \leq Q_{0,n}(\mf x)\leq \hat Q(\mf x)+ \frac{C_{\alpha,h} }{\sqrt{n}}h(\mf x),\forall \mf x\in \mathcal{D}_n\right)= 1-\alpha,
\end{equation*}
with
$$
\inf_{\mf x\in \mathcal{X}}\frac{\sigma_n(\mf x)}{h(\mf x)}\sqrt{\log n} \lesssim C_{\alpha,h} \lesssim\sup_{\mf x\in \mathcal{X}}\frac{\sigma_n(\mf x)}{h(\mf x)}\sqrt{\log n}.
$$
\end{proof} 


\subsection{Proof of Proposition \ref{prop:undersmoothing}}

\begin{proof}
Let
\[
T_{n,h}^{s}
=
\sup_{\mf x\in\mathcal D_n}
\frac{\sqrt n|\widehat Q(\mf x)-Q_{0,n}(\mf x)|}{h(\mf x)},
\qquad
T_{n,h}^{0}
=
\sup_{\mf x\in\mathcal D_n}
\frac{\sqrt n|\widehat Q(\mf x)-Q_0(\mf x)|}{h(\mf x)}.
\]
By Assumption \hyperref[(A5)]{(A5)}, there exists a constant \(C_A>0\) such that, for all
sufficiently large \(n\),
\[
\left|T_{n,h}^{0}-T_{n,h}^{s}\right|
\le
C_A\sqrt n K^{-\varsigma}
\sup_{\mf x\in\mathcal D_n}h(\mf x)^{-1}.
\]
Hence,
\begin{equation}
\{T_{n,h}^{s}\le C_{\alpha,h}-\delta_{n,h}\}
\subset
\{T_{n,h}^{0}\le C_{\alpha,h}\}
\subset
\{T_{n,h}^{s}\le C_{\alpha,h}+\delta_{n,h}\}.
\label{eq:sandwich-Q0-Q0n}
\end{equation}
Next define the Gaussian supremum
\[
Z_{n,h}
=
\sup_{\mf x\in\mathcal D_n}
\left|
\frac{
\mf b_\omega(\mf x)^\top \bar Q_n^{-1}\mf G_n}{h(\mf x)}
\right|,
\]
where \(\mf G_n\sim N(0,\Sigma_Z)\). Let
\[
\sigma_{h,n}(\mf x)
=
\frac{\sigma_n(\mf x)}{h(\mf x)},
\qquad
\underline\sigma_{h,n}
=
\inf_{\mf x\in\mathcal D_n}
\sigma_{h,n}(\mf x).
\]
Then $\operatorname{Var}(\mf b_\omega(\mf x)^\top \bar Q_n^{-1}\mf G_n/h(\mf x))\geq\underline\sigma_{h,n}$. By Theorem \ref{thm:Gaussian approximation} and the convexity of the sets
\[
A(t)
=
\left\{
u\in\mathbb R^K:
|u^\top \mf b_\omega(\mf x)|\le t h(\mf x),
\ \forall \mf x\in\mathcal D_n
\right\},
\]
we have the uniform distributional approximation
\begin{equation}
\sup_{t\in\mathbb R}
\left|
\mr P(T_{n,h}^{s}\le t)-\mr P(Z_{n,h}\le t)
\right|
=o(1).
\label{eq:uniform-dist-approx}    
\end{equation}
Note that the proof of Proposition \ref{prop:critical value} gives a polynomial
upper bound for the geometric constants in the volume-of-tube formula and therefore a tail bound of the form
\[
\mr P\left(
\sup_{\mf x\in\mathcal D_n}
\left|
\frac{
\mf b_\omega(\mf x)^\top \bar Q_n^{-1}\mf G_n
}{
\sigma_n(\mf x)
}
\right|>u
\right)
\le
C n^c e^{-c'u^2}
\]
for all sufficiently large \(u\). As a result, we have
\begin{equation}
\mb E\left[
\sup_{\mf x\in\mathcal D_n}
\left|
\frac{
\mf b_\omega(\mf x)^\top \bar Q_n^{-1}\mf G_n
}{
\sigma_n(\mf x)
}
\right|
\right]
\lesssim \sqrt{\log n}.
\label{eq:std-gaussian-complexity}
\end{equation}
Combining Lemma \ref{lem:gaussian-anti-concentration} with
\eqref{eq:std-gaussian-complexity}, we obtain
\begin{equation}
\sup_{t\in\mathbb R}
\mr P\left(
|Z_{n,h}-t|\le u
\right)
\lesssim
u\,
\frac{\sqrt{\log n}}{\underline\sigma_{h,n}},
\qquad u>0.
\label{eq:anti-concentration-Znh}    
\end{equation}
Using \eqref{eq:uniform-dist-approx} and \eqref{eq:anti-concentration-Znh}, for $\delta_{n,h}=\sqrt n K^{-\varsigma}
\sup_{\mf x\in\mathcal D_n}h(\mf x)^{-1}$
\[
\begin{aligned}
&\mr P(T_{n,h}^{s}\le C_{\alpha,h}+\delta_{n,h})
-
\mr P(T_{n,h}^{s}\le C_{\alpha,h}-\delta_{n,h})  \\
&\le
\mr P(Z_{n,h}\le C_{\alpha,h}+\delta_{n,h})
-
\mr P(Z_{n,h}\le C_{\alpha,h}-\delta_{n,h})
+
2\sup_{t\in\mathbb R}
\left|
\mr P(T_{n,h}^{s}\le t)-\mr P(Z_{n,h}\le t)
\right|  \\
&\le
\sup_{t\in\mathbb R}
P(|Z_{n,h}-t|\le \delta_{n,h})
+
2\sup_{t\in\mathbb R}
\left|
\mr P(T_{n,h}^{s}\le t)-\mr P(Z_{n,h}\le t)
\right|  \\
&\lesssim
\delta_{n,h}
\frac{\sqrt{\log n}}{\underline\sigma_{h,n}}
+
o(1).
\end{aligned}
\]
By condition \eqref{eq:undersmoothing condition},
\begin{equation}
\delta_{n,h}
\frac{\sqrt{\log n}}{\underline\sigma_{h,n}}=K^{-\varsigma}\sqrt{n\log n}
\frac{
\sup_{\mf x\in\mathcal D_n}h(\mf x)^{-1}
}{
\inf_{\mf x\in\mathcal D_n}\sigma_n(\mf x)/h(\mf x)
}
=o(1).
\label{eq:strong-undersmoothing}
\end{equation}
Hence,
\begin{equation}
\mr P(T_{n,h}^{s}\le C_{\alpha,h}+\delta_{n,h})
-
\mr P(T_{n,h}^{s}\le C_{\alpha,h}-\delta_{n,h})
=o(1).
\label{eq:small-shift}    
\end{equation}
Since \(C_{\alpha,h}\) satisfies \eqref{eq:SCR on Q}, equivalently,
$\lim_{n\to\infty}\mr P(T_{n,h}^{s}\le C_{\alpha,h})= 1-\alpha,$
\eqref{eq:small-shift} and \eqref{eq:sandwich-Q0-Q0n} yield
\[
\lim_{n\to\infty}\mr P(T_{n,h}^{0}\le C_{\alpha,h})= 1-\alpha.
\]
This is exactly
\[
\lim_{n\to\infty}
P\left(
\widehat Q(\mf x)-\frac{C_{\alpha,h}}{\sqrt n}h(\mf x)
\le Q_0(\mf x)\le
\widehat Q(\mf x)+\frac{C_{\alpha,h}}{\sqrt n}h(\mf x),
\ \forall \mf x\in\mathcal D_n
\right)
=
1-\alpha.
\]
The proof is complete.
\end{proof}

\subsection{Bootstrap validation}
\label{sec:bootstrap validation}

\begin{proof}[Proof of Theorem \ref{thm:bootstrap}]
    Recall the Bahadur representation $\mf W_n=: \bar Q_n^{-1}\mf Z_n$ for $\sqrt{n}\hat{\beta}=\sqrt{n}(\hat{\theta}-\theta_{0,n})$, we denote $\hat{\beta}(j)=:\hat{\theta}(j)-\theta_{0,n}$ and introduce Lemma \ref{lem:uniform Bahadur} to show the Bahadur representation for $\sqrt{M}\hat{\beta}(j)$ can be denoted as
 \begin{equation}
     \mf W(j)=:\frac{1}{\sqrt{M}}\sum_{i=j}^{j+M-1} \bar{Q}_n^{-1}\psi(\varepsilon_i)\mf b_\omega(\mf X_i).
 \end{equation}
 By Lemma \ref{lem:Fang2024} and similar arguments in \eqref{eq:decomp of Kolmogorov dist}, and \eqref{eq:Gaussian vector comparison}, on condition $\Upsilon_n$, we have
\begin{align}
    &\sup_{A\in\mathcal{A}}\left|\mr P\left(\Phi_M\in A\Big|\mathcal{H}_n,\mathcal{G}_n\right)-\mr P\left(\bar Q_n^{-1}\mf G_n\in A\right)\right|\notag\\
    &\lesssim \epsilon_3 K^{1/4}+\sup_{A \in \mathcal{A}}\left|\mb E\left[ h_{A, \epsilon_3}\left(\Phi_M\right)\Big | \Upsilon_n\right]-\mb E\left[ h_{A, \epsilon_3}\left(\tilde{\Phi}_M \right)\Big | \Upsilon_n\right] \right|\notag\\
    &+\sup_{A \in \mathcal{A}}\left|\mb E\left[ h_{A, \epsilon_3}\left(\tilde{\Phi}_M \right)\Big | \Upsilon_n\right]-\mb E h_{A, \epsilon_3}\left(\bar Q_n^{-1}\mf G_n \right)\right|\notag\\
    &\lesssim \epsilon_3 K^{1/4}+ \frac{1}{\epsilon_3} \mb E\left(|\Phi_M-\tilde{\Phi}_M|\Big| \Upsilon_n\right)+\frac{1}{\epsilon_3}\left\| \tilde{\boldsymbol{\Sigma}}_M-\boldsymbol{\Sigma}_n \right\|_F.\label{eq:decomp K-dist bootstrap}
\end{align}
where
\begin{align}
       \tilde{\Phi}_M&=:\frac{1}{\sqrt{n-M+1}}\sum_{j=1}^{n-M+1} \mf W(j)V_j,\label{eq:def tilde_Phi}\\
       \boldsymbol{\Sigma}_n&=:\operatorname{Cov}\left(\bar{Q}_n^{-1}\mf G_n\right)=\mb E\left\{\left(\bar Q_n^{-1}\mf G_n\right)\left(\bar Q_n^{-1}\mf G_n\right)^\top\right\},\notag\\
           \tilde{\boldsymbol{\Sigma}}_M&=: \operatorname{Cov}(\tilde{\Phi}_M|\Upsilon_n)=\frac{1}{n-M+1}\sum_{j=1}^{n-M+1}\mf W(j)\mf W(j)^\top,\label{eq:def tilde_Sigma}
\end{align}
and $V_j$ is the same standard Gaussian random variable in \eqref{eq:bootstrap sample}.

By Lemma \ref{lem: cov comparison-element} and Lemma \ref{lem: comparison between Pcov and cov}, we have
\begin{align}
    \left\| \tilde{\boldsymbol{\Sigma}}_M-\boldsymbol{\Sigma}_n \right\|_F&\leq \left\| \tilde{\boldsymbol{\Sigma}}_M-\mb E\tilde{\boldsymbol{\Sigma}}_M\right\|_F+\left\| \mb E\tilde{\boldsymbol{\Sigma}}_M-\boldsymbol{\Sigma}_n \right\|_F\notag\\
    &=O_p\left(\xi_{K,n}^2\Delta_{K,n}^{\alpha} \left(\sqrt{\frac{M}{n}}+\frac{1}{M}\right) \right)\label{eq:rate of comparison between tile_Sigma}
\end{align}
Note that $V_j,j=1,\dots, n-M+1$ in \eqref{eq:def tilde_Phi} are i.i.d. standard normal random variable, condition on data, by Proposition \ref{prop:quadra approx} and Lemma \ref{lem:uniform Bahadur}, 
\begin{align}
    \mb E\left\{\left| \Phi_M-\tilde{\Phi}_M \right|\Big |\Upsilon_n\right\}
    &\leq  \sqrt{\frac{M}{n-M+1}\sum_{j=1}^{n-M+1}\left|\hat\beta(j)-\frac{\sum_{i=j}^{j+M-1} \bar Q_n^{-1}\mf z_{i} }{2M}\right|^2} +  \sqrt{M}|\hat{\beta}|,\notag\\
    &=O_p\left(\Delta_{K,n}^{\alpha}K^{\frac{\eta+2}{4}}M^{-\frac{\eta}{4}}\log M+ (\xi_{K,n}\vee \sqrt{K})M^{\frac{1}{2}}n^{-\frac{1}{2}}\log n \right).\label{eq:rate of phi-tile_phi}
\end{align}
Combining \eqref{eq:decomp K-dist bootstrap}, \eqref{eq:rate of comparison between tile_Sigma}, and \eqref{eq:rate of phi-tile_phi}, with appropriate $\epsilon_3$, we have for any given $\alpha\in(0,1)$,
\begin{align}
&\sup_{A\in\mathcal{A}}\left|\mr P\left(\Phi_M\in A\Big|\mathcal{H}_n,\mathcal{G}_n\right)-\mr P\left(\bar Q_n^{-1}\mf G_n\in A\right)\right|\notag\\
    &=O_p\left( \sqrt{ K^{\frac{1}{4}}\left[\Delta_{K,n}^{\alpha}\xi_{K,n}^2 \left(\sqrt{\frac{M}{n}}+\frac{1}{M}\right) + \Delta_{K,n}^{\alpha} K^{\frac{\eta+2}{4}} M^{-\frac{\eta}{4}} \log^2 n + (\xi_{K,n} \vee \sqrt{K})\sqrt{\frac{M}{n}}\log n \right]}  \right)\notag\\
    &= o_p\left(\log n \sqrt{K^{\frac{1}{4}} \left(\Delta_{K,n}^{\alpha}\xi_{K,n}^2\left(\sqrt{\frac{M}{n}}+\frac{1}{M}\right) + \Delta_{K,n}^{\alpha} K^{\frac{\eta+2}{4}} M^{-\frac{\eta}{4}}+ (\xi_{K,n} \vee \sqrt{K})\sqrt{\frac{M}{n}}\right)} \right).\label{eq:final rate of K-dist boot}
\end{align}

Note that the set
$$
A(C)=: \{\mf S\in \mb R^K:|\mf S^\top \mf b_\omega(\mf x)|\leq Ch(\mf x),\forall \mf x\in\mathcal{D}_n \}
$$
is a convex set and $A(C)\in \mathcal{A}$ for any $C>0$ where $\mathcal{A}$ is the collection of all convex sets in $\mb R^K$.
By the bootstrap Algorithm \ref{alg:self-convolved bootstrap}, there exists $\hat{C}_{\alpha,h}\in >0$ such that, as bootstrap sample $B\rightarrow\infty$,
\begin{equation}
    \mr P\left(|\Phi_M^\top\mf b_\omega(\mf x)|\leq \hat{C}_{\alpha, h}h(\mf x),\forall \mf x\in\mathcal{D}_n\Big| \Upsilon_n\right)\rightarrow_p 1-\alpha.
\end{equation}
If \eqref{eq:final rate of K-dist boot} is of rate $o_p(1)$, we then have
\begin{equation}
   \mr P\left(\left|\left(\bar Q_n^{-1}\mf G_n\right)^\top \mf b_\omega(\mf x)\right|\leq \hat{C}_{\alpha,h}h(\mf x),\forall \mf x\in\mathcal{D}_n\Big| \Upsilon_n\right)\rightarrow_p 1-\alpha,
\end{equation}
as $n\to \infty$ and $B\to \infty$. By Theorem \ref{thm:Gaussian approximation} and the boundedness of the conditional probability, taking expectations yields \eqref{eq:valid boot strap}.
\end{proof}

\begin{lemma}
\label{lem:max beta-j}
Under Assumptions \hyperref[(A1)]{(A1)}-\hyperref[(A4)]{(A4)} and \hyperref[(B1)]{(B1)}-\hyperref[(B3)]{(B3)}, suppose $M\rightarrow\infty$ and $M= O(n^{\kappa})$ for some $\kappa\in(0,1)$.
Define $\bar{r}_M = (\xi_{K,n} \vee \sqrt{K})M^{-1/2}\log^{1/2}n$ and $r_{\max} = \bar{r}_M \sqrt{\log n}$. If $\xi_{K,n}\bar{r}_M=o(1)$, then 
$$
\mr P\left(\max_{1\leq j\leq n-M+1}|\hat{\beta}(j)|>r_{\max}\right)=o(1).
$$
\end{lemma}

\begin{proof}
Let $T = n-M+1$. For any given block $j\in\{1,\dots,n-M+1\}$, define the block empirical criterion as
$$
\mathbb P_{n,j}(\beta) =: \frac{1}{M}\sum_{i=j}^{j+M-1}L_i(\beta).
$$
By similar arguments in \eqref{eq:inf lower bound} and the convexity of the empirical criterion $\mathbb{P}_{n,j}(\beta)$, the maximal deviation is governed by the boundary of the local neighborhood. We have for some constant $c_0>0$,
$$ \left\{ \max_{1\le j\le T} |\hat{\beta}(j)| > r_{\max} \right\} \subseteq \left\{ \max_{1\le j\le T} \sup_{|\beta| \le r_{\max}} \left| \mathbb{P}_{n,j}(\beta) - \mathbb{P}(\beta) \right| \ge c_0 r_{\max}^2 \right\}, $$
where $\mathbb{P}_{n,j}(\beta) - \mathbb{P}(\beta) = \frac{1}{M}\sum_{i=j}^{j+M-1} (L_i(\beta) - \mathbb{E}L_i(\beta))$.

Following the $m$-decomposition arguments in \eqref{eq:loss error for M-decomp} and Theorem \ref{thm:maximal inequality}, recall $L_i^{(m)}(\beta) = \mathbb{E}[L_i(\beta)|\Upsilon_i^{(m)}]$. 
We split the block $\mathcal{I}_{j} = \{j,\dots, j+M-1\}$ into $m$ independent interleaving subsets $\mathcal{I}_{j,k} = \{ i \in \mathcal{I}_{j} \mid i \equiv k \pmod m \}$. The effective sample size for each subset is $N_k \asymp M/m$. Define the supremum of the empirical process over each subset as
$$ 
\mathcal{S}_{j,k}=:\sup_{|\beta|\le r_{\max}} \left| \sum_{i \in \mathcal{I}_{j,k}} \big( L_i^{(m)}(\beta) - \mathbb{E}L_i(\beta) \big) \right|,
$$ 
then
$$ \sum_{j=1}^T \sum_{k=1}^m \mr P\left( \sup_{|\beta|\le r_{\max}} \left| \sum_{i \in \mathcal{I}_{j,k}} \big( L_i^{(m)}(\beta) - \mathbb{E}L_i(\beta) \big) \right| \ge \frac{M}{m} \frac{c_0}{8} r_{\max}^2 \right) = \sum_{j=1}^T \sum_{k=1}^m \mr P\Big( \mathcal{S}_{j,k} \ge \frac{M}{m} \frac{c_0}{8} r_{\max}^2 \Big). $$
Since the indices in each $\mathcal{I}_{j,k}$ are separated by exactly $m$, and $L_i^{(m)}(\beta)$ depends only on the recent $m$ innovations $\Upsilon_i^{(m)}$, the summands within $\mathcal{S}_{j,k}$ are i.i.d.
Applying Lemma \ref{lem:Bousquet's inequality} for independent empirical processes gives
$$ \mr P\Big( \mathcal{S}_{j,k} \ge \mathbb{E}[\mathcal{S}_{j,k}] + x \Big) \le \exp\left( - \frac{x^2}{2(v + 2U \mathbb{E}[\mathcal{S}_{j,k}]) + 2Ux/3} \right), $$
where $U$ is a uniform envelope bounding the function family defined in \eqref{eq: R2 function family} replacing $r_n$ by $r_{\max}$, and $v$ is the variance proxy satisfying $\sup_{|\beta| \le r_{\max}} \sum_{i \in \mathcal{I}_k} \text{Var}(L_i^{(m)}(\beta)) \le v$.
Over the ball $|\beta| \le r_{\max}$, the required parameters are evaluated as $U \asymp \xi_{K,n} r_{\max}$ and $v \asymp \frac{M}{m} r_{\max}^2$, with the expectation bound $\mathbb{E}[\mathcal{S}_{j,k}] \lesssim \frac{M}{m} r_{\max} \bar{r}_M = o\left( \frac{M}{m} r_{\max}^2 \right)$.

Substituting $x = \frac{M}{m} \frac{c_0}{8} r_{\max}^2 - \mathbb{E}[\mathcal{S}_{j,k}]$ into the exponent directly yields
\begin{align}
    \mr P\Big( \mathcal{S}_{j,k} \ge \frac{M}{m} \frac{c_0}{8} r_{\max}^2 \Big)&\leq \exp\left( -\frac{ \left( \frac{M}{m} r_{\max}^2 \right)^2 }{ \frac{M}{m} r_{\max}^2 + (\xi_{K,n} r_{\max}) \left( \frac{M}{m} r_{\max} \bar{r}_M \right) + (\xi_{K,n} r_{\max}) \left( \frac{M}{m} r_{\max}^2 \right) } \right),\notag\\
    &\leq \exp\left(- \frac{M}{m}r_{\max}^2\right)\leq \exp\left( -(\xi_{K,n}^2\vee K)m^{-1}\log^2 n \right)
\end{align}
Finally, by the triangle inequality and union bound, we explicitly decompose the probability into the $m$-decompostion error and the summation over the independent empirical processes. By \eqref{eq:loss error for M-decomp} and Markov's inequality for the m-decomposition error, and summing over $T \le n$ blocks and sub-sums with appropriately selected $m\asymp \log n$ (e.g., $m= \lfloor (1+\kappa+\omega_1^\prime)\log n /|\log \chi|\rfloor$), we have
\begin{align*}
\mr P\left( \max_{1\le j\le T} |\hat{\beta}(j)| > r_{\max} \right) 
&\le \mr P\left( \max_{1\le j\le T} \sup_{|\beta|\le r_{\max}} \left| \frac{1}{M}\sum_{i=j}^{j+M-1} \big( L_i(\beta) - L_i^{(m)}(\beta) \big) \right| \ge \frac{c_0}{2} r_{\max}^2 \right) \\
&\quad + \sum_{j=1}^T \sum_{k=1}^m \mr P\Big( \mathcal{S}_{j,k} \ge \frac{M}{m} \frac{c_0}{2} r_{\max}^2 \Big) \\
&\le \frac{C n \xi_{K,n} \Delta_{K,n}^\alpha \chi^m}{r_{\max}^2} + T m\exp\left(- \frac{M}{m}r_{\max}^2\right)=o(1),
\end{align*}
which concludes the proof.
\end{proof}

\begin{lemma}
\label{lem:uniform Bahadur}
    Under Assumptions \hyperref[(A1)]{(A1)}-\hyperref[(A4)]{(A4)} and \hyperref[(B1)]{(B1)}-\hyperref[(B3)]{(B3)}, suppose $M\rightarrow\infty$ and $M= O(n^{\kappa})$ for some $\kappa\in(0,1)$. If $\max\{K,\xi_{K,n}^2K^{1/2}\}=o(M^{1/2})$, we have for any given $\alpha\in(0,1)$,
    \begin{equation}
    \max_{1\leq j\leq n-M+1}\left|\hat{\beta}(j)-\bar Q_n^{-1}\frac{1}{M}\sum_{i=j}^{j+M-1} \mf z_i\right|=O_p\left( \Delta^{\alpha}_{K,n}(K\vee \xi_{K,n})^{(\eta+2)/4}M^{-(\eta+2)/4}\log^2 n \right).\label{eq: uniform Bahadur}
    \end{equation}    
\end{lemma}

\begin{proof}
Denote the target linear sequence and the estimation error on block $j$ as
\begin{equation}
    \beta_Z^{(j)}=:\bar Q_n^{-1}\frac{1}{M}\sum_{i=j}^{j+M-1}\mf z_i, \qquad e_j=:\hat\beta^{(j)}-\beta_Z^{(j)}.\label{eq:blockk Z}
\end{equation}
Recall the increment process $\widetilde R_i(\theta,h)$ in \eqref{eq:remainder}, and notice the algebraic identity $L_i(\theta+h) - L_i(\theta) = \widetilde R_i(\theta,h) - h^\top \mf z_i$.
Evaluating this at $\theta = \beta_Z^{(j)}$ and $h = e_j$, and averaging over the block $j$, we have
\begin{align}
\mathbb P_{n,j}(\hat\beta^{(j)}) - \mathbb P_{n,j}(\beta_Z^{(j)})
&= \widetilde{\mathbb R}_{n,j}(\beta_Z^{(j)}, e_j) - e_j^\top \bar Q_n \beta_Z^{(j)}, \label{eq: block diff}
\end{align}
where $\widetilde{\mathbb R}_{n,j}(\theta,h)=\frac{1}{M}\sum_{i=j}^{j+M-1} \widetilde{R}_i(\theta,h)$ is the block empirical increment. By the local deterministic expansion in Lemma \ref{lem:deterministic expansion}, the corresponding population increment is
\begin{equation}
\widetilde{\mathbb R}(\beta_Z^{(j)}, e_j) = e_j^\top \bar Q_n \beta_Z^{(j)} + \frac{1}{2} e_j^\top \bar Q_n e_j + r_P(\beta_Z^{(j)}, e_j). \label{eq: pop diff}
\end{equation}
Combining \eqref{eq: pop diff} and \eqref{eq: block diff}, using the minimizer property $\mathbb P_{n,j}(\hat\beta^{(j)}) \le \mathbb P_{n,j}(\beta_Z^{(j)})$ for any $1\leq j\leq n-M+1$,
\begin{equation}
 \frac{1}{2}e_j^\top \bar Q_n e_j 
\le 
 \Big| \widetilde{\mathbb R}_{n,j}(\beta_Z^{(j)}, e_j) - \widetilde{\mathbb R}(\beta_Z^{(j)}, e_j) \Big| 
- r_P(\beta_Z^{(j)}, e_j), \label{eq:uniform quad bound}
\end{equation}
where $r_P(\beta_Z^{(j)}, e_j)=O(t(\beta_Z^{(j)})|\beta_Z^{(j)}| |e_j|+t(e_j)|e_j|^2+t(\beta_Z^{(j)}) |e_j|^2 )$ with big $O(\cdot)$ uniformly with $j$. 

For any given $\alpha\in(0,1)$, denote $\tilde{r}_M=\Delta_{K,n}^\alpha(\xi_{K,n}\vee \sqrt{K})M^{-1/2}\log n$ and event 
\begin{equation}
    A_{M,0}=:\left\{ \max_{1\leq j\leq T}|\hat{\beta}(j)|\leq \tilde{r}_M,\max_{1\leq j\leq T}|\beta_Z^{(j)}|\leq \tilde{r}_M,\max_{1\leq j\leq T}|t(\beta_Z^{(j)})|\leq \tilde{r}_M \right\},\quad T=:n-M+1.
\end{equation}
By Lemma \ref{lem:max beta-j} and (ii) in Lemma \ref{lem:sup_bZ}, $\mr P(A_{M,0})=1-o(1)$. On event $A_{M,0}$, for any $1\leq j\leq T$,
$$
t(e_j)\leq t(\hat \beta(j))+t(\beta_Z^{(j)})\leq \nu_{M,0}
$$
where $\nu_{0,M}=(1+\xi_{K,n})\tilde{r}_M$. Denote the localized stochastic modulus
\begin{equation}
    \tilde \Gamma_M(a_\theta,a_h) = : \max_{1\leq j\leq n-M+1}\sup \left\{ \left|\frac{\tilde{\mb R}_{n,j}(\theta,h)-\tilde{\mb R}(\theta,h) }{|h|}\right|:|\theta|\leq \tilde r_M,0<|h|\leq 2\tilde r_M,t(\theta)\leq a_\theta,t(h)\leq a_h \right\}.
\end{equation}
Then for the event
$$
B_{M,0}=:\left\{\tilde{\Gamma}_M(\tilde r_M,\nu_{M,0})\leq g_nK^{1/2}M^{-1/2}(\tilde{r}_M+\nu_{M,0})^{\eta/2}\log^2 n\right\},
$$
Lemma \ref{lem:bound centered remainder} ensures $\mr P(B_{M,0}\cap A_{M,0})=1-o(1)$ with $g_n\to \infty$ at arbitrarily slowly.


\paragraph{Beginning step.}
On $A_{M,0}$, by Assumption \hyperref[(B3)]{(B3)} and \hyperref[(A2)]{(A2)}, using similar arguments in \eqref{eq:ineq iteration bahadur bound}, \eqref{eq:uniform quad bound} yields
\begin{equation}
\max_{1\leq j\leq n-M+1}|e_j| \le O_p\left( \tilde\Gamma_n(\tilde r_M,\nu_{M,0})+\tilde r_M^2\right),\label{eq:main iteration uniform bahadur bound}    
\end{equation}
Define
$$
\Phi_M(a,b)=:g_nK^{1/2}M^{-1/2}(a+b)^{\eta/2}\log^2 n+\tilde{r}_M^2.
$$
Notice that $\mr P(A_{M,0}\cap B_{M,0})=1-o(1)$, we have
\begin{equation}
    \max_{1\leq j\leq T}|e_j|=O_p(r_{M,0}),\quad r_{M,0}=\Phi_M(\tilde{r}_M,\nu_{M,0}).\label{eq:beginning uniform bahadur rate}
\end{equation}
Consequently $\max_{1\leq j\leq T}t(\hat \beta(j))=O_p(\nu_{M,1})$, $\nu_{M,1}=\tilde{r}_M+\xi_{K,n}r_{M,0}$.

\paragraph{First self-refinement.}
Based on $\nu_{M,1}=\tilde{r}_M+\xi_{K,n}r_{M,0}$, apply the same blockwise version of Lemma \ref{lem:bound centered remainder} with $a_\theta=\tilde{r}_M$, $a_h=\nu_{M,1}$. Define the event
$$
B_{M,1}=:\left\{\tilde \Gamma_M(\tilde{r}_M,\nu_{M,1})\leq g_nK^{1/2}M^{-1/2}(\tilde{r}_M+\nu_{M,1})^{\eta/2}\log^2 n\right\}.
$$
Then $\mr P(B_{M,1})=1-o(1)$ by Lemma \ref{lem:bound centered remainder}. Repeating the arguments leading to \eqref{eq:beginning uniform bahadur rate} yields
\begin{equation}
    \max_{1\leq j\leq T}|e_j|=O_p(r_{M,1}),\quad r_{M,1}=\Phi_M(\tilde{r}_M,\nu_{M,1}),
\end{equation}
and  $\max_{1\leq j\leq T}t(\hat \beta(j))=O_p(\nu_{M,2})$, $\nu_{M,2}=\tilde{r}_M+\xi_{K,n}r_{M,1}$.

\paragraph{Second self-refinement.} Apply the same argument once more with $a_h=\nu_{M,2}$. Define
$$
B_{M,2}=:\left\{\tilde \Gamma_M(\tilde{r}_M,\nu_{M,1})\leq g_nK^{1/2}M^{-1/2}(\tilde{r}_M+\nu_{M,2})^{\eta/2}\log^2 n\right\}
$$
Then $\mr P(B_{M,2})=1-o(1)$ by Lemma \ref{lem:bound centered remainder} and the same localization inequality yields
\begin{equation}
    \max_{1\leq j\leq T}|e_j|=O_p(r_{M,2}),\quad r_{M,1}=\Phi_M(\tilde{r}_M,\nu_{M,2}).
\end{equation}
Elementary calculation yields $r_{M, 2}$ can attain \eqref{eq: uniform Bahadur} by arguments
exactly as in the last part of the proof of Proposition \ref{prop:Bahadur representation} with the condition $\xi_{K,n}K^{\eta/4}M^{-\eta/4}\log^2 n=o(1)$. 
\end{proof}

\subsection{Lemmas for maximal inequality}


\begin{lemma}[Corollary 5.1 in \cite{Chernouzhukov2014EP}]
    \label{lem:modified maximal ineq uniform entropy}
    Let $\left(W_i\right)_{i=1}^n$ be a sequence of independent random elements taking values in a measurable space $\left(\mathcal{W}, \mathcal{A}_{\mathcal{W}}\right)$ according to the same probability law $P$. Let $\mathcal{F}$ be a set of suitably measurable functions $f: \mathcal{W} \rightarrow \mathbb{R}$, equipped with a measurable envelope $F: \mathcal{W} \rightarrow$ $\mathbb{R}$.
    Suppose that $F \geqslant \sup _{f \in \mathcal{F}}|f|$ is a measurable envelope for function family $\mathcal{F}$ with $\|F(W_i)\|_{q}<$ $\infty$ for some $q \geqslant 2$. Let $M_n=\max _{i \leqslant n} F\left(W_i\right)$ and $\sigma^2>0$ be any positive constant such that $\sup _{f \in \mathcal{F}}\|f(W_i)\|^2 \leqslant \sigma^2 \leqslant\|F(W_i)\|^2$. Suppose that there exist constants $a \geqslant e$ and $v \geqslant 1$ such that
$$
\log \sup _Q N\left(\epsilon\|F\|_{Q, 2}, \mathcal{F},\|\cdot\|_{Q, 2}\right) \leqslant v \log (a / \epsilon), 0<\epsilon \leqslant 1 .
$$
Then for empirical process $\mb G_n(f)=\sum_{i=1}^n (f(W_i)-\mb Ef(W_i))/\sqrt{n}$,
$$
\mb{E}\left[\sup_{f\in\mathcal{F}}|\mathbb{G}_n(f)|\right] \leqslant  C\left(\sqrt{v \sigma^2 \log \left(\frac{a\|F(W_i)\|}{\sigma}\right)}+\frac{v\|M_n\|}{\sqrt{n}} \log \left(\frac{a\|F(W_i)\|}{\sigma}\right)\right),
$$
where $C$ is an absolute constant.
\end{lemma}

\begin{lemma}[Theorem 4.17 in \cite{sen2018gentle}]
\label{lem:maximal ineq uniform entropy}
If $\mathcal{F}$ is a class of measurable functions with measurable envelope function $F$, then for i.i.d. $X_i, i=1, \ldots, n$,
$$
\mathbb{E}\left[\sup _{f \in \mathcal{F}}\left|\frac{1}{\sqrt{n}} \sum_{i=1}^n\left(f\left(X_i\right)-\mathbb{E} f\left(X_i\right)\right)\right|\right] \lesssim J(1, \mathcal{F}, F)\|F(X_i)\|,
$$
where
$$
J(\delta, \mathcal{F}, F)=:\int_0^\delta \sup _Q \sqrt{\log N\left(\epsilon\|F\|_{Q, 2}, \mathcal{F} \cup\{0\}, L_2(Q)\right)} d \epsilon, \quad \delta>0 .
$$
\end{lemma}

\begin{lemma}[Nazarov's inequality \cite{nazarov2003maximal}]
\label{lem:gaussian-anti-concentration}
Let \(Y=(Y_1,\ldots,Y_p)^\top\) be a centered Gaussian random vector with $\sigma_j^2=\mb E(Y_j^2)>0$, $\underline\sigma=\min_{1\le j\le p}\sigma_j$.
Then there exists a universal constant \(C>0\) such that, for every
\(\epsilon>0\),
\[
\sup_{t\in\mathbb R}
\mr P\left(
\left|\max_{1\le j\le p}Y_j-t\right|\le \epsilon
\right)
\le
C\epsilon\,
\frac{a_p+1}{\underline\sigma},
\]
where
$a_p=\mb E\left[\max_{1\le j\le p}\frac{Y_j}{\sigma_j}
\right]$. Consequently, the same inequality applies to
\(\max_{1\le j\le p}|Y_j|\) by applying the preceding bound to the
\(2p\)-dimensional Gaussian vector $(Y_1,\ldots,Y_p,-Y_1,\ldots,-Y_p)^\top$.
\end{lemma}

\begin{lemma}
\label{lem:Bousquet's inequality}
Let $X_1, \dots, X_n$ be independent random variables taking values in a measurable space $\mathcal{X}$, and let $\mathcal{F}$ be a countable class of measurable functions from $\mathcal{X}$ to $\mathbb{R}$. Assume there exists a constant $U > 0$ such that $f(X_i) - \mathbb{E}[f(X_i)] \le U$ almost surely for all $f \in \mathcal{F}$ and $i=1,\dots,n$. Define the supremum of the empirical process as
$$ Z = \sup_{f \in \mathcal{F}} \sum_{i=1}^n \big( f(X_i) - \mathbb{E}[f(X_i)] \big), $$
and let the variance proxy be $v = \sup_{f \in \mathcal{F}} \sum_{i=1}^n \mathrm{Var}(f(X_i))$. Then for any $x > 0$,
$$ \mr P\Big( Z \ge \mathbb{E}[Z] + x \Big) \le \exp\left( - \frac{x^2}{2(v + 2U \mathbb{E}[Z]) + 2Ux/3} \right). $$
\end{lemma}
\begin{proof}
This is a directly scaled version of Bousquet's inequality (Theorem 12.5 in \cite{boucheron2013concentration}). Let $\mathcal{G} = \{ f/U \mid f \in \mathcal{F} \}$. For any $g \in \mathcal{G}$, the scaled functions satisfy $g(X_i) - \mathbb{E}[g(X_i)] \le 1$ almost surely. Define the supremum for the scaled class as $Z' = \sup_{g \in \mathcal{G}} \sum_{i=1}^N \big( g(X_i) - \mathbb{E}[g(X_i)] \big) = Z / U$. The variance proxy for $\mathcal{G}$ is exactly $\sigma^2 = \sup_{g \in \mathcal{G}} \sum_{i=1}^N \mathrm{Var}(g(X_i)) = v / U^2$. 

Applying Theorem 12.5 of \cite{boucheron2013concentration} to the normalized class $\mathcal{G}$ with a deviation parameter $y = x / U > 0$, we obtain
$$ \mr P\Big( Z' \ge \mathbb{E}[Z'] + y \Big) \le \exp\left( - \frac{y^2}{2(\sigma^2 + 2\mathbb{E}[Z']) + 2y/3} \right). $$
Substituting $Z' = Z/U$, $\sigma^2 = v/U^2$, and $y = x/U$ into the above inequality yields
$$ \mr P\Big( Z/U \ge \mathbb{E}[Z]/U + x/U \Big) \le \exp\left( - \frac{(x/U)^2}{2(v/U^2 + 2\mathbb{E}[Z]/U) + 2x/(3U)} \right), $$
which concludes the proof.
\end{proof}

\begin{lemma}
\label{lem:bound centered remainder}
Denote
\[
\mathcal A_n(r_\theta,r_h;a_\theta,a_h)
=:
\left\{
(\theta,h): |\theta|\le r_\theta,\ 0<|h|\le r_h,\ 
t(\theta)\le a_\theta,\
t(h)\le a_h
\right\}.
\]
Under Assumptions \hyperref[(A1)]{(A1)}-\hyperref[(A4)]{(A4)} and
\hyperref[(B1)]{(B1)}-\hyperref[(B3)]{(B3)}, suppose $K^{1/2}\xi_{K,n}n^{-1/2}=o(1)$ and positive sequences $a_\theta=o(1),a_h=o(1)$. Then
\begin{align}
\mb E\sup_{(\theta,h)\in\mathcal A_n(r_\theta,r_h;a_\theta,a_h)}
\left|
\frac{1}{|h|}
\big(
\widetilde{\mb R}_n(\theta,h)-\widetilde{\mb R}(\theta,h)
\big)
\right|
=O\left(
K^{\frac12}n^{-\frac12}(a_\theta+a_h)^{\frac{\eta}{2}}\log^2 n\right).
\label{eq:centered remainder bound update}
\end{align}
\end{lemma}

\begin{proof}
Recall
\[
\mathcal H_i^{(m)}=(\zeta_{i-m+1},\dots,\zeta_{i-1},\zeta_i),\qquad
\mathcal G_i^{(m)}=(\eta_{i-m+1},\dots,\eta_{i-1},\eta_i),
\]
and
\[
\Upsilon_i^{(m)}=(\gamma_{i-m+1},\dots,\gamma_{i-1},\gamma_i),
\qquad \gamma_i=(\zeta_i,\eta_i).
\]
Define the $m$-dependent approximation
\[
\widetilde R_i^{(m)}(\theta,h)
=:
\mb E\!\left(
\widetilde R_i(\theta,h)\mid \Upsilon_i^{(m)}
\right).
\]
Without loss of generality, suppose $n/m\in\mb Z$ and write $L=n/m$.
Then for each fixed $j=1,\dots,m$, the process
$\left\{\widetilde R_{m(i-1)+j}^{(m)}(\theta,h)\right\}_{i=1}^{L}$ is i.i.d. over $i$.
For $k\ge m$, define the coupled variables
$\mf X_{i,k}=H(\mathcal H_{i,k})$, $\bar\varepsilon_{i,k}=a_n(\mf X_{i,k})+G(\mathcal H_{i,k},\mathcal G_{i,k})$,and
\[
\widetilde R_{i,k}(\theta,h)
=:
L_{i,k}(\theta+h)-L_{i,k}(\theta)+\psi(\bar\varepsilon_{i,k})h^\top \mf b_\omega(\mf X_{i,k}),
\]
where $L_{i,k}(\beta)=\rho(\bar\varepsilon_{i,k}-\beta^\top\mf b_\omega(\mf X_{i,k}))-\rho(\bar\varepsilon_{i,k})$. Then by similar arguments in \eqref{eq:sum up over m},
\begin{align}
\mb E \sup_{(\theta,h)\in\mathcal A_n(r_\theta,r_h;a_\theta,a_h)}
\frac{|\widetilde R_i(\theta,h)-\widetilde R_i^{(m)}(\theta,h)|}{|h|}
&\le
\sum_{k\ge m}
\mb E \sup_{(\theta,h)\in\mathcal A_n(r_\theta,r_h;a_\theta,a_h)}
\frac{|\widetilde R_i(\theta,h)-\widetilde R_{i,k}(\theta,h)|}{|h|}.
\label{eq:centered sum up over m}
\end{align}
Notice that
\begin{align}
\widetilde R_i(\theta,h)
&=
h^\top\mf b_\omega(\mf X_i)
\int_0^1
\Big\{
\psi(\bar\varepsilon_i)-\psi\big(\bar\varepsilon_i-\theta^\top\mf b_\omega(\mf X_i)-t h^\top\mf b_\omega(\mf X_i)\big)
\Big\}\mr dt.
\label{eq:centered increment representation}
\end{align}
Therefore,
\begin{align}
&\frac{1}{|h|}
\left|
\widetilde R_i(\theta,h)-\widetilde R_{i,k}(\theta,h)
\right|
\notag\\
&\le
\left|
\int_0^1
\frac{h^\top}{|h|}
\Big[
\mf b_\omega(\mf X_i)
\Big\{
\psi(\bar \varepsilon_i)-\psi(\bar\varepsilon_i-\theta^\top\mf b_\omega(\mf X_i)-t h^\top\mf b_\omega(\mf X_i))
\Big\}
\right.\notag\\
&\hspace{3cm}\left.
-
\mf b_\omega(\mf X_{i,k})
\Big\{
\psi(\bar\varepsilon_{i,k})-\psi(\bar\varepsilon_{i,k}-\theta^\top\mf b_\omega(\mf X_{i,k})-t h^\top\mf b_\omega(\mf X_{i,k}))
\Big\}
\Big]\mr dt
\right|.
\notag
\end{align}
By Assumption \hyperref[(A3)]{(A3)}, Assumption \hyperref[(B2)]{(B2)}, \eqref{eq:bound L},
and Lemma \ref{lem:phd on bx}, there exists a constant $\chi\in(0,1)$ such that for any
$\alpha\in(0,1)$,
\begin{align}
\mb E \sup_{(\theta,h)\in\mathcal A_n(r_\theta,r_h;a_\theta,a_h)}
\frac{|\widetilde R_i(\theta,h)-\widetilde R_{i,k}(\theta,h)|}{|h|}
=
O\left(
\xi_{K,n}\Delta_{K,n}^\alpha \chi^{\alpha k}
\right).
\label{eq:centered remainder one k}
\end{align}
Combining \eqref{eq:centered sum up over m} and \eqref{eq:centered remainder one k}, we obtain
\begin{align}
\mb E \sup_{(\theta,h)\in\mathcal A_n(r_\theta,r_h;a_\theta,a_h)}
\frac{|\widetilde R_i(\theta,h)-\widetilde R_i^{(m)}(\theta,h)|}{|h|}
=
O\left(
\xi_{K,n}\Delta_{K,n}^\alpha \chi^{\alpha m}
\right).
\label{eq:centered remainder error for M-decomp}
\end{align}
Consider the function class
\[
\bar{\mathcal R}_{L}^{\,c}
=:
\left\{
\bar f_{\theta,h}(\varepsilon,\mf x)
=
\frac{h^\top\mf b_\omega(\mf x)}{|h|}
\int_0^1
\Big[
\psi(\varepsilon)-\psi\big(\varepsilon-\theta^\top\mf b_\omega(\mf x)-t h^\top\mf b_\omega(\mf x)\big)
\Big]\mr dt
:
(\theta,h)\in\mathcal A_n(r_\theta,r_h;a_\theta,a_h)
\right\}.
\]
Denote the envelope function of $\bar{\mathcal R}_{L}^{\,c}$ as
\[
\bar F_c(\varepsilon,\mf x)
=
\sup_{(\theta,h)\in\mathcal A_n(r_\theta,r_h;a_\theta,a_h)}
|\bar f_{\theta,h}(\varepsilon,\mf x)|.
\]
By Assumption \hyperref[(B1)]{(B1)} and \eqref{eq:condition psi}, there exist constants $M_1,M_2>0$ s.t. 
\begin{align}
\|\bar F_c(\bar\varepsilon_i,\mf X_i)\|
&\le
\sup_{\mf x}|\mf b_\omega(\mf x)|
\left\|
\sup_{(\theta,h)\in\mathcal A_n(r_\theta,r_h;a_\theta,a_h)}
\sup_{0\le t\le 1}
\left|
\psi(\bar\varepsilon_i)-\psi\big(\bar\varepsilon_i-\theta^\top\mf b_\omega(\mf X_i)-t h^\top\mf b_\omega(\mf X_i)\big)
\right|
\right\|
\notag\\
&\le
\sup_{\mf x}|\mf b_\omega(\mf x)|
\left\|
M_1+M_2
\sup_{(\theta,h)\in\mathcal A_n(r_\theta,r_h;a_\theta,a_h)}
\sup_{0\le t\le 1}
\left|
\theta^\top\mf b_\omega(\mf X_i)+t h^\top\mf b_\omega(\mf X_i)
\right|
\right\|
\notag\\
&=
O(\xi_{K,n}),
\label{eq:centered bound envelope}
\end{align}
where the last line in \eqref{eq:centered bound envelope} is obtained from
\[
\sup_{(\theta,h)\in\mathcal A_n(r_\theta,r_h;a_\theta,a_h)}
\sup_{0\le t\le 1}
\left|
\theta^\top\mf b_\omega(\mf X_i)+t h^\top\mf b_\omega(\mf X_i)
\right|
\le a_\theta+a_h=o(1).
\]
Moreover, by Assumptions \hyperref[(A2)]{(A2)} and \hyperref[(B3)]{(B3)},
\begin{align}
\left\|
\max_{1\le i\le n}\bar F_c(\bar\varepsilon_i,\mf X_i)
\right\|
&\le
\sup_{\mf x}|\mf b_\omega(\mf x)|
\sqrt{
\mb E
\max_{1\le i\le n}
\sup_{(\theta,h)\in\mathcal A_n(r_\theta,r_h;a_\theta,a_h)}
\sup_{0\le t\le 1}
\left|
\psi(\bar\varepsilon_i)-\psi\big(\bar\varepsilon_i-\theta^\top\mf b_\omega(\mf X_i)-t h^\top\mf b_\omega(\mf X_i)\big)
\right|^2
}
\notag\\
&=O\left(
\xi_{K,n}(a_\theta+a_h)^{\eta/2}
\right).
\label{eq:centered bound Mn}
\end{align}
Since $|\theta+t h|\le 2r_n$, by Assumptions \hyperref[(A2)]{(A2)} and \hyperref[(B3)]{(B3)}, for any $(\theta,h)\in\mathcal A_n(r_\theta,r_h;a_\theta,a_h)$, 
\begin{align}
\|\bar f_{\theta,h}(\bar\varepsilon_i,\mf X_i)\|^2
&\le
\mb E\left[
\frac{|h^\top\mf b_\omega(\mf X_i)|^2}{|h|^2}
\sup_{0\le t\le 1}
\left|
\psi(\bar\varepsilon_i)-\psi\big(\bar\varepsilon_i-\theta^\top\mf b_\omega(\mf X_i)-t h^\top\mf b_\omega(\mf X_i)\big)
\right|^2
\right].\notag\\
&=O((a_\theta+a_h)^\eta).
\label{eq:centered bound sigma}
\end{align}
By Assumption \hyperref[(A4)]{(A4)}, the class
$\bar{\mathcal R}_{L}^{\,c}$ satisfies 
\[
\log \sup_Q N\left(
\epsilon \|\bar F_c\|_{Q,2},
\bar{\mathcal R}_{L}^{\,c},
\|\cdot\|_{Q,2}
\right)
\lesssim
K\log(a/\epsilon),
\qquad 0<\epsilon\le 1,
\]
for some $a\gtrsim e$.

To obtain the localized rate, we further combine this entropy bound with
Lemma \ref{lem:modified maximal ineq uniform entropy}, using
\eqref{eq:centered bound envelope}, \eqref{eq:centered bound Mn}, and
\eqref{eq:centered bound sigma}. For fixed $j=1,\dots,m$, define
\[
\bar f_{\theta,h}^{(m)}(\bar\varepsilon_i,\mf X_i)
=:
\mb E\big(
\bar f_{\theta,h}(\bar\varepsilon_i,\mf X_i)\mid \Upsilon_i^{(m)}
\big).
\]
Since $\{\widetilde R_{m(i-1)+j}^{(m)}(\theta,h)\}_{i=1}^L$ is i.i.d. over $i$,
Lemma \ref{lem:modified maximal ineq uniform entropy} implies
\begin{align}
&\mb E
\left\{
\sup_{(\theta,h)\in\mathcal A_n(r_\theta,r_h;a_\theta,a_h)}
\left|
\frac{1}{\sqrt{L}}
\sum_{i=1}^{L}
\left(
\bar f_{\theta,h}^{(m)}(\bar\varepsilon_{m(i-1)+j},\mf X_{m(i-1)+j})
-
\mb E\bar f_{\theta,h}^{(m)}(\bar\varepsilon_{m(i-1)+j},\mf X_{m(i-1)+j})
\right)
\right|
\right\}
\notag\\
&=
O\left(
K^{\frac12}(a_\theta+a_h)^{\frac{\eta}{2}}\log n
+
\frac{K\xi_{K,n}(a_\theta+a_h)^{\frac{\eta}{2}}}{\sqrt{L}}\log n
\right),
\label{eq:centered maximal ineq block}
\end{align}
uniformly over $j=1,\dots,m$. Therefore, combining \eqref{eq:centered remainder error for M-decomp} and \eqref{eq:centered maximal ineq block}, we have
\begin{align}
&\mb E\sup_{(\theta,h)\in\mathcal A_n(r_\theta,r_h;a_\theta,a_h)}
\left|
\frac{1}{|h|}
\big(
\widetilde{\mb R}_n(\theta,h)-\widetilde{\mb R}(\theta,h)
\big)
\right|
\notag\\
&\le
\frac1n\sum_{i=1}^n
\mb E\sup_{(\theta,h)\in\mathcal A_n(r_\theta,r_h;a_\theta,a_h)}
\frac{|\widetilde R_i(\theta,h)-\widetilde R_i^{(m)}(\theta,h)|}{|h|}
\notag\\
&\quad+
\frac1m\sum_{j=1}^m
\mb E\sup_{(\theta,h)\in\mathcal A_n(r_\theta,r_h;a_\theta,a_h)}
\left|
\frac{1}{L}
\sum_{i=1}^L
\left(
\frac{\widetilde R_{m(i-1)+j}^{(m)}(\theta,h)}{|h|}
-
\mb E\frac{\widetilde R_{m(i-1)+j}^{(m)}(\theta,h)}{|h|}
\right)
\right|
\notag\\
&=
O\left(
\xi_{K,n}\Delta_{K,n}^\alpha\chi^{\alpha m}
+
K^{\frac12}(a_\theta+a_h)^{\frac{\eta}{2}}n^{-\frac12}\log^2 n
+
K\xi_{K,n}(a_\theta+a_h)^{\frac{\eta}{2}}n^{-1}\log^2 n
\right).
\label{eq:centered final combine}
\end{align}

Finally, choose an appropriate $m\asymp \log n$, for example
$m=\frac{\log n}{\alpha\log(1/\chi)},$
with sufficiently small $0<\alpha<\min\{\omega_0/\omega_1',1\}$. Then the approximation term $\xi_{K,n}\Delta_{K,n}^\alpha\chi^{\alpha m}$ in \eqref{eq:centered final combine} is negligible, and \eqref{eq:centered remainder bound update} follows by $K^{1/2}\xi_{K,n}n^{-1/2}=o(1)$. 
\end{proof}

\begin{lemma}\label{lem:sup_bZ}
Assume \hyperref[(A1)]{(A1)}–\hyperref[(A2)]{(A2)} and \hyperref[(B1)]{(B1)}–\hyperref[(B3)]{(B3)}. Then we have:
\begin{itemize}
    \item [(i)] Recall $\mf Z_n=n^{-1/2}\sum_{i=1}^n\mf z_i$ with $\mf z_i=\psi(\bar\varepsilon_i)\mf b_\omega(\mf X_i)$ in \eqref{eq:Z_n},
    \begin{equation} 
\sup_{\mf x\in \mathcal{D}_n}\Bigl|\frac{\mf b_\omega(\mf x)^\top \mf Z_n}{\sqrt n}\Bigr|
=O_p\!\Bigl(\xi_{K,n}\,n^{-1/2}\log n\Bigr).    \label{eq:sup_bZ}
\end{equation}
\item[(ii)] Recall $\beta_{Z}^{(j)}=\bar{Q}_n^{-1}\frac{1}{M}\sum_{i=j}^{j+M-1}\mf z_i$ in \eqref{eq:blockk Z}. Suppose $M\rightarrow\infty$ and $M= O(n^{\kappa})$ for some $\kappa\in(0,1)$, then for any given $\alpha\in(0,1)$,
\begin{equation}
    \max_{1\leq j\leq n-M+1}\sup_{\mf x\in\mathcal{D}_n}\left| \mf b_\omega(\mf x)^\top \beta_Z^{(j)} \right|=O_p\left( \Delta_{K,n}^{\alpha}\xi_{K,n}M^{-1/2}\log n \right).
\end{equation}
\end{itemize}
\end{lemma}

\begin{proof} 
Define the projection kernel
\[
K_\omega(\mf x,\mf u)=:\mf b_\omega(\mf x)^\top \mf b_\omega(\mf u),\qquad \mf x,\mf u\in \mathcal{D}_n,
\]
and note that, for each $\mf x\in \mathcal{D}_n$,
\begin{equation}\label{eq:kernel_rep_bZ}
\mf b_\omega(\mf x)^\top \mf Z_n
=\frac{1}{\sqrt n}\sum_{i=1}^n \psi(\bar\varepsilon_i)\,K_\omega(\mf x,\mf X_i).
\end{equation}
\paragraph{Proof of (i).}
Set $\mathcal K=:\{\mf u\mapsto K_\omega(\mf x,\mf u):\mf x\in \mathcal{D}_n\}$ and write
\[
\mathbb G_n(f)=:\frac{1}{\sqrt n}\sum_{i=1}^n\Bigl(\psi(\varepsilon_i+a_n(\mf X_i))f(\mf X_i)-E[\psi(\varepsilon_i+a_n(\mf X_i))f(\mf X_i)]\Bigr),
\qquad f\in\mathcal K.
\]
Because $\theta_{0,n}$ is the population sieve M-target, $\mb E[\psi(\bar\varepsilon_i)K_\omega(\mf x,\mf X_i)]=\mf b_\omega(\mf x)\mb E[\psi(\bar\varepsilon_i)\mf b_\omega(\mf X_i)]=0$,
so \eqref{eq:kernel_rep_bZ} becomes
\[
\mf b_\omega(\mf x)^\top\mf Z_n=\mathbb G_n(K_\omega(\mf x,\mf X_i)),\qquad \mf x\in \mathcal{D}_n,
\]
and therefore,
\begin{equation}\label{eq:goal_sup_emp}
\sup_{\mf x\in \mathcal{D}_n}|\mf b_\omega(x)^\top\mf Z_n|
=\sup_{f\in\mathcal K}|\mathbb G_n(f)|.
\end{equation}

Let $\mf X$ be a generic copy of $\mf X_i$ and define the semi-metric
\[
d(\mf x,\mf x')=:\|K_\omega(\mf x,\mf X)-K_\omega(\mf x',\mf X)\|.
\]
Note that
\[
\|K_\omega(\mf x,\mf X)-K_\omega(\mf x',\mf X)\|
=\bigl\|(\mf b_\omega(\mf x)-\mf b_\omega(\mf x'))^\top \mf b_\omega(\mf X)\bigr\|
\le |\mf b_\omega(\mf x)-\mf b_\omega(\mf x')|\,\|\mf b_\omega(\mf X)\|.
\]
Using Assumption \hyperref[(A1)]{(A1)} and \hyperref[(A2)]{(A2)}, the mean value theorem gives 
\begin{equation}\label{eq:L2_Lip_kernel}
d(\mf x,\mf x')\le (\xi_{K,n}\vee \sqrt{K})\Delta_{K,n}|\mf x-\mf x'|,\qquad \mf x,\mf x'\in \mathcal{D}_n.
\end{equation}
Since $\mathcal{D}_n$ is compact in $\mathbb R^d$, its Euclidean covering numbers satisfy
$N(\epsilon,\mathcal{D}_n,|\cdot|)\lesssim \epsilon^{-d}$ for $0<\epsilon\le 1$.
Combining with \eqref{eq:L2_Lip_kernel} yields, for all sufficiently small $\epsilon>0$,
\begin{equation}\label{eq:covering_kernel_class}
N(\epsilon,\mathcal K,d(\cdot,\cdot))
\le N\!\Bigl(\epsilon/(\xi_{K,n}\Delta_{K,n}),\mathcal{D}_n,|\cdot|\Bigr)
\lesssim\Bigl(\frac{(\xi_{K,n}\vee \sqrt{K})\Delta_{K,n}}{\epsilon}\Bigr)^d.
\end{equation}

Let $\Upsilon_i=(\ldots,\gamma_{i-1},\gamma_i)$ with $\gamma_i=(\zeta_i,\eta_i)$ as in
Appendix~\ref{sec:proof of thm:maximal inequality} and define the
length-$m$ window $\Upsilon_i^{(m)}=:(\gamma_{i-m+1},\ldots,\gamma_i)$.
For $f\in\mathcal K$, set
\[
g_{i,f}^{(m)}=:\mb E\!\left[\psi(\bar\varepsilon_i)f(\mf X_i)\mid \Upsilon_i^{(m)}\right],
\qquad 
\mathbb G_n^{(m)}(f)=:\frac{1}{\sqrt n}\sum_{i=1}^n\Bigl(g_{i,f}^{(m)}-\mb E[g_{i,f}^{(m)}]\Bigr).
\]
Then, for each fixed $f\in\mathcal K$, in the sense that
$\sigma(g_{i,f}^{(m)}: i\le t)$ is independent of $\sigma(g_{i,f}^{(m)}: i\ge t+m)$
for every $t\ge 1$. Equivalently, the subsequences
\[
\bigl\{g_{j+\ell m,f}^{(m)}\bigr\}_{\ell\ge 0},\qquad j=1,\ldots,m,
\]
are i.i.d.\ (and hence independent across $\ell$) under stationarity.
Consequently, for each $f\in\mathcal K$ we write
\[
\sum_{i=1}^n\bigl(g_{i,f}^{(m)}-\mb E[g_{i,f}^{(m)}]\bigr)
=\sum_{j=1}^m \sum_{\ell:\, j+\ell m\le n}\bigl(g_{j+\ell m,f}^{(m)}-\mb E[g_{j+\ell m,f}^{(m)}]\bigr).
\]
 By the same telescoping and physical-dependence argument used in Appendix~\ref{sec:proof of thm:maximal inequality} (e.g., the derivation around \eqref{eq:loss error for M-decomp}),
Assumption \hyperref[(B2)]{(B2)} implies that there exists $\chi\in(0,1)$ such that for any given $\alpha\in(0,1)$, 
\begin{equation}\label{eq:dep_approx_err}
\mb E\Bigl[\sup_{f\in\mathcal K}|\mathbb G_n(f)-\mathbb G_n^{(m)}(f)|\Bigr]
\lesssim \xi_{K,n}\Delta_{K,n}^{\alpha}\chi^{\alpha m}.
\end{equation}
Choosing appropriate $m\asymp \log n$ makes the right-hand side $o(\xi_{K,n}n^{-1/2})$ under Assumption \hyperref[(A1)]{(A1)}.

For each $j\in\{1,\ldots,m\}$ and write $I_j=:\{j,j+m,j+2m,\ldots\}$, with $|I_j|\asymp n/m$.
Conditioning on the $\sigma$-field generated by $\{\Upsilon_i^{(m)}: i\in I_j\}$,
the summands in $\mathbb G_n^{(m)}(f)$ over $i\in I_j$ are independent.
Consider the class
\[
\mathcal F_j=:\Bigl\{(\varepsilon,\mf x)\mapsto \psi(\varepsilon)K_\omega(\mf x_0,\mf x): \mf x_0\in \mathcal{D}_n\Bigr\}
\]
where the covering number is bounded by \eqref{eq:covering_kernel_class}.
Note that $|K_\omega(\mf x,\mf u)|\le \xi_{K,n}^2$, thus we have the envelope $F_i=|\psi(\bar \varepsilon_i)|\,\xi_{K,n}^2$ 
\begin{equation}
\| \max_{1\le i\le n}F_i\|\le \xi_{K,n}^2\,\|\max_{1\le i\le n}|\psi(\bar\varepsilon_i)|\|_2\leq \xi_{K,n}^2\,\left(\sum_{i=1}^n\mb E|\psi(\bar\varepsilon_i)|^4\right)^{1/4}\lesssim\xi_{K,n}^2\,n^{1/4}\|\psi(\bar\varepsilon_i)\|_4,    \label{eq:max envelope bound} 
\end{equation}
where the inequality follows from standard moment bounds for maxima and stationarity. For the variance proxy, by \hyperref[(A2)]{(A2)} and $\|\psi(\bar\varepsilon_i)\|_4<\infty$ from \hyperref[(B2)]{(B2)},
\begin{equation}
\sup_{\mf x}\mb E\bigl[\psi(\bar\varepsilon_i)^2K_\omega(\mf x,\mf X_i)^2\bigr]
=\sup_{\mf x}\mb E\!\Bigl[\psi(\bar \varepsilon_i)^2\,\mf b_\omega(\mf x)^\top \mf B_K \mf b_\omega(\mf x)\Bigr]
\lesssim \xi_{K,n}^2,   \label{eq:variance proxy} 
\end{equation}
where $\mf B_K=E[\mf b_\omega(\mf X_i)\mf b_\omega(\mf X_i)^\top]$ and $\lambda_{\max}(B_K)\le C$ by Assumption \hyperref[(A2)]{(A2)}.

Therefore, Lemma~\ref{lem:modified maximal ineq uniform entropy} combined with \eqref{eq:covering_kernel_class}, \eqref{eq:max envelope bound}, and \eqref{eq:variance proxy} yields
\begin{equation}\label{eq:iid_block_bound}
E\Bigl[\sup_{f\in\mathcal K}|\mathbb G_n^{(m)}(f)|\Bigr]
\lesssim \xi_{K,n}\sqrt{\log n}+\xi_{K,n}^2\,n^{-1/4}\log n.
\end{equation}
Combine \eqref{eq:goal_sup_emp} and the decomposition
\[
\sup_{f\in\mathcal K}|\mathbb G_n(f)|
\le \sup_{f\in\mathcal K}|\mathbb G_n(f)-\mathbb G_n^{(m)}(f)|
+\sup_{f\in\mathcal K}|\mathbb G_n^{(m)}(f)|,
\]
the dependence approximation bound \eqref{eq:dep_approx_err}, and \eqref{eq:iid_block_bound}
with $m\asymp \log n$. Since the second term
in \eqref{eq:iid_block_bound} is negligible compared to $\xi_{K,n}\log n$, we finally have
\[
\sup_{\mf x\in \mathcal{D}_n}|\mf b_\omega(\mf x)^\top \mf Z_n|
=\sup_{f\in\mathcal K}|\mathbb G_n(f)|
=O_p(\xi_{K,n}\log n),
\]
which yields \eqref{eq:sup_bZ}.

\paragraph{Proof of (ii).} By the definition of $\beta_{Z}^{(j)}$ in \eqref{eq:blockk Z}, we can write 
\[
\mf b_\omega(\mf x)^\top \beta_Z^{(j)} = \mf b_\omega(\mf x)^\top \bar{Q}_n^{-1} \frac{1}{M}\sum_{i=j}^{j+M-1} \mf z_i.
\]
By Assumption \hyperref[(A2)]{(A2)} and \hyperref[(B3)]{(B3)}, the smallest eigenvalues of $\bar{Q}_n$ is bounded above zero. Consequently, it suffices to bound the maximal uniform deviation of the block process
\[
\max_{1\le j\le n-M+1} \sup_{f\in\mathcal{K}} \left| \frac{1}{M}\sum_{i=j}^{j+M-1} \psi(\bar\varepsilon_i)f(\mf X_i) \right|.
\]
We inherit the $m$-dependent projection $g_{i,f}^{(m)} = \mb E[\psi(\bar\varepsilon_i)f(\mf X_i) \mid \Upsilon_i^{(m)}]$ and the window size $m \asymp \log n$ established in the proof of (i). 

Following the block-splitting notation used in Lemma \ref{lem:max beta-j}, for each rolling block $j \in \{1,\ldots,T\}$, we partition the index set $\{j, \ldots, j+M-1\}$ into $m$ interleaved subsets. For each $k \in \{1, \ldots, m\}$, define the index sequence $I_{j,k} = \bigl\{ j + k - 1 + \ell m : 0 \le \ell \le \lfloor (M-k)/m \rfloor \bigr\}.$
The cardinality of each subset $|I_{j,k}| \asymp M/m$. Crucially, for any fixed $j$ and $k$, the elements in $\{g_{t,f}^{(m)}\}_{t \in I_{j,k}}$ are separated by exactly $m$ time steps. Since $g_{t,f}^{(m)}$ depends only on $\Upsilon_t^{(m)}$, the sequence consists of strictly independent (and identically distributed under stationarity) functional variables indexed by $f \in \mathcal{K}$. 
For the fixed independent subset $I_{j,k}$, inheriting the expected supremum bound evaluated via the metric entropy integral in \eqref{eq:iid_block_bound}, we have 
\begin{equation}
\mb E\left[\sup_{f\in\mathcal{K}} |S_{j,k}(f)| \right]\leq C\sqrt{\frac{M}{m}}\xi_{K,n}\log n,    \label{eq:bound sup_S E}
\end{equation}
where the constant $C>0$ is independent of $j$ and $k$ due to stationarity.

Applying Lemma \ref{lem:Bousquet's inequality} with the result in \eqref{eq:variance proxy} and the fact that $\|F_i\|\lesssim \xi_{K,n}^2$, we have for any $t > 0$,
\begin{equation}
\mr P\bigg(\sup_{f\in\mathcal{K}} |S_{j,k}(f)|\ge \mb E[\sup_{f\in\mathcal{K}} |S_{j,k}(f)|] + t\bigg) \le \exp\left( -\frac{t^2}{2(\sigma_m^2 + 2U\mb E[\sup_{f\in\mathcal{K}} |S_{j,k}(f)|] + Ut/3)} \right),\label{eq:max bousquet beta_Z}
\end{equation}
where $U\lesssim \xi_{K,n}^2$ and $\sigma_m^2\lesssim M/m\xi_{K,n}^2$.

Note that
\begin{align}
&\mr P\bigg( \max_{1\le j\le n-M+1} \sup_{f\in\mathcal{K}} \frac{1}{M} \bigg| \sum_{i=j}^{j+M-1} \psi(\bar\varepsilon_i)f(\mf X_i) \bigg| > x \bigg) \notag \\
\le~& \mr P\bigg( \max_{1\le j\le n-M+1} \sup_{f\in\mathcal{K}} \frac{1}{M} \bigg| \sum_{i=j}^{j+M-1} \bigl(\psi(\bar\varepsilon_i)f(\mf X_i) - g_{i,f}^{(m)}\bigr) \bigg| > \frac{x}{2} \bigg)  + \mr P\bigg( \max_{1\le j\le n-M+1} \sup_{f\in\mathcal{K}} \frac{1}{M} \bigg| \sum_{k=1}^m S_{j,k}(f) \bigg| > \frac{x}{2} \bigg), \notag \\
\le~& \frac{2}{x M} \mb E\bigg[ \max_{1\le j\le n-M+1} \sup_{f\in\mathcal{K}} \bigg| \sum_{i=j}^{j+M-1} \bigl(\psi(\bar\varepsilon_i)f(\mf X_i) - g_{i,f}^{(m)}\bigr) \bigg| \bigg] + \sum_{j=1}^{n-M+1} \sum_{k=1}^m \mr P\bigg( \sup_{f\in\mathcal{K}} |S_{j,k}(f)| > \frac{Mx}{2m} \bigg), \label{eq:prob_decomposition}
\end{align}
where the first term in \eqref{eq:prob_decomposition} is bounded via Markov inequality and the second term is bounded by the union bound over all overlapping blocks and independent sub-sequences.

To bound the $m$-decomposition error in the first term of \eqref{eq:prob_decomposition}, applying the dependence approximation result in \eqref{eq:dep_approx_err} and the subadditivity of the expected maximum over $n-M+1 \le n$ blocks, we obtain
\begin{equation}
\mb E\bigg[ \max_{1\le j\le n-M+1} \sup_{f\in\mathcal{K}} \bigg| \sum_{i=j}^{j+M-1} \bigl(\psi(\bar\varepsilon_i)f(\mf X_i) - g_{i,f}^{(m)}\bigr) \bigg| \bigg] \lesssim nM \xi_{K,n}\Delta_{K,n}^{\alpha}\chi^{\alpha m}.\label{eq:decomp_error_bound}
\end{equation}
For the independent maxima in the second term of \eqref{eq:prob_decomposition}, we evaluate the tail probability at the target rate $x = C^* \Delta_{K,n}^{\alpha}\xi_{K,n}M^{-1/2}\log n$ for a sufficiently large constant $C^*>0$. This choice yields the threshold
\[
\frac{Mx}{2m} = \frac{C^*}{2m} \Delta_{K,n}^{\alpha}\xi_{K,n}M^{1/2}\log n.
\]
Choosing $m \asymp \log n$, it strictly dominates the expected supremum $2\mb E[\sup_{f\in\mathcal{K}} |S_{j,k}(f)|]$ by \eqref{eq:bound sup_S E} for sufficiently large $n$. Setting the deviation $t = \frac{Mx}{2m} - \mb E[\sup_{f\in\mathcal{K}} |S_{j,k}(f)|] \asymp \frac{Mx}{m}$ in \eqref{eq:max bousquet beta_Z}, using the fact $U\lesssim \xi_{K,n}^2$ and $\sigma_m^2\lesssim M/m\xi_{K,n}^2$, we have
\begin{align}\label{eq:bousquet_bound}
\mr P\bigg(\sup_{f\in\mathcal{K}} |S_{j,k}(f)| > \frac{Mx}{2m}\bigg) &\le \exp\left( - \frac{t^2}{2(\sigma_m^2 + 2U\mb E[\sup_{f\in\mathcal{K}} |S_{j,k}(f)|] + Ut/3)} \right) \notag\\
&\le \exp\left( - c_1 \min\bigg\{ \frac{t^2}{\sigma_m^2}, \, \frac{t}{U} \bigg\} \right) \le \exp\left( - c_2 \Delta_{K,n}^{2\alpha} \log n \right),
\end{align}
for some universal positive constants $c_1, c_2$. The last inequality holds because the heavy-tail envelope term $Ut$ dominates the denominator, utilizing $U \lesssim \xi_{K,n}^2$ and $\sigma_m^2 \lesssim (M/m)\xi_{K,n}^2$.

Combining \eqref{eq:decomp_error_bound} and \eqref{eq:bousquet_bound} into \eqref{eq:prob_decomposition}, the total target probability
\begin{align}
    &\mr P\bigg( \max_{1\le j\le n-M+1} \sup_{f\in\mathcal{K}} \frac{1}{M} \bigg| \sum_{i=j}^{j+M-1} \psi(\bar\varepsilon_i)f(\mf X_i) \bigg| > C^* \Delta_{K,n}^{\alpha}\xi_{K,n}M^{-1/2}\log n  \bigg)\notag\\
    &\lesssim \frac{n}{\Delta_{K,n}^{\alpha}\xi_{K,n}M^{-1/2}\log n} \xi_{K,n}\Delta_{K,n}^{\alpha}\chi^{\alpha m} + n m \exp(-c_2 \Delta_{K,n}^{\alpha}\log n).
\end{align}
Finally, choosing an appropriate $m\asymp \log n$ (e.g., 
$m=\frac{(1+\kappa/2)\log n}{\alpha\log(1/\chi)}$), both the $m$-decomposition error probability and the union bound tail probability $n m \exp(-c_2 \Delta_{K,n}^{\alpha}\log n)$ vanish to $o(1)$ by Assumption \hyperref[(A1)]{(A1)}. This concludes the proof of (ii).
\end{proof}

\subsection{Lemmas for dependence measure and moment results}
\label{sec:lemmas for physical dependence measure}
The following Lemmas show the physical dependence measure and moment results on loss functions of quantile, Huber's, expectile, and $\mathcal{L}^q$ ($1<q\leq 2$) regressions. Recall $\mf X_{i,k}=\mf H(\mathcal{H}_{i,k}),\varepsilon_{i,k}=G(\mathcal{H}_{i,k},\mathcal{G}_{i,k})$ where $\mathcal{H}_{i,k}=(\mathcal{H}_{i-k-1},\zeta_{i-k}^\prime,\zeta_{i-k+1},\dots,\zeta_i)$, $\mathcal{G}_{i,k}=(\mathcal{G}_{i-k-1},\eta_{i-k}^\prime,\eta_{i-k+1},\dots,\eta_i)$.

\begin{lemma}
    \label{lem:phd on bx} 
    For $q\geq 1$, denote
    \begin{equation}
        \delta_{\mf b}(k,q)=:\left\|\mf b_\omega(\mf X_i)-\mf b_\omega(\mf X_{i,k})\right\|_q. \label{eq:phd on bx def}
    \end{equation}
    If Assumptions \hyperref[(A1)]{(A1)}-\hyperref[(A4)]{(A4)} and \hyperref[(B1)]{(B1)}-\hyperref[(B3)]{(B3)} hold, for some constant $\chi\in(0,1)$, for any given $\alpha\in(0,1)$,
    \begin{equation}
        \label{eq:phd on bx}
         \delta_{\mf b}(k,q)= O(q\xi_{K,n}\Delta_{K,n}^\alpha\chi^{\alpha k}).
    \end{equation}
\end{lemma}

\begin{proof}
    For $q\geq 1$, by H\"older's inequality
    \begin{align}
        \left\| \mf X_i-\mf X_{i,k}\right\|_q &=\left\| |\mf X_i-\mf X_{i,k}|^{1-1/2q}|\mf X_i-\mf X_{i,k}|^{1/2q}  \right\|_q,\notag\\
        &\leq \left\| |\mf X_i-\mf X_{i,k}|^{1-1/2q}\right\|_{2q} \left\| |\mf X_i-\mf X_{i,k}|^{1/2q}  \right\|_{2q},\notag\\
        &\leq \|\mf X_i-\mf X_{i,k}\|_{2q-1}^{(2q-1)/2q}\|\mf X_i-\mf X_{i,k}\|_{1}^{1/2q},\notag\\
        &=O(q^{\frac{2q-1}{2q}}\chi^{k/2q}),\label{eq:phd on x}
    \end{align}
    where the last line is obtained from the moment property of sub-exponential random vectors (see Proposition E.4 of \cite{wuzhou2023multiscale}), i.e. $\|\mf X_i\|_q\lesssim q$, $\forall q\geq 1$ and $\delta_{\mf H}(k,1)=O(\chi^k)$ for some constant $\chi\in(0,1)$ in Assumption \hyperref[(B1)]{(B1)}. By \eqref{eq:phd on x} and Assumptions \hyperref[(A1)]{(A1)}, 
    \begin{equation}
        \left\|\mf b_\omega(\mf X_i)-\mf b_\omega(\mf X_{i,k})\right\|_q= O(\Delta_{K,n}\chi^{k/2q}).\label{eq:phd on bx q}
    \end{equation}
    On the other hand
    \begin{equation}
           \left\|\mf b_\omega(\mf X_i)-\mf b_\omega(\mf X_{i,k})\right\|_q\leq 2\sup_{\mf x\in\mathcal{X}}|\mf b_\omega(\mf x)|=O(\xi_{K,n}).\label{eq:phd sup bx}
    \end{equation}
    Using the fact $\min\{|x|,1\}\leq |x|^\alpha$ for any given $\alpha\in[0,1]$, \eqref{eq:phd on bx} holds by combining \eqref{eq:phd sup bx} and \eqref{eq:phd on bx q}. 
\end{proof}

\begin{lemma}
    \label{lem:phd on z} 
    For $q\geq 2$, denote
    \begin{equation}
        \delta_{\mf z}(k,q)=:\left\|\psi(\bar \varepsilon_i)\mf b_\omega(\mf X_i)-\psi(\bar \varepsilon_{i,k})\mf b_\omega(\mf X_{i,k})\right\|_q.\label{eq:phd on z def}
    \end{equation}
    Under Assumptions \hyperref[(A1)]{(A1)}-\hyperref[(A4)]{(A4)} and \hyperref[(B1)]{(B1)}-\hyperref[(B3)]{(B3)}, then for any given $q\in[2,4)$ and $\alpha\in(0,1)$,
    \begin{equation}
        \label{eq:phd on z-4}
        \delta_{\mf z}(k,q)= O(\xi_{K,n}\Delta_{K,n}^\alpha\chi^{\alpha k}),
    \end{equation}
\end{lemma}
\begin{proof}
    
    By Assumption \hyperref[(A1)]{(A1)} and $\delta_\psi(k,q)\leq \delta_\psi(k,4)=O(\chi^k)$ with $q<4$, we have $\|\psi(\bar\varepsilon_i)\|_4<\infty$ and
    $$
        \delta_{\mf z}(k,q)\leq 2\|\psi(\bar\varepsilon_i)\|_4\xi_{K,n}=O(\xi_{K,n}).
    $$
    On the other hand, by Lemma \ref{lem:phd on bx} and Assumption \hyperref[(B2)]{(B2)}, H\"older inequality yields
    \begin{align}
        \delta_{\mf z}(k,q)&\leq\left\|[\psi(\bar\varepsilon_i)-\psi(\bar\varepsilon_{i,k})]\mf b_\omega(\mf X_i)\right\|_q+\left\|[\psi(\bar\varepsilon_i)|\mf b_\omega(\mf X_i)-\mf b_\omega(\mf X_{i,k})|\right\|_q\notag\\
        &\leq \xi_{K,n}\delta_{\psi}(k,4)+\|\psi(\bar\varepsilon_i)\|_4 \delta_{\mf b}\left(k,\frac{4q}{4-q}\right),\notag\\
        &=O(\xi_{K,n} \Delta_{K,n}^{\alpha}\chi^{\alpha k} ),
    \end{align}
    for any given $q\in[2,4)$ and $\alpha\in(0,1)$.

\end{proof}

\begin{lemma}
\label{lem:phd on rho}
    If Assumptions \hyperref[(A1)]{(A1)}-\hyperref[(A4)]{(A4)} and \hyperref[(B1)]{(B1)}-\hyperref[(B3)]{(B3)} hold, for some constant $\chi\in(0,1)$ and any given $\alpha\in(0,1)$,Recall
\[
L_i(\beta)
=
\rho\!\left(\bar\varepsilon_i-\beta^\top \mf b_\omega(\mf X_i)\right)-\rho(\bar\varepsilon_i),\quad
L_{i,k}(\beta)
=
\rho\!\left(\bar\varepsilon_{i,k}-\beta^\top\mf b_\omega(\mf X_{i,k})\right)-\rho(\bar\varepsilon_{i,k}),
\]
Then for any fixed $\alpha\in(0,1)$ and $r_n\xi_{K,n}=o(1)$,
\[
\mb E\sup_{|\beta|\le r_n}\big|L_i(\beta)-L_{i,k}(\beta)\big|
=
O\!\left(r_n\,\xi_{K,n}\Delta_{K,n}^{\alpha}\chi^{\alpha k}\right).
\]
\end{lemma}

\begin{proof}
By the fundamental theorem of calculus,
\[
L_i(\beta)
=
\mf b_\omega(\mf X_i)\int_0^1
\Big\{
\psi(\bar\varepsilon_i)-\psi(\bar\varepsilon_i-t\,\mf b_\omega(\mf X_i))
\Big\}\,dt,
\]
\[
L_{i,k}(\beta)
=
\mf b_\omega(\mf X_{i,k})\int_0^1
\Big\{
\psi(\bar\varepsilon_{i,k})-\psi(\bar\varepsilon_{i,k}-t\,\mf b_\omega(\mf X_{i,k}))
\Big\}\,dt.
\]
Hence
\[
L_i(\beta)-L_{i,k}(\beta)=A_{1,i}(\beta)+A_{2,i}(\beta),
\]
where
\[
A_{1,i}(\beta)
=
\bigl(\mf b_\omega(\mf X_i)-\mf b_\omega(\mf X_{i,k})\bigr)
\int_0^1
\Big\{
\psi(\bar\varepsilon_i)-\psi(\bar\varepsilon_i-t\,\mf b_\omega(\mf X_i))
\Big\}\,dt,
\]
\[
A_{2,i}(\beta)
=
\mf b_\omega(\mf X_{i,k})\int_0^1
\Big[
\psi(\bar\varepsilon_i)-\psi(\bar\varepsilon_i-t\,\mf b_\omega(\mf X_i))
-\psi(\bar\varepsilon_{i,k})+\psi(\bar\varepsilon_{i,k}-t\,\mf b_\omega(\mf X_{i,k}))
\Big]dt.
\]
We first bound $A_{1,i}(\beta)$.
By Assumption \hyperref[(A1)]{(A1)}, \hyperref[(B3)]{(B3)}, and Lemma \ref{lem:phd on bx}, we have
\[
\mb E\sup_{|\beta|\le r_n}|A_{1,i}(\beta)|
\lesssim
\mb E\sup_{|\beta|\le r_n}|\mf b_\omega(\mf X_i)-\mf b_\omega(\mf X_{i,k})|
\lesssim
r_n\,\xi_{K,n}\Delta_{K,n}^{\alpha}\chi^{\alpha k}.
\]
Next consider $A_{2,i}(\beta)$. Notice that
\begin{align*}
&\psi(\bar\varepsilon_i)-\psi(\bar\varepsilon_i-t\,\mf b_\omega(\mf X_i))
-\psi(\bar\varepsilon_{i,k})+\psi(\bar\varepsilon_{i,k}-t\,\mf b_\omega(\mf X_{i,k})) \\
&=
\bigl[\psi(\bar\varepsilon_i)-\psi(\bar\varepsilon_{i,k})\bigr]
-
\bigl[\psi(\bar\varepsilon_i-t\,\mf b_\omega(\mf X_i))-\psi(\bar\varepsilon_{i,k}-t\,\mf b_\omega(\mf X_{i,k}))\bigr].
\end{align*}
Insert and subtract $\psi(\bar\varepsilon_{i,k}-t\,\mf b_\omega(\mf X_i))$ to obtain
\begin{align*}
&\bigl[\psi(\bar\varepsilon_i)-\psi(\bar\varepsilon_{i,k})\bigr]
-
\bigl[\psi(\bar\varepsilon_i-t\,\mf b_\omega(\mf X_i))-\psi(\bar\varepsilon_{i,k}-t\,\mf b_\omega(\mf X_{i,k}))\bigr] \\
&=
\Big(
\psi(\bar\varepsilon_i)-\psi(\bar\varepsilon_{i,k})
-\psi(\bar\varepsilon_i-t\,\mf b_\omega(\mf X_i))
+\psi(\bar\varepsilon_{i,k}-t\,\mf b_\omega(\mf X_i))
\Big) \\
&\quad+
\Big(
\psi(\bar\varepsilon_{i,k}-t\,\mf b_\omega(\mf X_i))
-\psi(\bar\varepsilon_{i,k}-t\,\mf b_\omega(\mf X_{i,k}))
\Big).
\end{align*}
For the first term above, Assumption \hyperref[(B2)]{(B2)} gives 
$\psi(\bar\varepsilon_i)-\psi(\bar\varepsilon_{i,k})$, while Assumption \hyperref[(B3)]{(B3)} controls the change
of the score under the local perturbation $t\,\mf b_\omega(\mf X_i)$, uniformly for $|\beta|\le r_n$. Using H\"older's inequality together with $\sup_{|\beta|\le r_n}|\mf b_\omega(\mf X_{i,k})|\lesssim r_n\xi_{K,n}$,
we obtain
\[
\mb E\sup_{|\beta|\le r_n}|A_{2,i}(\beta)|
\lesssim
r_n\,\xi_{K,n}\Delta_{K,n}^{\alpha}\chi^{\alpha k}
+
r_n^2\xi_{K,n}^2 (r_n\xi_{K,n})^{\eta/2}\chi^{\alpha k}.
\]
Combining the bounds for $A_{1,i}(\beta)$ and $A_{2,i}(\beta)$, using the fact $r_n\xi_{K,n}=o(1)$ and $\eta\in[1,2]$, we finally have
\[
\mb E\sup_{|\beta|\le r_n}|L_i(\beta)-L_{i,k}(\beta)|
=
O\!\left(
r_n\,\xi_{K,n}\Delta_{K,n}^{\alpha}\chi^{\alpha k}
+
r_n^2\xi_{K,n}^2 (r_n\xi_{K,n})^{\eta/2}\chi^{\alpha k}
\right)=O(r_n\,\xi_{K,n}\Delta_{K,n}^{\alpha}\chi^{\alpha k}).
\]
\end{proof}

\subsection{Lemmas for convex Gaussian approximation}

\begin{lemma}[Theorem 2.1 in \cite{Mies2022seq_high-dim}]
    \label{lem:strong GA in mies}
 Let $X_1, \ldots, X_n$ be independent, centered, $d$-variate random vectors which admit the bound $\left\|X_t\right\|_q \leq b_t, t=1, \ldots, n$, for some $q>2$. On a different probability space, there exist independent random vectors $\tilde{X}_t \stackrel{d}{=} X_t$, and independent Gaussian random vectors $Y_t \sim \mathcal{N}\left(0, \operatorname{Cov}\left(X_t\right)\right)$, such that for some universal $C>0$,
$$
\left\|\frac{1}{\sqrt{n}} \sum_{t=1}^n\left(\tilde{X}_t-Y_t\right)\right\| \leq \frac{C}{\sqrt{q-2} \wedge 1}\left(\frac{d}{n}\right)^{\frac{1}{2}-\frac{1}{q}} \sqrt{\log (n)} \sqrt{\frac{1}{n} \sum_{t=1}^n b_t^2} .
$$
\end{lemma}

\begin{lemma}\label{lem:basic_smoothing}
Let $U$ and $V$ be random vectors in $\mathbb{R}^K$.
For any Borel set $A \subset \mathbb{R}^K$ and any $\varepsilon>0$,
define the $\varepsilon$-enlargement of $A$ by
\[
A^{\varepsilon}
=: \bigl\{ \mf x\in\mathbb{R}^K :\operatorname{dist(\mf x,A)}\leq \varepsilon
 \bigr\},
\]
where $\operatorname{dist(\mf x,A)}=:\inf_{\mf a\in A} |\mf x-\mf  a|$.
Then for any $\varepsilon>0$,
\[
\mr{P}(U \in A)
\le
\mr{P}(V \in A^{\varepsilon})
+
\mr{P}(|U-V| \ge \varepsilon).
\]
Moreover, suppose $G$ is a nondegenerate Gaussian random vector, then we have for any $\varepsilon>0$,
\begin{equation}
    \mathcal{K}(U,G)\lesssim \mathcal{K}(V,G)+\mr P(|U-V|\geq \varepsilon)+K^{1/4}\varepsilon  \label{eq:basic smoothing}.
\end{equation}
\end{lemma}

\begin{proof}[Proof.]
Note that for any Borel set $A$,
\[
\{U\in A\}
\subset
\{V\in A^{\varepsilon}\}
\cup
\{|U-V|\ge \varepsilon\}.
\]
Suppose $U(\omega)\in A$ and $|U(\omega)-V(\omega)|<\varepsilon$, then
\[
\operatorname{dist}\bigl(V(\omega),A\bigr)
\le
|V(\omega)-U(\omega)|
<
\varepsilon,
\]
which implies $V(\omega)\in A^{\varepsilon}$.
Therefore, whenever $U\in A$ but $V\notin A^{\varepsilon}$,
it must hold that $|U-V|\ge \varepsilon$.
Taking probabilities yields
\[
\mr{P}(U \in A)
\le
\mr{P}(V \in A^{\varepsilon})
+
\mr{P}(|U-V| \ge \varepsilon).
\]
Moreover, for any convex set $A\in\mb R^K$, we further have
\begin{align}
    \mr{P}(U \in A)-\mr{P}(G \in A)
&\leq \mr{P}(V \in A^{\varepsilon})-\mr{P}(G \in A)+\mr{P}(|U-V| \ge \varepsilon),\notag\\
&=\mr{P}(V \in A^{\varepsilon})-\mr{P}(G \in A^\epsilon)+\mr P(G\in A^\varepsilon\setminus A)+\mr{P}(|U-V| \ge \varepsilon).\notag
\end{align}

Notice that $A^{\varepsilon}$ is also convex, taking the supremum over the class of all convex sets $\mathcal{A}$, we have
\begin{equation}
\mathcal{K}(U,G)\leq \mathcal{K}(V,G)+\mr{P}(|U-V| \ge \varepsilon)+\sup_{A\in \mathcal{A}}\mr P(G\in A^\varepsilon\setminus A).
\end{equation}
Note that $G$ is a nondegenerate Gaussian random vector, by Remark 2.6 of \cite{chen2011multivariate} (see also \cite{bentkus2003Essen-bound} and \cite{ball1993reverse}), $\sup_{A\in \mathcal{A}}\mr P(G\in A^\varepsilon\setminus A)=O(K^{1/4}\varepsilon)$, which yields \eqref{eq:basic smoothing}.

\end{proof}

\begin{lemma}[Lemma 5.3 in \cite{Fang2024Larege-dim}]
\label{lem:Fang2024}
For any d-dimensional random vector $\mf W$, for any $\epsilon>0$,
$$
\sup_{A\in\mathcal{A}}\left|\mr P\left(\mf W\in A\right)-\mr P\left(\mf Z\in A\right)\right| \leq 4 d^{1 / 4} \epsilon+\sup _{A \in \mathcal{A}}\left|\mb E h_{A, \epsilon}(\mf W)-\mb E h_{A, \epsilon}(\mf Z)\right| ,
$$
where $\mf Z$ is a d-dimensional Gaussian random vector with invertible covariance matrix and $\mathcal{A}$ is the collection of all the convex sets in $\mb{R}^d$.
\end{lemma}

\begin{lemma}[Lemma B.1 in \cite{liu2025wasserstein}] 
\label{lem:liu2025}
Suppose $\operatorname{Cov}(\mf W)= I_d$ and $\mf Z$ is a standard $d$-dimensional Gaussian vector. For any $\epsilon>0$, we have
$$
\sup_{A\in\mathcal{A}}\left|\mr P\left(\mf W\in A\right)-\mr P\left(\mf Z\in A\right)\right|\leq 4 d^{1 / 4} \epsilon+\min \left\{\sup _{A \in \mathcal{A}}\left|\mathbb{E}\left[h_{A, \epsilon}(\mf W)-h_{A, \epsilon}(\mf Z)\right]\right|, \mr P(|\mf W-\mf Z|>\epsilon)\right\},
$$
$\mathcal{A}$ is the collection of all the convex sets in $\mb{R}^d$.
\end{lemma}

\begin{lemma}[Remark 2.2 in \cite{Fang2016CLT}]
\label{lem:Fang2016}
Let $\mf{W}=\sum_{i=1}^n \mf{X}_i$ be a sum of d-dimensional random vectors such that $\mb E\left(\mf{X}_i\right)=0$ and $\operatorname{Cov}(\mf{W})=\Sigma_w$. Suppose $\mf{W}$ can be decomposed as follows:
\begin{description}
    \item[1.] $\forall i \in[n], \exists N_i \subset[n]$, such that $i\in N_i$ and $\mf{W}-\mf{X}_{N_i}$ is independent of $\mf{X}_i$, where $[n]=$ $\{1, \cdots, n\}$.
    \item[2.] $\forall i \in[n], j \in N_i, \exists N_{i j} \subset[n]$, such that $N_i \subset N_{ij}$ and $\mf{W}-\mf{X}_{N_{i j}}$ is independent of $\left\{\mf{X}_i, \mf{X}_j\right\}$.
    \item[3.] $\forall i \in[n], j \in N_i, k \in N_{i j}, \exists N_{i j k} \subset[n]$ such that $N_{i j} \subset N_{ijk}$ and $\mf{W}-\mf{X}_{N_{i j k}}$ is independent of $\left\{\mf{X}_i, \mf{X}_j, \mf{X}_k\right\}$.
\end{description}
Suppose further that for each $i \in[n], j \in N_i, k \in N_{i j},\left|\mf{X}_i\right| \leqslant \kappa,\left|N_i\right| \leqslant n_1,\left|N_{i j}\right| \leqslant$ $n_2,\left|N_{i j k}\right| \leqslant n_3$. Then there exists a universal constant $C$ such that
$$
\mathcal{K}\left(\mf{W}, \Sigma_w^{1 / 2} \boldsymbol{Z}\right) \leqslant C d^{1 / 4} n\left|\Sigma_w^{-1 / 2}\right|^3 \kappa^3 n_1\left(n_2+\frac{n_3}{d}\right),
$$
where $\boldsymbol{Z}$ is a d-dimensional standard Gaussian random vector.
\end{lemma}

\subsection{Lemmas for covariance estimation}

\begin{lemma}
\label{lem: cov comparison-element}
    Under Assumptions \hyperref[(A1)]{(A1)}-\hyperref[(A4)]{(A4)} and \hyperref[(B1)]{(B1)}-\hyperref[(B3)]{(B3)}, we have for any given $\alpha\in(0,1)$,
    \begin{equation}
        \left\|\tilde{\boldsymbol{\Sigma}}_M-\mb E\tilde{\boldsymbol{\Sigma}}_M\right\|_F=O_p\left(\xi_{K,n}^2\Delta_{K,n}^{\alpha}\sqrt{\frac{M}{n}} \right).\label{eq: covariance comparison-element}
    \end{equation}
\end{lemma}
\begin{proof}
     Recall $\mf W(j)=\frac{1}{2\sqrt{M}p(0)}\sum_{i=j}^{j+M-1} \mf z_{i}$, $\mf z_i=\psi(\bar\varepsilon_i)\mf b_\omega(\mf X_i)$, the causal representation \eqref{eq:causal representation}. Denote 
     $$
        \mf W_{\{j-l\}}(j)=\frac{1}{2\sqrt{M}p(0)}\sum_{i=j}^{j+M-1} \psi(\varepsilon_{i,\{j-l\}})\mf b_\omega(\mf X_{i,\{j-l\}}),
     $$
     where $\mf X_{i,\{k\}}=\mf H(\mathcal{H}_{i,\{k\}})$, $\varepsilon_{i,\{k\}}=G(\mathcal{H}_{i,\{k\}})$, and $\mathcal{H}_{i,\{k\}}=(\dots,\zeta_{k-1},\zeta^\prime_{k},\zeta_{k+1},\dots,\zeta_i)$, $\mathcal{G}_{i,\{k\}}=(\dots,\eta_{i-l-1},\eta^\prime_{k},\eta_{k+1},\dots,\eta_i)$ (If $k>i$, $\mathcal{H}_{i,\{k\}}=\mathcal{H}_i$, $\mathcal{G}_{i,\{k\}}=\mathcal{G}_i$). By \eqref{eq:phd on z-4} in Lemma \ref{lem:phd on z}, we have for any given $\alpha\in(0,1)$,
    \begin{align}
        \|\mf W(j)-\mf W_{\{j-l\}}(j)\|_4\lesssim \frac{1}{\sqrt{M}}\sum_{i=j}^{j+M-1}\delta_{\mf z}(i-(j-l),4)=O\left(\frac{1}{\sqrt{M}} \xi_{K,n}\Delta_{K,n}^{\alpha} \sum_{k=l}^{l+M-1}\chi^{\alpha k}\right).\label{eq:W norm4}
    \end{align}
    Then we have 
    \begin{align}
        &\left(\mb E\left\|\mf W(j)\mf W(j)^\top-\mf W_{\{j-l\}}(j)\mf W_{\{j-l\}}(j)^\top\right\|_F^2\right)^{1/2}\notag\\ 
        &\leq \left(\mb E\left\|\mf W(i)\left[\mf W(i)-\mf W_{\{j-l\}}(i)\right]^\top\right\|_F^2\right)^{1/2}+\left(\mb E\left\|\left[\mf W(j)-\mf W_{\{j-l\}}(j)\right]\mf W_{\{j-l\}}(j)^\top\right\|_F^2\right)^{1/2}  \notag\\
        &\leq 2 \|\mf W(j)\|_4 \|\mf W(j)-\mf W_{\{j-l\}}(j)\|_4
        .\label{eq: operator W_i,j bound}
    \end{align}
    Knowing that $\|\mf A\|_F=|\operatorname{vec}(\mf A)|$ for any matrix $\mf A$, \eqref{eq: operator W_i,j bound} and \eqref{eq:W norm4} imply that
    \begin{equation}
        \left\| \mathcal{P}_{j-l}\left[\operatorname{vec}\left(\mf W(j) \mf W(j)^\top\right)\right]\right\|=O\left(\frac{1}{\sqrt{M}} \xi_{K,n}^2\Delta_{K,n}^{\alpha} \sum_{k=l}^{l+M-1}\chi^{\alpha k}\right).\label{eq: operator f-norm bound}
    \end{equation}
    Note that $\mathcal{P}_{j-l}\left[\operatorname{vec}\left(\mf W(j) \mf W(j)^\top\right)\right]$ for $1\leq j\leq n-M+1$ is martingale difference sequence, using the fact, by Burkholder's and \eqref{eq: operator f-norm bound} we have
    \begin{align}
        \left\| \frac{1}{n-M+1}\sum_{j=1}^{n-M+1}\mathcal{P}_{j-l}\left[\operatorname{vec}\left(\mf W(j) \mf W(j)^\top\right)\right] \right\|^2 &\leq \frac{2}{(n-M+1)^2} \sum_{j=1}^{n-M+1}\left\| \mathcal{P}_{j-l}\left[\operatorname{vec}\left(\mf W(j) \mf W(j)^\top\right)\right]\right\|^2 \notag\\
        &= O\left\{ \frac{\xi_{K,n}^4\Delta_{K,n}^{2\alpha}}{(n-M+1)M} \left(\sum_{j=l}^{l+M-1}\chi^{\alpha k}\right)^2 \right\},\label{eq: sum operator W_i,j bound}
    \end{align}
    Therefore,
    \begin{align}
        \sqrt{\mb E\left( \left\|\tilde{\boldsymbol{\Sigma}}_M-\mr P\tilde{\boldsymbol{\Sigma}} \right\|_F^2\right)}&=\left\| \sum_{l=0}^\infty\frac{1}{n-M+1}\sum_{j=1}^{n-M+1} \mathcal{P}_{j-l}\left[\operatorname{vec}\left(\mf W(j) \mf W(j)^\top\right)\right] \right\|\notag\\
        &=O\left(  \xi_{K,n}^2\Delta_{K,n}^{\alpha}\sqrt{\frac{M}{n}} \right),\notag
    \end{align}
    which yields \eqref{eq: covariance comparison-element}.
    \end{proof}
\begin{lemma}
\label{lem: comparison between Pcov and cov}
    Under Assumptions \hyperref[(A1)]{(A1)}-\hyperref[(A4)]{(A4)} and \hyperref[(B1)]{(B1)}-\hyperref[(B3)]{(B3)}, we have for any given $\alpha\in(0,1)$,
    \begin{equation}
        \left\| \mb E\tilde{\boldsymbol{\Sigma}}_M-\boldsymbol{\Sigma}_n \right\|_F=O\left( \frac{\xi_{K,n}^2\Delta_{K,n}^\alpha}{M} \right).\label{eq: rate of Pcov-cov}
    \end{equation}
\end{lemma}
\begin{proof}
    Note that $\{\bar Q_n^{-1}\mf z_i\}_{i=1}^n$ is stationary with mean zero. Let $\Gamma(l)$ be its $l$th autocovariance, then we have
    \begin{align}
        &\mb E \left\{M\mf W(j)\mf W(j)^\top\right\}=\sum_{s=j}^{j+M-1}\sum_{t=j}^{j+M-1}\mb E\left\{ \bar Q_n^{-1}\mf z_{m(s-1)+j} \left( \bar Q_n^{-1}\mf z_{m(t-1)+j} \right)^\top\right\}=\sum_{l=1-M}^{M-1}\sum_{s=(1-M)\vee 1}^{(M-l)\wedge M}\Gamma(l),\label{eq: acf for tilde_Sigma}\\
        &\mb E \left\{n\mf W_n\mf W_n^\top\right\}=\sum_{s=1}^{n}\sum_{t=1}^{n}\mb E\left\{ \bar Q_n^{-1}\mf z_s \left( \bar Q_n^{-1}\mf z_t \right)^\top\right\}=\sum_{l=1-n}^{n-1}\sum_{s=(1-l)\vee 1}^{(n-l)\wedge n}\Gamma(l)\label{eq: acf for Sigma}.
    \end{align}
    Combining \eqref{eq:def tilde_Sigma}, \eqref{eq: acf for tilde_Sigma} and \eqref{eq: acf for Sigma}, we have
    \begin{align}
        \left\| \mb E\tilde{\boldsymbol{\Sigma}}_M-\boldsymbol{\Sigma}_n \right\|_F&= \left\| \frac{1}{(n-M+1)M}\sum_{j=1}^{n-M+1}\mb E \left\{M\mf W(j)\mf W(j)^{\top}\right\} - \frac{1}{n}\mb E \left\{n\mf W_n\mf W_n^\top\right\} \right\|_F,\notag\\
        &=\left\|  \frac{1}{M}\sum_{l=1-M}^{M-1}\sum_{s=(1-l)\vee 1}^{(M-l)\wedge M}\Gamma(l)-\frac{1}{n} \sum_{l=1-n}^{n-1}\sum_{s=(1-l)\vee 1}^{(n-l)\wedge n}\Gamma(l) \right\|_F,\notag\\
        &\leq 2\left\|   \sum_{l=0}^{M-1}\left(\frac{l}{n}-\frac{l}{M}\right)\Gamma(l)+\sum_{l=M}^{n-1} \frac{n-l}{n}\Gamma(l)   \right\|_F,\notag\\
        &\lesssim \frac{1}{M}\sum_{l=0}^{M-1}\|\Gamma(l)\|_F+\sum_{l\geq M }\|\Gamma(l)\|_F.\label{eq: comparison between Pcov and cov}
    \end{align}
    Note that $\mathcal{P}_k$ is orthogonal, by \eqref{eq:phd on z-4} in Lemma \ref{lem:phd on z}, we have
    \begin{align}
        \|\Gamma(l)\|_F&=\left\|\mb E\left\{\sum_{k\in \mb Z}  \left(\mathcal{P}_k\bar Q_n^{-1}\mf z_i\right) \left(  \mathcal{P}_k\bar Q_n^{-1}\mf z_{i+l} \right)^\top\right\} \right\|_F,\notag\\
        &\leq  \sum_{k\in\mb Z}\| \mathcal{P}_k \mf z_i \| \|\mathcal{P}_k \mf z_{i+l} \|, \notag\\
        &\leq \sum_{k\in \mb Z} \delta_{\mf z}(k,2)\delta_{\mf z}(k+l,2)=O( \xi_{K,n}^2\Delta_{K,n}^\alpha \chi^{\alpha l}).\label{eq: bound Gamma}
    \end{align}
    Combining \eqref{eq: comparison between Pcov and cov} and \eqref{eq: bound Gamma}, \eqref{eq: rate of Pcov-cov} holds.
\end{proof}

\section{Appendix}
\label{sec:appendix}

\subsection{Examples of time series models}
\label{sec:example of time series models}
We show that the following examples can all be rewritten in the form of \eqref{eq:causal representation}. 
\begin{example}[VARMA model]
Let $\mf e_t, t \in \mathbb{Z}$, be i.i.d. $d$-dimensional random vectors with mean 0 and $\|\mf e_t\|_q<\infty$; Let $\mf x_t$ be the $d$-variate ARMA process of orders $p$ and $q$, i.e.
\begin{equation}
    \mf A_0\mf x_t+\mf A_1\mf x_{t-1}+\dots+\mf A_p\mf x_{t-p}=\mf M_0\mf e_t+\mf M_1\mf e_{t-1}+\dots+\mf M_q\mf e_{t-q},\label{eq:VARMA}
\end{equation}
where $\mf A_0,\dots,\mf A_p,\mf M_0,\dots,\mf M_q$ are $d\times d$ matrices and $\mf A_0=\mf M_0=I_d$. The equation \eqref{eq:VARMA} can be written in summary notation as $\mf A(L)\mf x_t=\mf M(L)\mf e_t,$
where $L$ is the lag operator, which has the effect that $L \mf x_t=\mf x_{t-1}$, and where $\mf A(z)=\mf A_0+\mf A_1 z+\cdots+\mf A_p z^p$ and $\mf M(z)=\mf M_0+\mf M_1 z+\cdots+\mf M_q z^q$ are matrix-valued polynomials assumed to be of full rank. The $\{\mf x_t\}$ is stationary if and only if $\operatorname{det} (\mf A(z))\neq 0$ for all $z\in(-1,1)$; see \cite{hannan2012statistical}, \cite{LUTKEPOHL2006287}. Therefore, let $\mathcal{H}_t=(\dots,\mf e_{t-1},\mf e_{t})$, then the stationary $\{\mf x_t\}$ can be written as 
$$
\mf x_t=\mf H(\mathcal{H}_t),\quad \mf H(\mathcal{H}_t) = \Phi_0\mf e_t+\Phi_1 \mf e_{t-1}+\Phi_2 \mf e_{t-2}+\dots,
$$
wherein the matrices $\Phi_j$ are determined by the equation $\mf A(z)\Phi(z)=\mf M(z)$ with $\Phi(z)=\sum_{j=0}^\infty\Phi_j z^j$. Let $\{\mf e_t^\prime\}$ be the i.i.d. copy of $\{\mf e_t\}$, we can obtain the physical dependence measure of $\mf H(\mathcal{H}_t)$, i.e. 
$$
    \delta_{\mf H}(k,q)=\| \mf H(\mathcal{H}_0)-\mf H(\mathcal{H}_{0,k})  \|_q=\| \Phi_k\mf e_{-k}-\Phi_k\mf e_{-k}^\prime \|_q=\|\Phi_k(\mf e_1-\mf e_2)\|_q.
$$
\end{example}

\begin{example}[GARCH model]
Let $\varepsilon_t, t \in \mathbb{Z}$, be i.i.d. random variables with mean 0 and variance 1. Let $y_t=\sqrt{x_t} \varepsilon_t$ and $x_t=\alpha_0+\alpha_1 y_{t-1}^2+\cdots+\alpha_q y_{t-q}^2+\beta_1 x_{t-1}+\cdots+\beta_p x_{t-p}$ be the generalized autoregressive conditional heteroscedastic model $\operatorname{GARCH}(p, q)$, where $\alpha_0>0, \alpha_j \geqslant 0$ for $1 \leqslant j \leqslant q$ and $\beta_i \geqslant 0$ for $1 \leqslant i \leqslant p$. Then $\left\{y_t\right\}$ is stationary if $\sum_{j=1}^q \alpha_j+\sum_{i=1}^p \beta_i<1$; see \cite{bollerslev1986generalized}. Let $\mf z_t=\left(y_t^2, \ldots, y_{t-q+ 1}^2, x_t, \ldots, x_{t-p+1}\right)^\top, \quad \mf b_t=\left(\alpha_0 \varepsilon_t^2, 0, \ldots, 0, \alpha_0, 0, \ldots, 0\right)^\top$ and $\theta=\left(\alpha_1, \ldots, \alpha_q, \beta_1, \ldots, \beta_p\right)^{\mathrm{T}}$; let $\mf u_i=(0, \ldots, 0,1,0, \ldots, 0)^{\mathrm{T}}$ be the unit column vector with $i$ th element being $1,1 \leqslant i \leqslant p+q$. It is well known that GARCH models admit the following representation \cite{bougerol1992stationarity}:
\begin{equation}
    \mf z_t=\mf M_t \mf z_{t-1}+\mf b_t \text {, where } \mf M_t=\left(\theta \varepsilon_t^2, \mf u_1, \ldots, \mf u_{q-1}, \theta, \mf u_{q+1}, \ldots, \mf u_{p+q-1}\right)^\top.\label{eq:GARCH}
\end{equation}
Note that both $\mf M_t$ and $\mf b_t$ are measurable functions of $\varepsilon_t$, thus $\mf z_t$ and $y_t$ can be written as measurable functions of $\mathcal{G}_t=(\dots,\varepsilon_{t-1},\varepsilon_t)$, i.e.
$$
y_t=G(\mathcal{G}_t)=\sqrt{\mf u_{q+1}^\top \mf Z(\mathcal{G}_t)}\varepsilon_t,\quad \mf z_t=\mf Z(\mathcal{G}_t)=\mf b_t+\Phi_1\mf b_{t-1}+\Phi_2\mf b_{t-2}+\dots,
$$
where $\Phi_j=\prod_{i=t-j+1}^t\mf M_i$. The physical dependence measure of $y_t$ can be specified and bounded by the largest eigenvalue of $(\mf M^\top \mf M)^{1/2}$; see \cite{wu2005linear}, \cite{Wu2011AsymptoticTF}, and \cite{Wu2011gaussian} for details.
\end{example}

\subsubsection{Verification of the dependence Assumption \hyperref[(B2)]{(B2)}}
In this section, we give a loss-specific explanation for why it is reasonable to assume Assumption \hyperref[(B2)]{(B2)}. Recall the form of \eqref{eq:causal representation}, the time series $\{\bar\varepsilon_i\}$ can be written as $\bar\varepsilon_i=\bar G(\mathcal{H}_i,\mathcal{G}_i)$ where $\bar G(\cdot,\cdot)$ is an unknown measurable function. 

\paragraph{Quantile loss.} For $\psi_\tau(v)=\tau-1(v<0)$,
$\delta_{\psi}(k,4)^4 = \mr P\bigl(\bar\varepsilon_i\bar\varepsilon_{i,k}<0\bigr).$
Hence, for any $\eta>0$,
\begin{align*}
\mr P(\bar\varepsilon_i\bar\varepsilon_{i,k}<0)&\le \mr P(|\bar\varepsilon_i|\le \eta)
+ \mr P(|\bar\varepsilon_i-\bar\varepsilon_{i,k}|>\eta) \le C\eta + \eta^{-1}\|\bar\varepsilon_i-\bar\varepsilon_{i,k}\|_1,
\end{align*}
where the first term uses boundedness of the shifted conditional density near $0$ and $\bar\varepsilon_{i,k}=\bar G(\mathcal{H}_{i,k},\mathcal{G}_{i,k})$.
If $\|\bar\varepsilon_i-\bar\varepsilon_{i,k}\|_1=O(\chi^k)$, i.e., $\delta_{\bar G}(k,1)=: \|\bar G(\mathcal{H}_i,\mathcal{G}_i)-\bar G(\mathcal{H}_{i,k},\mathcal{G}_{i,k})\|_1=O(\chi^k)$, we can choose
$\eta\asymp \chi^{k/2}$ to obtain $\delta_{\psi}(k,4)=O(\tilde\chi^k)$ for some $\tilde\chi\in(0,1)$. Thus Assumption \hyperref[(B2)]{(B2)} is natural for quantile regression.

\paragraph{Huber loss.}
Since the Huber score is globally $1$-Lipschitz, $|\psi_c(\bar\varepsilon_i)-\psi_c(\bar\varepsilon_{i,k})|
\le |\bar\varepsilon_i-\bar\varepsilon_{i,k}|.$ Therefore $\delta_{\psi}(k,4)
\le \|\bar\varepsilon_i-\bar\varepsilon_{i,k}\|_4.$
Thus $\|\bar G(\mathcal{H}_i,\mathcal{G}_i)-\bar G(\mathcal{H}_{i,k},\mathcal{G}_{i,k})\|_4=O(\chi^k)$ immediately
implies Assumption \hyperref[(B2)]{(B2)}.

\paragraph{Expectile loss.}
The expectile score $\psi(\cdot)$ is globally Lipschitz with constant $2\max\{\tau,1-\tau\}$, and hence $\delta_{\psi}(k,4)
\le 2\max\{\tau,1-\tau\}\,\|\bar\varepsilon_i-\bar\varepsilon_{i,k}\|_4.$ Therefore Assumption \hyperref[(B2)]{(B2)} follows from the same geometric dependence $\|\bar G(\mathcal{H}_i,\mathcal{G}_i)-\bar G(\mathcal{H}_{i,k},\mathcal{G}_{i,k})\|_4=O(\chi^k)$.
\paragraph{$L_q$ loss.}
For $1<q\le2$, the score satisfies the H\"older inequality
$|\psi_q(x)-\psi_q(y)|\le C_q |x-y|^{q-1}$. Consequently,
\begin{align*}
\delta_{\psi}(k,4)^4
\le C_q \mb E|\bar\varepsilon_i-\bar\varepsilon_{i,k}|^{4(q-1)} 
\le C_q \Bigl(\mb E|\bar\varepsilon_i-\bar\varepsilon_{i,k}|^4\Bigr)^{q-1},
\end{align*}
where the last line follows from Lyapunov's inequality. Therefore
$\delta_{\psi}(k,4)
\le C_q \|\bar\varepsilon_i-\bar\varepsilon_{i,k}\|_4^{q-1}.$ Again, a geometric physical dependence bound for $\bar G(\mathcal{H}_i,\mathcal{G}_i)$ implies Assumption \hyperref[(B2)]{(B2)}.

\subsection{Verification of Assumption \hyperref[(B3)]{(B3)}}
\label{sec:supp-verify B3}
In this section, we take the quantile, Huber, expectile, and $L_q$ $(1<q\leq 2)$ regressions as examples to verify Assumption \hyperref[(B3)]{(B3)}. For notational convenience, we write
$$
\bar F_i(u)=:\mr P(\bar\varepsilon_i\leq u|\mf X_i),\quad \bar f_i(u)=:\frac{\partial}{\partial u}\bar F_i(u),
$$
whenever the conditional density exists. Recall $|\bar\varepsilon_i-\varepsilon_i|\leq a_n$ and $a_n\to 0$ thanks to the sieve approximation. Hence, all local conditions below can be evaluated on a uniformly $o(1)$-shifted neighborhood. Specifically, we check the following three generic conditions. Suppose that for all
$|x|,|u|\le \delta_0$,
\begin{align}
\bar \Xi_n(x+u\mid \mf X_i)-\bar\Xi_n(x\mid \mf X_i)-\bar\Xi_n^{(1)}(x\mid \mf X_i)u &= O(u^2),
\tag{B.3a}\label{eq:B3a-shifted}\\
\bar\Xi_n(\theta^\top\mf b_\omega(\mf X_i)\mid \mf X_i)-\bar\Xi_n^{(1)}(0\mid \mf X_i)\theta^\top\mf b_\omega(\mf X_i)
&= O\bigl((\theta^\top\mf b_\omega(\mf X_i))^2\bigr),
\tag{B.3b}\label{eq:B3b-shifted}\\
\bar\Xi_n^{(1)}(\theta^\top\mf b_\omega(\mf X_i)\mid \mf X_i)-\bar\Xi_n^{(1)}(0\mid \mf X_i)
&= O\bigl(\omega(t(\theta))\bigr),
\tag{B.3c}\label{eq:B3c-shifted}
\end{align}
where $t(\theta)=\sup_x|\theta^\top\mf b_\omega(x)|$ and $\omega(\cdot)$ is a modulus function.
Then Assumption \hyperref[(B3)]{(B3)}(i) follows immediately from Assumption \hyperref[(A2)]{(A2)}, because
\[
\mb E\bigl|h^\top\mf b_\omega(\mf X_i)\bigr|^2 \lesssim |h|^2,
\qquad
\mb E\bigl|\theta^\top\mf b_\omega(\mf X_i)\bigr|^2 \lesssim |\theta|^2,
\]
and therefore
\begin{align*}
&\mb E\Bigl[\bigl\{\bar\Xi_n(\theta^\top\mf b_\omega(\mf X_i)\mid \mf X_i)-\bar\Xi_n^{(1)}(0\mid \mf X_i)\theta^\top\mf  b_\omega(\mf X_i)\bigr\}
h^\top\mf b_\omega(\mf X_i)\Bigr] \\
&\qquad\lesssim \mb E\Bigl[|\theta^\top\mf b_\omega(\mf X_i)|^2\,|h^\top\mf b_\omega(\mf X_i)|\Bigr] \\
&\qquad\le t(\theta)
\Bigl(\mb E|\theta^\top\mf b_\omega(\mf X_i)|^2\Bigr)^{1/2}
\Bigl(\mb E|h^\top\mf b_\omega(\mf X_i)|^2\Bigr)^{1/2}
\lesssim t(\theta)|\theta||h|,
\end{align*}
and likewise
\[
\mb E\Bigl[\bigl\{\bar\Xi_n^{(1)}(\theta^\top\mf b_\omega(\mf X_i)\mid \mf X_i)-\bar\Xi_n^{(1)}(0\mid \mf X_i)\bigr\}
(h^\top\mf b_\omega(\mf X_i))^2\Bigr]
\lesssim \omega(t(\theta))|h|^2.
\]
Thus, in the loss-specific calculations below, it suffices to prove
\eqref{eq:B3a-shifted}--\eqref{eq:B3c-shifted} for verifying \hyperref[(B3)]{(B3)}(i).

\paragraph{Quantile regression.}
For the $\tau$-quantile regression $\psi_\tau(v)=\tau-1(v<0)$, we have
\[
\bar\Xi_n(x\mid \mf X_i)
= E\bigl[\psi_\tau(\bar\varepsilon_i+x)-\psi_\tau(\bar\varepsilon_i)\mid \mf X_i\bigr]
= \bar F_i(0)-\bar F_i(-x),
\]
\[
\bar\Xi_n^{(1)}(x\mid \mf X_i)=\bar f_i(-x)=f_{\varepsilon_i\mid \mf X_i}(-x-r_i),
\]
whenever the conditional density exists with $r_i=\bar\varepsilon_i-\varepsilon_i$. Assume that there exists $\delta_0>0$ such that, for all sufficiently large $n$, the conditional density of $\varepsilon_i$ given $\mf X_i$ satisfies
\begin{enumerate}
\item[(Q1)]\phantomsection\label{(Q1)} $f_{\varepsilon_i\mid \mf X_i}(u)$ exists on $[-\delta_0,\delta_0]$ almost surely,
\item[(Q2)]\phantomsection\label{(Q2)} $f_{\varepsilon_i\mid \mf X_i}(u)$ is bounded away from $0$ and $\infty$ on interval $[-\delta_0,\delta_0]$,
uniformly almost surely,
\item[(Q3)]\phantomsection\label{(Q3)} $f_{\varepsilon_i\mid \mf X_i}(u)$ is locally Lipschitz on interval $[-\delta_0,\delta_0]$, uniformly almost surely.
\end{enumerate}
Since $a_n\to 0$, for all sufficiently large $n$, the shifted arguments $r_i+x$ with $|x|\leq \delta_0/2$ remain in $[-\delta_0,\delta_0]$. 
Then, for $|x|,|u|\le \delta_0/2$ and all sufficiently large $n$,
\begin{align*}
&\bar\Xi_n(x+u\mid \mf X_i)-\bar\Xi_n(x\mid \mf X_i)-\bar\Xi_n^{(1)}(x\mid \mf X_i)u \\
&\qquad=
\int_0^u \bigl\{f_{\varepsilon_i\mid \mf X_i}(-x-r_i-s)-f_{\varepsilon_i\mid \mf X_i}(-x-r_i)\bigr\}\,\mr ds.
\end{align*}
By the local Lipschitz continuity of $f_{\varepsilon_i\mid \mf X_i}$, the integrand is
$O(|s|)$ uniformly over $|r_i|\le a_n$, hence
\[
\bar\Xi_n(x+u\mid \mf X_i)-\bar\Xi_n(x\mid \mf X_i)-\bar\Xi_n^{(1)}(x\mid \mf X_i)u = O(u^2),
\]
which proves \eqref{eq:B3a-shifted}. Likewise, by Taylor expansion around $0$,
\[
\bar\Xi_n(\theta^\top\mf b_\omega(\mf X_i)\mid \mf X_i)-\bar\Xi_n^{(1)}(0\mid \mf X_i)\theta^\top\mf b_\omega(\mf X_i)
= O\bigl((\theta^\top\mf b_\omega(\mf X_i))^2\bigr),
\]
which is \eqref{eq:B3b-shifted}. Finally,
\[
\bar\Xi_n^{(1)}(\theta^\top\mf b_\omega(\mf X_i)\mid \mf X_i)-\bar\Xi_n^{(1)}(0\mid \mf X_i)
= f_{\varepsilon_i\mid \mf X_i}(-r_i-\theta^\top\mf b_\omega(\mf X_i)) - f_{\varepsilon_i\mid \mf X_i}(-r_i)
= O\bigl(t(\theta)\bigr),
\]
which proves \eqref{eq:B3c-shifted} with $\omega(t)=t$.

For Assumption \hyperref[(B3)]{(B3)}(ii), note that $\bar\Xi_n^{(1)}(0\mid X_i)=f_{\varepsilon_i\mid X_i}(-r_i),$
which is bounded away from $0$ almost surely.
For Assumption \hyperref[(B3)]{(B3)}(iii), let $|u|\le t$. Since
\[
\bigl[\psi_\tau(\bar\varepsilon_i+u)-\psi_\tau(\bar\varepsilon_i-u)\bigr]^2
= 1\{-u\le \bar\varepsilon_i < u\}
\le 1\{|\bar\varepsilon_i|\le |u|\},
\]
we have, conditionally on $\mf X_i$,
\begin{align*}
\mb E\Bigl(\bigl[\psi_\tau(\bar\varepsilon_i+u)-\psi_\tau(\bar\varepsilon_i-u)\bigr]^2\mid X_i\Bigr)\le \mr  P(|\bar\varepsilon_i|\le |u|\mid \mf X_i)= \bar F_i(|u|)-\bar F_i(-|u|)
\le 2\sup_{|v|\le |u|}\bar f_i(v)\,|u|,
\end{align*}Taking the supremum over $|u|\le t$, since $f_{\varepsilon_i|\mf X_i}(v)$ is bounded away from 0 and $\infty$ in \hyperref[(Q2)]{(Q2)}, it yields  $\eta=1$ in Assumption \hyperref[(B3)]{(B3)}.

\paragraph{Huber regression.}
For Huber loss with threshold $c>0$,
$\psi_c(v)=v1(|v|\le c)+c\,\mathrm{sign}(v)1(|v|>c)$ and $\bar \Xi_n(x\mid \mf X_i)=\mb E\bigl[\psi_c(\bar\varepsilon_i+x)-\psi_c(\bar\varepsilon_i)\mid\mf  X_i\bigr].$
For every $x$ at which the derivative exists, we have
\[
\bar\Xi_n^{(1)}(x\mid \mf X_i)
= \mr P(|\bar\varepsilon_i+x|\le c\mid X_i)
= \bar F_i(c-x)-\bar F_i(-c-x).
\]
Assume that there exists $\delta_0>0$ such that, for all sufficiently large $n$,
\begin{enumerate}
\item[(H1)]\phantomsection\label{(H1)} $f_{\varepsilon_i\mid \mf X_i}(u)$ exists and is locally Lipschitz on the two neighborhoods of $\pm c$, i.e., $[c-\delta_0,c+\delta_0]$ and $[-c-\delta_0,-c+\delta_0]$.
\item[(H2)]\phantomsection\label{(H2)}
$\mr P\bigl(-c-r_i \le \varepsilon_i \le c-r_i\mid X_i\bigr) \ge c_0$ for some constant $c_0>0$ almost surely.
\end{enumerate}
Then
\[
\bar\Xi_n^{(1)}(x+u\mid \mf X_i)-\bar \Xi_n^{(1)}(x\mid \mf X_i)
=
\{\bar F_i(c-x-u)-\bar F_i(c-x)\} - \{\bar F_i(-c-x-u)-\bar F_i(-c-x)\},
\]
which is $O(|u|)$ uniformly over $|x|,|u|\le \delta_0/2$ by local Lipschitz continuity of the
density near $\pm c$. Therefore
\[
\bar\Xi_n(x+u\mid \mf X_i)-\bar\Xi_n(x\mid \mf X_i)-\bar\Xi_n^{(1)}(x\mid \mf X_i)u = O(u^2),
\]
which proves \eqref{eq:B3a-shifted}. The same local Lipschitz property yields
\[
\bar\Xi_n(\theta^\top \mf b_\omega(\mf X_i)\mid \mf X_i)-\bar\Xi_n^{(1)}(0\mid \mf X_i)\theta^\top \mf b_\omega(\mf X_i)
= O\bigl(|\theta^\top \mf b_\omega(\mf X_i)|^2\bigr),
\]
and
\[
\bar \Xi_n^{(1)}(\theta^\top \mf b_\omega(\mf  X_i)\mid\mf X_i)-\bar\Xi_n^{(1)}(0\mid \mf X_i)
= O\bigl(t(\theta)\bigr).
\]
Thus \eqref{eq:B3b-shifted} and \eqref{eq:B3c-shifted} follow with $\omega(t)=t$. Assumption \hyperref[(B3)]{(B3)}(ii) follows from \hyperref[(H2)]{(H2)} by the fact
$$
\bar\Xi_n^{(1)}(0\mid X_i)
= \mr P(|\bar\varepsilon_i|\le c\mid X_i)
= \mr P(-c-r_i\le \varepsilon_i \le c-r_i\mid X_i)\ge c_0.
$$
For Assumption \hyperref[(B3)]{(B3)}(iii), note that $\psi(\cdot)$ is globally $1$-Lipschitz. Therefore,
for every $u\in\mathbb R$,
\[
|\psi_c(\bar\varepsilon_i+u)-\psi_c(\bar\varepsilon_i-u)| \le 2|u|,
\]
which implies $\eta=2$.

\paragraph{Expectile regression.}
For the expectile loss at level $\tau\in(0,1)$, $\psi_\tau(v)=2\tau v\,1(v\ge0) + 2(1-\tau)v\,1(v<0)$. Then
\[
\bar \Xi_n(x\mid \mf X_i)=\mb E\bigl[\psi_\tau(\bar\varepsilon_i+x)-\psi_\tau(\bar\varepsilon_i)\mid\mf X_i\bigr],
\]
and differentiation gives
\[
\bar \Xi_n^{(1)}(x\mid \mf X_i)
= 2\tau + 2(1-2\tau)\bar F_i(-x)
= 2\tau + 2(1-2\tau)F_{\varepsilon_i\mid X_i}(-x-r_i).
\]
Assume that $f_{\varepsilon_i\mid \mf X_i}(u)$ is locally bounded on a neighborhood of
$[-\delta_0,\delta_0]$, uniformly almost surely. Then
$\bar \Xi_n^{(2)}(x\mid \mf X_i) = -2(1-2\tau)\bar f_i(-x)$
is locally bounded for $|x|\le \delta_0/2$, and hence Taylor's theorem yields
\[
\bar \Xi_n(x+u\mid \mf X_i)-\bar \Xi_n(x\mid \mf X_i)-\bar \Xi_n^{(1)}(x\mid \mf X_i)u = O(u^2),
\]
which proves \eqref{eq:B3a-shifted}. Moreover,
\[
\bar \Xi_n(\theta^\top \mf b_\omega(\mf X_i)\mid \mf X_i)-\bar\Xi_n^{(1)}(0\mid \mf X_i)\theta^\top \mf b_\omega(\mf X_i)
= O\bigl((\theta^\top \mf b_\omega(\mf X_i))^2\bigr),
\]
and
\[
\bar \Xi_n^{(1)}(\theta^\top \mf b_\omega(\mf X_i)\mid\mf X_i)-\bar\Xi_n^{(1)}(0\mid \mf X_i)
= O\bigl(t(\theta)\bigr),
\]
so \eqref{eq:B3b-shifted} and \eqref{eq:B3c-shifted} follow with $\omega(t)=t$.

For Assumption \hyperref[(B3)]{(B3)}(ii), observe that
$\bar\Xi_n^{(1)}(0\mid \mf X_i)=2\tau + 2(1-2\tau)\bar F_i(0).$
Since $\bar F_i(0)\in[0,1]$, we have the uniform bound
$2\min\{\tau,1-\tau\}
\le \Xi_n^{(1)}(0\mid X_i)
\le 2\max\{\tau,1-\tau\}$, which implies the required positivity.

For Assumption \hyperref[(B3)]{(B3)}(iii), the expectile score is globally Lipschitz with constant$2\max\{\tau,1-\tau\}$. Hence
\[
|\psi_\tau(\bar\varepsilon_i+u)-\psi_\tau(\bar\varepsilon_i-u)|
\le 4\max\{\tau,1-\tau\}|u|,
\]
and therefore the expectile case also corresponds to $\eta=2$.

\paragraph{$L_q$ regression.}
For the $L_q$ regression $\psi_q(v)=q|v|^{q-1}\,\mathrm{sign}(v)$,
$1<q\le 2$. When $q=2$, this is ordinary least squares and all the conclusions below are immediate. Therefore it suffices to consider $1<q<2$. Then
\[
\bar \Xi_n(x\mid \mf X_i)
= q\int
\Bigl(|x-v|^{q-1}\mathrm{sign}(x-v)-|-v|^{q-1}\mathrm{sign}(-v)\Bigr)
\bar f_i(v)\mr dv,
\]
and $\Xi_n^{(1)}(x\mid X_i)
= q(q-1)\int_{\mathbb R}|x-v|^{q-2}\bar f_i(v)\,\mr dv.$ Suppose that
\begin{enumerate}
\item[(L1)]\phantomsection\label{(L1)} $\bar f_i(v)$ is bounded on $\mathbb R$, uniformly almost surely,
\item[(L2)]\phantomsection\label{(L2)} $\bar f_i(v)$ is locally Lipschitz on a neighborhood of $0$, uniformly almost surely.
\end{enumerate}
Since $\bar f_i(v)=f_{\varepsilon_i\mid X_i}(v-r_i)$ and $|r_i|\le a_n=o(1)$, conditions
\hyperref[(L1)]{(L1)}-\hyperref[(L2)]{(L2)} follow, for all sufficiently large $n$, from the corresponding conditions on $f_{\varepsilon_i\mid X_i}$ on a fixed neighborhood of $0$.
To control the variation of $\bar \Xi_n^{(1)}$, notice that
\[
\bar \Xi_n^{(2)}(x\mid \mf X_i)
= c_q\int_0^{\infty} r^{q-3}
\bigl\{\bar f_i(x+r)-\bar f_i(x-r)\bigr\}\mr dr,
\]
where $c_q$ is a finite constant depending only on $q$. The odd difference inside the
integral is crucial. For small $r$, local Lipschitz continuity gives $|\bar f_i(x+r)-\bar f_i(x-r)| \le 2Lr$, so the integrand is bounded by $Cr^{q-2}$, which is integrable near $0$ because $q>1$. For large $r$, boundedness of $\bar f_i$ implies integrability because $q<2$. Hence
$\bar \Xi_n^{(2)}(x\mid\mf X_i)$ is locally bounded near $0$, uniformly almost surely. It follows that
$\bar \Xi_n^{(1)}(x\mid\mf X_i)$ is locally Lipschitz near $0$, and consequently
\[
\bar\Xi_n(x+u\mid\mf X_i)-\bar \Xi_n(x\mid \mf X_i)-\bar \Xi_n^{(1)}(x\mid \mf X_i)u = O(u^2),
\]
which proves \eqref{eq:B3a-shifted}. The same local Lipschitz continuity yields
\[
\bar \Xi_n(\theta^\top \mf b_\omega(\mf X_i)\mid \mf X_i)-\bar \Xi_n^{(1)}(0\mid \mf X_i)\theta^\top\mf b_\omega(\mf X_i)
= O\bigl((\theta^\top \mf b_\omega(\mf X_i))^2\bigr),
\]
and
\[
\bar \Xi_n^{(1)}(\theta^\top \mf b_\omega(\mf X_i)\mid \mf X_i)-\bar \Xi_n^{(1)}(0\mid \mf X_i)
= O\bigl(t(\theta)\bigr).
\]
Thus \eqref{eq:B3b-shifted} and \eqref{eq:B3c-shifted} also hold with $\omega(t)=t$.

For Assumption \hyperref[(B3)]{(B3)}(ii), notice that $\bar \Xi_n^{(1)}(0\mid \mf X_i)
= q(q-1)\int_{\mathbb R}|v|^{q-2}\bar f_i(v)\,\mr dv.$
Since $q>1$, the kernel $|v|^{q-2}$ is locally integrable around $0$. Therefore
$\bar \Xi_n^{(1)}(0\mid \mf X_i)$ is finite under \hyperref[(L1)]{(L1)}. It is strictly positive whenever the conditional law of $\bar\varepsilon_i$ given $\mf X_i$ is not degenerate. If a uniform positive lower bound is desired, it suffices to assume that $\bar f_i(v)$ is bounded below by a positive constant on a fixed neighborhood of $0$.

For Assumption \hyperref[(B3)]{(B3)}(iii), fix $u\in\mathbb R$ and split the event according to $|\bar\varepsilon_i|\le 2|u|$ and $|\bar\varepsilon_i|>2|u|$. On the local region $|\bar\varepsilon_i|\le 2|u|$, using the fact $|\psi_q(\bar\varepsilon_i+u)-\psi_q(\bar\varepsilon_i-u)|
\le C_q |u|^{q-1}$, we have
\[
\mb E\Bigl(\bigl[\psi_q(\bar\varepsilon_i+u)-\psi_q(\bar\varepsilon_i-u)\bigr]^2
1\{|\bar\varepsilon_i|\le 2|u|\}\Bigm|\mf X_i\Bigr)
\le C_q |u|^{2q-2} \mr P(|\bar\varepsilon_i|\le 2|u|\mid \mf X_i)
= O(|u|^{2q-1}).
\]
On the complement $|\bar\varepsilon_i|>2|u|$, using the fact
$|\psi_q(\bar\varepsilon_i+u)-\psi_q(\bar\varepsilon_i-u)|
\le C_q |u|\,|\bar\varepsilon_i|^{q-2}$, we can imply that
$$
\mb E\Bigl(\bigl[\psi_q(\bar\varepsilon_i+u)-\psi_q(\bar\varepsilon_i-u)\bigr]^2
1\{|\bar\varepsilon_i|>2|u|\}\Bigm|\mf X_i\Bigr) \le C_q u^2\mb E\bigl(|\bar\varepsilon_i|^{2q-4}1\{|\bar\varepsilon_i|>2|u|\}\mid \mf X_i\bigr).
$$
If $3/2<q\le2$, then $2q-4>-1$, and boundedness of $\bar f_i$ near $0$ implies that the last
conditional expectation is bounded uniformly for small $u$. Hence the contribution of the
outer region is $O(u^2)$. If $1<q\le3/2$, then $2q-4\le -1$, and a direct integral calculation shows that
\[
\mb E\bigl(|\bar\varepsilon_i|^{2q-4}1\{|\bar\varepsilon_i|>2|u|\}\mid\mf  X_i\bigr)
= O(|u|^{2q-3}),
\]
so the outer contribution is $O(|u|^{2q-1})$. Combining the two regions and taking the supremum over $|u|\le t$, we can bound the exponent in Assumption \hyperref[(B3)]{(B3)}(iii) as
\[
\eta=
\begin{cases}
2, & 3/2<q\le2,\\[4pt]
2q-1, & 1<q\le3/2.
\end{cases}
\]

\subsection{Verification of Assumption \hyperref[(A4)]{(A4)}}
\label{sec:example of loss}
In this section, we take the Huber's, quantile, expectile, and $L_q$ $(1<q\leq 2)$ regressions as examples to verify Assumption \hyperref[(A4)]{(A4)}. We first introduce the definition of covering number and \textit{VC}-dimension for a function class as follows.
\begin{definition}[Covering number]
\label{def:covering}
    Let $(T,d)$ be an arbitrary semi-metric space where $T\subset \Xi$ and $d(\cdot,\cdot)$ is defined on the space $\Xi$. A $\varepsilon$-cover ($\varepsilon>0$) of the set $T$ with respect to semi-metric $d$ is a set $\{T_1,\dots,T_N\}\subset \Xi$ such that for any $T_0\in T$, there exists some $i\in\{1,\dots,N\}$ with $d(T_0,T_i)\leq \varepsilon$. Then the $\varepsilon$-covering number of $T$ with respect to $d(\cdot,\cdot)$ is
the minimal number of balls $B(x;\varepsilon)=\{y\in T:d(x,y)\leq \varepsilon\}$ of radius $\varepsilon$ needed to cover the set $T$, i.e.,
        $N(\varepsilon,T,d)=:\inf_N\{N\in\mb N:\text{there exists a $\varepsilon$-cover $\{T_1,\dots,T_N\}\subset T$ } \}.$
\end{definition}

\begin{definition}
Let $\mathcal{C}$ be a collection of subsets of a set $\mathcal{X}$. Let $\left\{x_1, \ldots, x_n\right\} \subset \mathcal{X}$ be an arbitrary set of $n$ points. Say that $\mathcal{C}$ picks out a certain subset $A$ of $\left\{x_1, \ldots, x_n\right\}$ if $A$ can be expressed as $C \cap\left\{x_1, \ldots, x_n\right\}$ for some $C \in \mathcal{C}$. The collection $\mathcal{C}$ is said to shatter $\left\{x_1, \ldots, x_n\right\}$ if each of its $2^n$ subsets can be picked out in this manner.
The \textit{VC}-dimension $V(\mathcal{C})$ of the class $\mathcal{C}$ is the largest $n$ such that some set of size $n$ is shattered by $\mathcal{C}$.

The subgraph of $f: \mathcal{X} \rightarrow \mathbb{R}$ is a subset of $\mathcal{X} \times \mathbb{R}$ defined as
$$
C_f=:\{(x, t) \in \mathcal{X} \times \mathbb{R}: t<f(x)\}.
$$
For the collection $\mathcal{F}$ of measurable functions on $\mathcal{X}$, the \textit{VC}-dimension $V(\mathcal{F})$ of the function class $\mathcal{F}$ is defined as the \textit{VC}-dimension of the class $\{C_f:f\in \mathcal{F}\}$.
\end{definition}

To facilitate maximal inequality in the empirical process, we shall check whether our function family is VC-class and specify its size accompanied with $K$ so we can bound the covering number by following lemmas. We introduce Lemma \ref{lem: bound cover number}, \ref{lem: VC vec space}, \ref{lem: VC operation}, \ref{lem: bound plus cover number} to bound the covering number with \textit{VC}-dimension. Lemma \ref{lem:maximal ineq uniform entropy} shows how this covering number bound works in the maximal inequality. Consider a basic function family
$$
\mathcal{F}_1=\{|\varepsilon-\theta^\top\mf b_\omega(\mf x)|: \theta\in \mb R^K\}.
$$
By Lemma \ref{lem: VC vec space} and Lemma \ref{lem: VC operation}, we can have $V(\mathcal{F}_1)\leq K+1$. For function family in Assumption \hyperref[(A4)]{(A4)}
$$
\mathcal{F}_\rho=\{ f_\theta(\varepsilon,\mf x)=\rho(\varepsilon-\theta^\top\mf b_\omega(\mf x)):\theta\in\Theta \},
$$
where function $f_\theta:\mb R^{d+1}\rightarrow\mb R$ is indexed by $\theta$ and $\rho(\cdot)$ is the loss function including Huber's, quantile, expectile, and $L_q$ $(1<q\leq 2)$ regression. By the strictly monotonicity of $\rho(x),x\geq 0$, for any $n$ points $\{(\varepsilon_i,\mf x_i)\}_{i=1}^n$ and any values $t_1,\dots,t_n\in \mb R$, we have
$$
f_\theta(\varepsilon_i,\mf x_i)\leq t_i \iff \rho(\varepsilon_i-\theta^\top\mf b_\omega(\mf x_i))\leq t_i\iff |\varepsilon_i-\theta^\top\mf b_\omega(\mf x_i)|\leq t_i.
$$
By the definition of \textit{VC}-dimension, we have $V(\mathcal{F}_\rho)\lesssim K$, which ensures Assumptions \hyperref[(A4)]{(A4)}.

\begin{lemma}[Theorem 7.12. in \cite{sen2018gentle}]
    \label{lem: bound cover number}
     For a $V C$ class of functions $\mathcal{F}$ with measurable envelope function $F$ and \textit{VC}-dimension $V(\mathcal{F})$, one has for any probability measure $Q$ with $\|F\|_{Q, r}>0$, $r \geq 1$,
    $$
    N\left(\epsilon\|F\|_{Q, r}, \mathcal{F}, \|\cdot\|_{Q,r}\right) \leq C V(\mathcal{F})(4 e)^{V(\mathcal{F})}\left(\frac{2}{\epsilon}\right)^{r V(\mathcal{F})},
    $$
    for a universal constant $C>0$ and $0<\epsilon<1$.
\end{lemma}

\begin{lemma}[Lemma 7.15 in \cite{sen2018gentle}]
\label{lem: VC vec space}
    Any finite-dimensional vector space $\mathcal{F}$ of measurable functions $f: \mathcal{X} \rightarrow \mathbb{R}$ is $V C$ subgraph of dimension smaller than or equal to $\operatorname{dim}(\mathcal{F})+1$.
\end{lemma}

\begin{lemma}[Lemma 7.19 in \cite{sen2018gentle}]
\label{lem: VC operation}
Let $\mathcal{F}$ and $\mathcal{G}$ be VC subgraph classes of functions on a set $\mathcal{X}$ and $g: \mathcal{X} \rightarrow$ $\mathbb{R}, \phi: \mathbb{R} \rightarrow \mathbb{R}$, and $\psi: \mathcal{Z} \rightarrow \mathcal{X}$ fixed functions. Then, (i) $\mathcal{F} \wedge \mathcal{G}=\{f \wedge g: f \in \mathcal{F} ; g \in \mathcal{G}\}$ is VC subgraph;
(ii) $\mathcal{F} \vee \mathcal{G}$ is $V C$ subgraph;
(iii) $\{\mathcal{F}>0\}=:\{\{f>0\}: f \in \mathcal{F}\}$ is $V C$;
(iv) $-\mathcal{F}$ is $V C$;
(v) $\mathcal{F}+g=:\{f+g: f \in \mathcal{F}\}$ is VC subgraph;
(vi) $\mathcal{F} \cdot g=\{f g: f \in \mathcal{F}\}$ is VC subgraph;
(vii) $\mathcal{F} \circ \psi=\{f(\psi): f \in \mathcal{F}\}$ is VC subgraph;
(viii) $\phi \circ \mathcal{F}$ is VC subgraph for monotone $\phi$.
\end{lemma}

\begin{lemma}[Lemma 7.21.(i). in \cite{sen2018gentle}]
\label{lem: bound plus cover number}
Fix $r \geq 1$. Suppose that $\mathcal{F}$ and $\mathcal{G}$ are classes of measurable functions with envelopes $F$ and $G$ respectively. Then, for every $0<\epsilon<1$, the covering number
$$
N\left(2 \epsilon\|F+G\|_{Q, r}, \mathcal{F}+\mathcal{G}, \|\cdot\|_{Q,r}\right) \leq N\left(\epsilon\|F\|_{Q, r}, \mathcal{F}, \|\cdot\|_{Q,r}\right) \cdot N\left(\epsilon\|G\|_{Q, r}, \mathcal{G}, \|\cdot\|_{Q,r}\right).
$$
\end{lemma}

\subsection{Verification of Assumption \hyperref[(A5)]{(A5)} under H\"older smoothness}
\label{sec:supp-verify A5}
In this section, we verify Assumption \hyperref[(A5)]{(A5)} by showing that one may take \(\varsigma\leq p/d\) when \(Q_0\) belongs to a \(p\)-smooth H\"older class and the sieve basis is chosen from the standard families considered in Section \ref{sec:examples of sieve basis}. For simplicity, suppose that the covariate support is \(\mathcal X=[0,1]^d\) and take \(\mathcal D_n=\mathcal X\), so that \(\mf b_\omega(x)=\mf b(x)\). Assumption \hyperref[(A5)]{(A5)} requires
\[
  \sup_{\mf x\in[0,1]^d}|Q_{0,n}(\mf x)-Q_0(\mf x)|=O(K^{-\varsigma})
\]
for some \(\varsigma>0\). Let \(p=s+\alpha\), where \(s\) is a nonnegative integer and
\(\alpha\in(0,1]\). Assume that \(Q_0\in\mathcal H^p([0,1]^d,L)\), namely,
all mixed partial derivatives \(D^aQ_0\) with \(|a|\le s\) are uniformly
bounded, and the derivatives of order \(s\) satisfy
\[
  \sup_{\mf x\ne \mf x'}
  \frac{|D^aQ_0(\mf x)-D^aQ_0(\mf x')|}{|\mf x-\mf x'|^\alpha}
  \le L,\qquad |a|=s .
\]
Under this condition, there exists \(q_K^*(\mf x)=\theta_K^{*\top}\mf b(\mf x)\) such that 
\begin{equation}
      \sup_{\mf x\in[0,1]^d}|q_K^*(\mf x)-Q_0(\mf x)|\lesssim K^{-p/d},\label{eq:best approximation}
\end{equation}
if $\mf b(\mf x)$ is the tensor product of sieve bases such as Legendre polynomials, orthogonal wavelets, or Fourier series if it is periodic; see \cite{timan2014theory}, \cite{meyer1992ondelettes}, and \cite{chen2007large}. Write $e_K(\mf x)=q_K^*(\mf x)-Q_0(\mf x)$, $\nu_K(\mf x)=Q_{0,n}(\mf x)-q_K^*(\mf x)$ and denote $\|\cdot\|_\infty$ denotes $\|e_K\|_\infty=\sup_{\mf x\in[0,1]^d}|e_K(\mf x)|$. 

Define the local curvature
$$
\lambda_0(s)=:-\frac{\partial}{\partial u}\mb E[\psi(\varepsilon_i-u)|\mf X_i=s)]\Big|_{u=0}.
$$
Under mild condition, we can have $0<c_\lambda\leq \lambda_0(s)\leq C_\lambda<\infty$ uniformly over $s\in[0,1]^d$ and for all sufficiently small $|u|$,
\begin{equation}
\mb E[\psi(\varepsilon_i-u)|\mf X_i=s)]=-\lambda_0(s)u+r_0(s,u),\quad \sup_{s\in[0,1]^d}r_0(s,u)=O(u^2).\label{eq:unshifted local expand}    
\end{equation}
As discussed in Section \ref{sec:supp-verify B3}, this condition holds for quantile regression under a uniformly bounded and Lipschitz conditional density near zero, and for the Huber, expectile, least-squares, and $L_q$ losses under the corresponding local differentiability conditions used in Section \ref{sec:supp-verify B3}. 

Let $R(q)=\mb E[\rho(Y_i-q(\mf X_i))]$. Combining \eqref{eq:unshifted local expand}, \eqref{eq:best approximation}, and Assumption \hyperref[(A2)]{(A2)}, by similar arguments in the proof of Proposition \ref{prop:quadra approx}, we can have
\begin{align}
       R(q)-R(Q_0)&=\mb E \int_0^{u_q(\mf X_i)} \lambda_0(\mf X_i)t-r_0(\mf X_i,t)\mr dt,  \notag\\
       &=\frac{1}{2}\mb E[\lambda_0(\mf X_i)u_q(\mf X_i)^2]+o(\|u_q(\mf X_i)\|^2),\label{eq:quadra approx Q0}
\end{align}
where $u_q(\mf x)=q(\mf x)-Q_0(\mf x)$. Using the fact $R(Q_{0,n})\leq R(q^*_K)$ and $\lambda_0(s)\geq c_\lambda>0$, \eqref{eq:quadra approx Q0} with \eqref{eq:best approximation} yields
\begin{equation}
    \|Q_{0,n}(\mf X_i)-Q_0(\mf X_i)\|\lesssim \|q_K^*(\mf X_i)-Q_0(\mf X_i)\|=O(K^{-p/d}).
\end{equation}
Therefore, $\|\nu_K\|_\infty=O(\xi_{K,n}K^{-p/d})=o(1)$ for typical $\xi_{K,n}\lesssim \sqrt{K}$ and $p/d>1/2$.

Since $Q_{0,n}$ is the population sieve target as defined in \eqref{eq:sieve basis}, the first-order condition implies that
$$
\mb E[\mf b(\mf X_i)\psi(Y_i-Q_{0,n}(\mf X_i))]=0.
$$
By \eqref{eq:basic model}, we have
\begin{equation}
\mb E[\mf b(\mf X_i)\psi(\varepsilon_i-u_K(\mf X_i))]=0,\label{eq:first-order uk}    
\end{equation}
where $u_K(\mf X_i)=e_K(\mf X_i)+\nu_K(\mf X_i)$. Combining \eqref{eq:first-order uk} and \eqref{eq:unshifted local expand}, we have
\begin{equation}
    \mb E[\lambda_0(\mf X_i)\mf b(\mf X_i)u_K(\mf X_i)]=\mb E(\mf b(\mf X_i)r_0(\mf X_i,u_K(\mf X_i))).\label{eq:first-order expand uk}
\end{equation}
Let $Q_\lambda=\mb E[\lambda_0(\mf X_i)\mf b(\mf X_i)\mf b(\mf X_i)^\top]$. By Assumption \hyperref[(A2)]{(A2)}, the eigenvalues of $Q_\lambda$ are bounded away from zero and infinity. Define the actual weighted population projection
$$
\Pi_{K,\lambda} g(s)=: \mf b(s)^\top Q_\lambda^{-1}\mb E[\lambda_0(\mf X_i)\mf b(\mf X_i)g(\mf X_i)].
$$
Assume the projection $\Pi_{K,\lambda}$ is stable from $L^\infty$ to $L^\infty$, i.e.,
\begin{equation}
\sup_{\|g\|_{\infty}\leq 1} \|\Pi_{K,\lambda} g\|_{\infty}\leq L_K.\label{eq:stable projection}    
\end{equation}
Combining \eqref{eq:stable projection} and \eqref{eq:first-order expand uk}, with the lower bound on $\lambda_0$, it yields
\begin{equation}
    \|\nu_K\|_\infty\lesssim L_K\|e_K\|_\infty+L_K(\|e_K\|_\infty+\|\nu_K\|_\infty)^2.
\end{equation}
If $L_K=O(1)$ (or $L_K$ grows only logarithmically), using the fact $\|e_K\|_\infty=O(K^{-p/d})$ and $L_K\|\nu_K\|_\infty=o(1)$, we finally have
$$
\sup_{\mf x\in[0,1]^d}|Q_0(\mf x)-Q_{0,n}(\mf x)|\le \|e_K\|_\infty+\|\nu_K\|_\infty=O(L_KK^{-p/d}),
$$
and Assumption \hyperref[(A5)]{(A5)} holds for every $\varsigma\leq p/d$ (or $\varsigma <p/d$).

The verification of \eqref{eq:stable projection} and specific $\varsigma$ for the listed basis families is as follows:
\begin{itemize}
  \item \textbf{Legendre polynomial partitions.}
  Partition \([0,1]^d\) into \(J^d\) quasi-uniform cells and use, on each cell,
a tensor-product polynomial basis of fixed degree \(r\), represented by
Legendre polynomials on a reference cell. This is a local polynomial partition basis written
in an orthogonal polynomial form. Since the projection is cellwise and the
dimension of the polynomial space on each cell is fixed, Assumption \hyperref[(A2)]{(A2)} and the uniform bounds on \(\lambda_0(\mf x)\) imply that the weighted
\(L^2\)-projection has a uniformly bounded \(L^\infty\)-operator norm. Hence $L_K=O(1)$ and the standard local polynomial approximation bound gives
  $\|e_K\|_\infty\lesssim J^{-\min(p,r+1)}
  \asymp K^{-\min(p,r+1)/d}$. Therefore, if \(r+1\ge p\), Assumption \hyperref[(A5)]{(A5)} holds with $\varsigma=p/d$.
  \item \textbf{Compactly supported wavelet bases}. For boundary-corrected compactly supported wavelets on \([0,1]^d\), or periodized wavelets in the periodic case, the \(L^2\)-projection is sup-norm stable under the usual bounded-design conditions. This is the bounded Lebesgue-constant condition used in \cite{CHEN2015uniformsieve}. Therefore $L_K=O(1)$. If the wavelet regularity and the number of vanishing moments exceed \(p\),
  then Assumption \hyperref[(A5)]{(A5)} holds with $\varsigma=p/d$.
  \item \textbf{Fourier/trigonometric polynomial bases.}
  In the periodic constant-weight case, the \(L^2\)-projection onto
\(
  \mathcal T_m=\mathrm{span}\{e^{i\ell t}:|\ell|\le m\}
\)
is the ordinary Fourier partial-sum operator
\(
  S_m g(t)=\frac{1}{2\pi}\int_{-\pi}^{\pi}D_m(t-z)g(z)\,dz,
\)
where
\(
  D_m(t)=\sum_{|\ell|\le m}e^{i\ell t}
\)
is the Dirichlet kernel. 
The classical Lebesgue-constant bound for Fourier partial sums gives 
\(
  \|S_m g\|_\infty
  \lesssim
  \log m\|g\|_\infty
\); see details in \cite{zygmund2002trigonometric}. Therefore, for tensor-product Fourier bases with \(K\asymp m^d\),
$L_K\lesssim(\log m)^d\lesssim(\log K)^d$. Combining this with the direct approximation bound
\(\|e_K\|_\infty\lesssim K^{-p/d}\) for periodic H\"older functions gives
$
  \sup_{\mf x\in[0,1]^d}|Q_{0,n}(\mf x)-Q_0(\mf x)|
  =
  O\{K^{-p/d}(\log K)^d\}
$. Since \((\log K)^d=O(K^\epsilon)\) for every fixed \(\epsilon>0\),
Assumption \hyperref[(A5)]{(A5)} holds for every \(\varsigma<p/d\).
\end{itemize}

\subsection{Examples of sieve basis}
\label{sec:examples of sieve basis}
In this section, we introduce the commonly used sieve basis and verify the condition \eqref{eq:critical value condition} in Proposition \ref{prop:critical value}. For simplicity, we only verify $\int_\mathcal{X}\lambda_{\min}^{d/2}(\mf M^\top \mf M)\mr d\mf x\gtrsim n^{\underline{c}}$ in condition \eqref{eq:critical value condition}. 
We refer to Assumption 4 of \cite{CHEN2015uniformsieve} and Example 1-2 in \cite{quan2024JASA} for the rest of the polynomial rate conditions in \eqref{eq:critical value condition}. 
We verify condition \eqref{eq:critical value condition} using trigonometric basis functions (see Example \ref{eg:fourier}) as a representative example. A similar procedure applies to other types of basis functions. 
Suppose that $d=1, \mathcal{S}=[-\pi,\pi]$, the number of basis functions is $ K=2\tilde K+1, \tilde K\ge 1$, and $\Phi(s)=(1,\sin(s),\cos(s),\cdots,\sin(\tilde Ks),\cos(\tilde Ks))^\top$. Then $|\Phi(s)|=\sqrt{\tilde K+1}$, and
$$
\mf M(s)=(0,\cos(s),-\sin(s),\cdots,\tilde K\cos(\tilde Ks),-\tilde K\sin(\tilde Ks))^\top/\sqrt{\tilde K+1},
$$ 
with $\mf M(s)^\top \mf M(s)=\sum_{i=1}^{\tilde K} i^2/(\tilde K+1)\gtrsim {\tilde K}^2$. As a result, we have $\int_{\mathcal{S}}\lambda_{min}(\mf M^\top(s)\mf M(s))\mr ds\gtrsim  K\gtrsim n^{\bar{\omega}}$ for some $\bar{\omega}>0$ s.t. $n^{\bar \omega}\ll K$.

    Now consider the case where $d=2$ and $\mathcal{S}=[-\pi,\pi]^2$. Suppose that the number of basis functions is $ K=(2\tilde K+1)^2$. Then $\Phi(s)=\phi(s_1)\otimes\phi(s_2)$ where 
    $$
    \phi(s)=(1,\sin(s),\cos(s),\cdots,\sin(\tilde Ks),\cos(\tilde Ks))^\top.
    $$
    and $M_1(s)=\psi(s_1)\otimes\phi(s_2)/(\tilde K+1)$, $M_2(s)=\phi(s_1)\otimes\psi(s_2)/(\tilde K+1)$
    where $\psi(s)=(0,\cos(s),-\sin(s),\cdots,\tilde K\cos(\tilde Ks),-\tilde K\sin(\tilde Ks))^\top$. 
    \begin{align*}
        \mf M^\top \mf M&=\frac{1}{(\tilde K+1)^2}\begin{pmatrix}
            \psi(s_1)^\top\psi(s_1)\phi(s_2)^\top\phi(s_2) & \psi(s_1)^\top\phi(s_1)\phi(s_2)^\top\psi(s_2)\\
            \psi(s_2)^\top\phi(s_2)\phi(s_1)^\top\psi(s_1)& \phi(s_1)^\top\phi(s_1)\psi(s_2)^\top\psi(s_2)
        \end{pmatrix}\\
        &=\frac{1}{(\tilde K+1)^2}\frac{1}{6}\tilde K(\tilde K+1)^2(2\tilde K+1)\begin{pmatrix}
            1 & 0\\
            0&  1
        \end{pmatrix}=\frac{1}{6}\tilde K(2\tilde K+1)\begin{pmatrix}
            1 & 0\\
            0&  1
        \end{pmatrix}.
    \end{align*}
    Then we have $\int_{\mathcal{S}}\lambda_{min}(\mf M^\top\mf M)\mr ds\gtrsim K\gtrsim n^{\bar \omega}$.
\begin{example}[Legendre]
\label{eg:legendre}
    Define Legendre polynomials 
    $$
        P_j(x)=\frac{1}{2^j j!}\frac{\mr d^j}{\mr d x^j} (x^2-1)^j,\quad x\in[-1,1].
    $$
    Then continuous function $f(x)$ on $[-1,1]$ can be written as
    $$
    f(x)=\sum_{j=0}^\infty a_jP_j(x),\quad a_j=(j+\frac{1}{2})\int_{-1}^1 P_j(x)f(x)\mr dx.
    $$
\end{example}

\begin{example}[Fourier] 
\label{eg:fourier}
Consider real-valued function $f(x)\in L_2[-1,1]$ i.e. $\int_{-1}^1 f(x)\mr dx<\infty$. By Fourier transformation, $f(x)$ can be written as 
$$
f(x)=\sum_{j=-\infty}^\infty a_j \phi_j(x),\quad a_j=\int_{\mb -1}^1 \phi_j(x)f(x)\mr d x
$$
where $\{\phi_j(x)\}_{j=-\infty}^\infty=\{\left(\cos(j\pi x)+i\sin(j\pi x)\right)/\sqrt{2}\}_{j=-\infty}^\infty$ forms an orthonormal basis for $L_2[-1,1]$.
\end{example}

\begin{example}[Haar wavelet]
\label{eg:Haar wavelet}
The Haar sequence proposed in \cite{Haar1910ZurTD} gives an example of an orthonormal system for the space of square-integrable functions. For every pair $n, k$ of integers in $\mathbb{Z}$, the Haar function $h_{n, k}$ is defined on the real line $\mathbb{R}$ by the formula
$$
h_{n, k}(t)=2^{n / 2} h\left(2^n t-k\right),
$$
where $h(t)$ is the Haar wavelet's mother wavelet function
$$
h(t)= \begin{cases}1 & 0 \leq t<\frac{1}{2} \\ -1 & \frac{1}{2} \leq t<1 \\ 0 & \text { otherwise }\end{cases}
$$
The Haar system on the real line is the set of functions
$$
\left\{h_{n, k}(t): n \in \mathbb{Z}, k \in \mathbb{Z}\right\},
$$
which is an orthonormal basis.
\end{example}

\begin{example}[Daubechies wavelet]
For $N \in \mathbb{N}$, a Daubechies mother wavelet of class Daubechies-$N$ is a function $\phi \in L_2(\mathbb{R})$ defined by
$$
\phi(x)=:\sqrt{2} \sum_{k=1}^{2 N-1}(-1)^k h_{2 N-1-k} \varphi(2 x-k),
$$
where $h_0, h_1, \cdots, h_{2 N-1} \in \mathbb{R}$ are constant and satisfy $\sum_{k=0}^{N-1} h_{2 k}=\frac{1}{\sqrt{2}}=\sum_{k=0}^{N-1} h_{2 k+1}$, as well as, for $l=0,1, \cdots, N-1$,
$$
\sum_{k=2 l}^{2 N-1+2 l} h_k h_{k-2 l}= \begin{cases}1, & l=0 \\ 0, & l \neq 0\end{cases}
$$
The $\varphi(x)$ is the scaling wavelet function supported on $[0,2 N-1)$ and satisfies the recursion equation $\varphi(x)=\sqrt{2} \sum_{k=0}^{2 N-1} h_k \varphi(2 x-k)$, as well as the normalization $\int_{\mathbb{R}} \varphi(x) d x=1$, $\int_{\mathbb{R}} \varphi(2 x-k) \varphi(2 x-l) d x=0, k \neq l$. As listed in \cite{Daubechies1992wavelet}, the filter coefficients $h_0,\dots,h_{2N-1}$ can be efficiently computed. The order $N$ decides the support $[0,2 N-1)$ and provides the regularity condition
$$
\int_{\mathbb{R}} x^j \phi(x) \mr d x=0, j=0, \cdots, N.
$$
The Haar wavelet as introduced above can be regarded as a special Daubechies wavelet with $N=1$. In our simulations and data analysis, we employ Daubechies wavelet with a sufficiently high order $N$ to construct a sequence of orthogonal sieve basis as proposed in \cite{daubechies1988orthonormal}. For a given $J_n$ and $J_0$, we consider the following periodized wavelets on $[0,1]$
$$
\begin{gathered}
\left\{\varphi_{J_0 k}(x), 0 \leq k \leq 2^{J_0}-1 ; \phi_{j k}(x), J_0 \leq j \leq J_n-1,0 \leq k \leq 2^j-1\right\}, \text { where } \\
\varphi_{J_0 k}(x)=2^{J_0 / 2} \sum_{l \in \mathbb{Z}} \varphi\left(2^{J_0} x+2^{J_0} l-k\right), \phi_{j k}(x)=2^{j / 2} \sum_{l \in \mathbb{Z}} \psi\left(2^j x+2^j l-k\right)
\end{gathered}
$$
or equivalently, by \cite{meyer1992ondelettes},
$$
\left\{\varphi_{J_n k}(x), 0 \leq k \leq 2^{J_n-1}\right\} .
$$
The $2^J_n$ equals to our basis number $K$. Additionally, we refer to \cite{chen2007large} for a more general example of orthogonal wavelets.
\end{example}

\begin{footnotesize}
\spacingset{0.5}
\putbib
\end{footnotesize}
\end{bibunit}

@article{zhao2014asymptotics,
  title={Asymptotics of nonparametric L-1 regression models with dependent data},
  author={Zhao, Zhibiao and Wei, Ying and Lin, Dennis KJ},
  journal={Bernoulli},
  volume={20},
  number={3},
  pages={1532},
  year={2014}
}

@article{Wu2017quantile,
author = {Weichi Wu and Zhou Zhou},
title = {Nonparametric Inference for Time-Varying Coefficient Quantile Regression},
journal = {Journal of Business \& Economic Statistics},
volume = {35},
number = {1},
pages = {98-109},
year = {2017},
publisher = {Taylor & Francis},
}

@article{article,
author = {Hjort, Nils and Pollard, David},
year = {2011},
month = {07},
pages = {},
title = {Asymptotics for Minimisers of Convex Processes}
}

@misc{sen2018gentle,
  title={A gentle introduction to empirical process theory and applications},
  author={Sen, Bodhisattva},
  year={2018},
  publisher={Apr}
}

@article{dahlhaus2019towards,
  title={Towards a general theory for nonlinear locally stationary processes},
  author={Dahlhaus, Rainer and Richter, Stefan and Wu, Wei Biao},
  journal={Bernoulli},
  volume={25},
  number={2},
  pages={1013--1044},
  year={2019},
  publisher={Bernoulli Society for Mathematical Statistics and Probability}
}

@article{wu2005nonlinear,
  title={Nonlinear system theory: Another look at dependence},
  author={Wu, Wei Biao},
  journal={Proceedings of the National Academy of Sciences},
  volume={102},
  number={40},
  pages={14150--14154},
  year={2005},
  publisher={National Acad Sciences}
}

@article{zhou2010Non,
author = {Zhou Zhou},
title = {{Nonparametric inference of quantile curves for nonstationary time series}},
volume = {38},
journal = {The Annals of Statistics},
number = {4},
publisher = {Institute of Mathematical Statistics},
pages = {2187 -- 2217},
year = {2010},
}

@article{Sun&Loader,
author = {Jiayang Sun and Clive R. Loader},
title = {{Simultaneous Confidence Bands for Linear Regression and Smoothing}},
volume = {22},
journal = {The Annals of Statistics},
number = {3},
publisher = {Institute of Mathematical Statistics},
pages = {1328 -- 1345},
year = {1994}
}

@article{wuzhou2023multiscale,
  title={Multiscale jump testing and estimation under
complex temporal dynamics
},
  author={Wu, W and Zhou, Z},
  journal={Bernoulli, to appear},
  year={2023}
}

@article{vogt2012nonparametric,
  title={NONPARAMETRIC REGRESSION FOR LOCALLY STATIONARY TIME SERIES},
  author={Vogt, Michael},
  journal={The Annals of Statistics},
  volume={40},
  number={5},
  pages={2601--2633},
  year={2012}
}

@article{Wu2011gaussian,
 title = {Gaussian Approximations for Non-stationary Multiple Time Series},
 author = {Wei Biao Wu and Zhou Zhou},
 journal = {Statistica Sinica},
 number = {3},
 pages = {1397--1413},
 volume = {21},
 year = {2011}
}

@article{karmakar2022simultaneous,
  title={Simultaneous inference for time-varying models},
  author={Karmakar, Sayar and Richter, Stefan and Wu, Wei Biao},
  journal={Journal of Econometrics},
  volume={227},
  number={2},
  pages={408--428},
  year={2022},
  publisher={Elsevier}
}

@article{Mies2022seq_high-dim,
author = {Fabian Mies and Ansgar Steland},
title = {Sequential Gaussian approximation for nonstationary time series in high dimensions},
volume = {29},
journal = {Bernoulli},
number = {4},
publisher = {Bernoulli Society for Mathematical Statistics and Probability},
pages = {3114 -- 3140},
year = {2023},
doi = {10.3150/22-BEJ1577},
}

@inproceedings{nazarov2003maximal,
  title={On the Maximal Perimeter of a Convex Set in $\mathbb{R}^{n}$ with Respect to a Gaussian Measure},
  author={Nazarov, Fedor},
  booktitle = {Geometric Aspects of Functional Analysis},
  pages={169--187},
  year={2003},
  publisher={Springer}
}

@book{gine2016mathematical,
  title={Mathematical Foundations of Infinite-Dimensional Statistical Models},
  author={Gin{\'e}, Evarist and Nickl, Richard},
  volume={40},
  year={2016},
  publisher={Cambridge University Press}
}

@inproceedings{pollard1990empirical,
  title={Empirical processes: theory and applications},
  author={Pollard, David},
  year={1990},
  organization={Ims}
}

@article{pollard1991asymptotics,
  title={Asymptotics for least absolute deviation regression estimators},
  author={Pollard, David},
  journal={Econometric Theory},
  volume={7},
  number={2},
  pages={186--199},
  year={1991},
  publisher={Cambridge University Press}
}

@misc{van1996weak,
  title={Weak convergence and empirical processes: with applications to statistics},
  author={Van Der Vaart, Aad W and Wellner, Jon A},
  year={1996},
  publisher={Springer}
}

@article{chen2007large,
  title={Large sample sieve estimation of semi-nonparametric models},
  author={Chen, Xiaohong},
  journal={Handbook of econometrics},
  volume={6},
  pages={5549--5632},
  year={2007},
  publisher={Elsevier}
}

@article{ding2021simultaneous,
  title={Simultaneous Sieve Inference for Time-Inhomogeneous Nonlinear Time Series Regression},
  author={Ding, Xiucai and Zhou, Zhou},
  journal={arXiv preprint arXiv:2112.08545},
  year={2021}
}

@book{meyer1992ondelettes,
author = {Meyer Yves},
address = {Paris},
booktitle = {Ondelettes et opérateurs . I . Ondelettes},
isbn = {2-7056-6125-5},
language = {fre},
publisher = {Hermann},
series = {Actualités mathématiques},
title = {Ondelettes et opérateurs . I, Ondelettes / Yves Meyer},
year = {1989},
}

@book{Daubechies1992wavelet,
author = {Daubechies, Ingrid},
title = {Ten Lectures on Wavelets},
publisher = {Society for Industrial and Applied Mathematics},
year = {1992},
eprint = {https://epubs.siam.org/doi/pdf/10.1137/1.9781611970104}
}

@article{Haar1910ZurTD,
  title={Zur Theorie der orthogonalen Funktionensysteme},
  author={Alfred Haar},
  journal={Mathematische Annalen},
  year={1910},
  volume={69},
  pages={331-371},
  url={https://api.semanticscholar.org/CorpusID:120024038}
}

@article{liu2024self,
      title={Self-convolved Bootstrap for M-regression under Complex Temporal Dynamics}, 
      author={Miaoshiqi Liu and Zhou Zhou},
      year={2024},
      journal={arXiv preprint arXiv:2310.11724}
}

@article{bentkus2003Essen-bound,
title = {On the dependence of the {Berry}–{Esseen} bound on dimension},
journal = {Journal of Statistical Planning and Inference},
volume = {113},
number = {2},
pages = {385-402},
year = {2003},
issn = {0378-3758},
author = {V. Bentkus},
}

@Article{Fang2016CLT,
  author={Xiao Fang},
  title={{A Multivariate CLT for Bounded Decomposable Random Vectors with the Best Known Rate}},
  journal={Journal of Theoretical Probability},
  year=2016,
  volume={29},
  number={4},
  pages={1510-1523},
  month={December}
}

@article{Wang2012Ultra-high-quantreg,
  title={Quantile regression for analyzing heterogeneity in ultra-high dimension},
  author={Wang, Lan and Wu, Yichao and Li, Runze},
  journal={Journal of the American Statistical Association},
  volume={107},
  number={497},
  pages={214--222},
  year={2012},
  publisher={Taylor \& Francis}
}

@article{Fang2024Larege-dim,
author = {Xiao Fang and Yuta Koike},
title = {{Large-dimensional central limit theorem with fourth-moment error bounds on convex sets and balls}},
volume = {34},
journal = {The Annals of Applied Probability},
number = {2},
publisher = {Institute of Mathematical Statistics},
pages = {2065 -- 2106},
year = {2024},
doi = {10.1214/23-AAP2014}
}

@article{quan2024JASA,
author = {Mingxue Quan and Zhenhua Lin},
title = {Optimal One-Pass Nonparametric Estimation Under Memory Constraint},
journal = {Journal of the American Statistical Association},
volume = {119},
number = {545},
pages = {285--296},
year = {2024},
publisher = {Taylor \& Francis},
doi = {10.1080/01621459.2022.2115374},
}

@article{CHEN2015uniformsieve,
title = {Optimal uniform convergence rates and asymptotic normality for series estimators under weak dependence and weak conditions},
journal = {Journal of Econometrics},
volume = {188},
number = {2},
pages = {447-465},
year = {2015},
issn = {0304-4076},
author = {Xiaohong Chen and Timothy M. Christensen}
}

@article{SU2016sieveIV,
title = {Sieve instrumental variable quantile regression estimation of functional coefficient models},
journal = {Journal of Econometrics},
volume = {191},
number = {1},
pages = {231-254},
year = {2016},
issn = {0304-4076},
author = {Liangjun Su and Tadao Hoshino}
}

@article{zhang2020non,
  title={On the non-asymptotic and sharp lower tail bounds of random variables},
  author={Zhang, Anru R and Zhou, Yuchen},
  journal={Stat},
  volume={9},
  number={1},
  pages={e314},
  year={2020},
  publisher={Wiley Online Library}
}

@article{politis2004automatic,
  title={Automatic block-length selection for the dependent bootstrap},
  author={Politis, Dimitris N and White, Halbert},
  journal={Econometric reviews},
  volume={23},
  number={1},
  pages={53--70},
  year={2004},
  publisher={Taylor \& Francis}
}

@article{chao2017confidence,
  title={Confidence corridors for multivariate generalized quantile regression},
  author={Chao, Shih-Kang and Proksch, Katharina and Dette, Holger and H{\"a}rdle, Wolfgang Karl},
  journal={Journal of Business \& Economic Statistics},
  volume={35},
  number={1},
  pages={70--85},
  year={2017},
  publisher={Taylor \& Francis}
}

@article{wu2018gradient,
author = {Weichi Wu and Zhou Zhou},
title = {{Gradient-based structural change detection for nonstationary time series M-estimation}},
volume = {46},
journal = {The Annals of Statistics},
number = {3},
publisher = {Institute of Mathematical Statistics},
pages = {1197 -- 1224},
year = {2018},
doi = {10.1214/17-AOS1582},
}

@article{hardle2010confidence,
  title={Confidence bands in quantile regression},
  author={H{\"a}rdle, Wolfgang K and Song, Song},
  journal={Econometric Theory},
  volume={26},
  number={4},
  pages={1180--1200},
  year={2010},
  publisher={Cambridge University Press}
}

@book{fan2008nonlinear,
  title={Nonlinear time series: nonparametric and parametric methods},
  author={Fan, Jianqing and Yao, Qiwei},
  year={2008},
  publisher={Springer Science \& Business Media}
}

@article{liu2010simultaneous,
 author = {Weidong Liu and Wei Biao Wu},
 journal = {The Annals of Statistics},
 number = {4},
 pages = {2388--2421},
 publisher = {Institute of Mathematical Statistics},
 title = {Simultaneous nonparametric inference of time series},
 urldate = {2024-09-09},
 volume = {38},
 year = {2010}
}

@article{cai2002regression,
  title={Regression quantiles for time series},
  author={Cai, Zongwu},
  journal={Econometric theory},
  volume={18},
  number={1},
  pages={169--192},
  year={2002},
  publisher={Cambridge University Press}
}

@article{tu2022quantile,
title = {Nonparametric inference for quantile cointegrations with stationary covariates},
journal = {Journal of Econometrics},
volume = {230},
number = {2},
pages = {453-482},
author = {Yundong Tu and Han-Ying Liang and Qiying Wang},
year = {2022},
issn = {0304-4076},
doi = {https://doi.org/10.1016/j.jeconom.2021.06.002}
}

@article{kong2010uniform,
  title={Uniform Bahadur representation for local polynomial estimates of M-regression and its application to the additive model},
  author={Kong, Efang and Linton, Oliver and Xia, Yingcun},
  journal={Econometric Theory},
  volume={26},
  number={5},
  pages={1529--1564},
  year={2010},
  publisher={Cambridge University Press}
}

@article{alman2022expectile,
author = {Almanjahie, Ibrahim M. and  Bouzebda, Salim and Zoulikha Kaid and  Laksaci, Ali},
title = {Nonparametric estimation of expectile regression in functional dependent data},
journal = {Journal of Nonparametric Statistics},
volume = {34},
number = {1},
pages = {250--281},
year = {2022},
publisher = {Taylor \& Francis},
doi = {10.1080/10485252.2022.2027412},
}

@article{chen2014seiveM,
title = {Sieve M inference on irregular parameters},
journal = {Journal of Econometrics},
volume = {182},
number = {1},
pages = {70-86},
year = {2014},
note = {Causality, Prediction, and Specification Analysis: Recent Advances and Future Directions},
issn = {0304-4076},
author = {Xiaohong Chen and Zhipeng Liao}
}

@article{bollerslev1986generalized,
  title={Generalized autoregressive conditional heteroskedasticity},
  author={Bollerslev, Tim},
  journal={Journal of econometrics},
  volume={31},
  number={3},
  pages={307--327},
  year={1986},
  publisher={Elsevier}
}

@incollection{LUTKEPOHL2006287,
title = {Chapter 6 Forecasting with VARMA Models},
editor = {G. Elliott and C.W.J. Granger and A. Timmermann},
series = {Handbook of Economic Forecasting},
publisher = {Elsevier},
volume = {1},
pages = {287-325},
year = {2006},
issn = {1574-0706},
author = {Helmut Lütkepohl},
}

@book{hannan2012statistical,
  title={The statistical theory of linear systems},
  author={Hannan, Edward James and Deistler, Manfred},
  year={2012},
  publisher={SIAM}
}

@article{bougerol1992stationarity,
  title={Stationarity of GARCH processes and of some nonnegative time series},
  author={Bougerol, Philippe and Picard, Nico},
  journal={Journal of econometrics},
  volume={52},
  number={1-2},
  pages={115--127},
  year={1992},
  publisher={Elsevier}
}

@article{wu2005linear,
  title={On linear processes with dependent innovations},
  author={Wu, Wei Biao and Min, Wanli},
  journal={Stochastic Processes and their Applications},
  volume={115},
  number={6},
  pages={939--958},
  year={2005},
  publisher={Elsevier}
}

@article{Wu2011AsymptoticTF,
  title={Asymptotic theory for stationary processes},
  author={Wei Biao Wu},
  journal={Statistics and Its Interface},
  year={2011},
  volume={4},
  pages={207-226},
}

@book{timan2014theory,
  title={Theory of approximation of functions of a real variable},
  author={Timan, Aleksandr Filippovich},
  year={2014},
  publisher={Elsevier}
}

@article{newey1997convergence,
  title={Convergence rates and asymptotic normality for series estimators},
  author={Newey, Whitney K},
  journal={Journal of econometrics},
  volume={79},
  number={1},
  pages={147--168},
  year={1997},
  publisher={Elsevier}
}

@article{wu2007m,
  title={M-estimation of linear models with dependent errors},
  author={Wu, Wei Biao},
  journal={Annals of statistics},
  volume={35},
  number={2},
  pages={495--521},
  year={2007}
}

@article{politis2001taper,
  title={Tapered block bootstrap},
  author={Paparoditis, Efstathios and Politis, Dimitris N},
  journal={Biometrika},
  volume={88},
  number={4},
  pages={1105--1119},
  year={2001},
  publisher={Oxford University Press}
}

@article{chapman2000short,
  title={Is the short rate drift actually nonlinear?},
  author={Chapman, David A and Pearson, Neil D},
  journal={The Journal of Finance},
  volume={55},
  number={1},
  pages={355--388},
  year={2000},
  publisher={Wiley Online Library}
}

@article{welsh1989m,
  title={On M-processes and M-estimation},
  author={Welsh, AH},
  journal={The Annals of Statistics},
  pages={337--361},
  year={1989},
  publisher={JSTOR}
}

@article{he2000parameters,
  title={On parameters of increasing dimensions},
  author={He, Xuming and Shao, Qi-Man},
  journal={Journal of multivariate analysis},
  volume={73},
  number={1},
  pages={120--135},
  year={2000},
  publisher={Elsevier}
}

@article{sun2020adaptive,
  title={Adaptive huber regression},
  author={Sun, Qiang and Zhou, Wen-Xin and Fan, Jianqing},
  journal={Journal of the American Statistical Association},
  volume={115},
  number={529},
  pages={254--265},
  year={2020},
  publisher={Taylor \& Francis}
}

@article{daubechies1988orthonormal,
  title={Orthonormal bases of compactly supported wavelets},
  author={Daubechies, Ingrid},
  journal={Communications on pure and applied mathematics},
  volume={41},
  number={7},
  pages={909--996},
  year={1988},
  publisher={Wiley Online Library}
}

@article{baur2019quantile,
  title={A quantile regression approach to estimate the variance of financial returns},
  author={Baur, Dirk G and Dimpfl, Thomas},
  journal={Journal of Financial Econometrics},
  volume={17},
  number={4},
  pages={616--644},
  year={2019},
  publisher={Oxford University Press}
}

@article{koenker2006quantile,
  title={Quantile autoregression},
  author={Koenker, Roger and Xiao, Zhijie},
  journal={Journal of the American statistical association},
  volume={101},
  number={475},
  pages={980--990},
  year={2006},
  publisher={Taylor \& Francis}
}

@article{baur2012stock,
  title={Stock return autocorrelations revisited: A quantile regression approach},
  author={Baur, Dirk G and Dimpfl, Thomas and Jung, Robert C},
  journal={Journal of Empirical Finance},
  volume={19},
  number={2},
  pages={254--265},
  year={2012},
  publisher={Elsevier}
}

@article{Chernouzhukov2014EP,
 author = {Victor Chernozhukov and Denis Chetverikov and Kengo Kato},
 journal = {The Annals of Statistics},
 number = {4},
 pages = {1564--1597},
 publisher = {Institute of Mathematical Statistics},
 title = {GAUSSIAN APPROXIMATION OF SUPREMA OF EMPIRICAL PROCESSES},
 urldate = {2024-12-19},
 volume = {42},
 year = {2014}
}

@article{xiao2004uqar,
title = "Unit root quantile autoregression inference",
author = "Roger Koenker and Zhijie Xiao",
year = "2004",
doi = "10.1198/016214504000001114",
language = "English (US)",
volume = "99",
pages = "775--787",
journal = "Journal of the American Statistical Association",
issn = "0162-1459",
publisher = "Taylor and Francis Ltd.",
number = "467",
}

@article{schwarz1978estimating,
  title={Estimating the dimension of a model},
  author={Schwarz, Gideon},
  journal={The annals of statistics},
  pages={461--464},
  year={1978},
  publisher={JSTOR}
}

@article{rissanen1978modeling,
  title={Modeling by shortest data description},
  author={Rissanen, Jorma},
  journal={Automatica},
  volume={14},
  number={5},
  pages={465--471},
  year={1978},
  publisher={Elsevier}
}

@article{chen2021multinon,
    author = {Chen, Likai and Smetanina, Ekaterina and Wu, Wei Biao},
    title = {Estimation of nonstationary nonparametric regression model with multiplicative structure},
    journal = {The Econometrics Journal},
    volume = {25},
    number = {1},
    pages = {176-214},
    year = {2021},
    month = {06},
    issn = {1368-4221},
    doi = {10.1093/ectj/utab018},
}

@article{liu2025wasserstein,
  title={Wasserstein and Convex Gaussian Approximations for Non-stationary Time Series of Diverging Dimensionality},
  author={Liu, Miaoshiqi and Yang, Jun and Zhou, Zhou},
  journal={arXiv preprint arXiv:2506.08723},
  year={2025}
}

@article{phandoidaen2022empirical,
  title={Empirical process theory for locally stationary processes},
  author={Phandoidaen, Nathawut and Richter, Stefan},
  journal={Bernoulli},
  volume={28},
  number={1},
  pages={453--480},
  year={2022},
  publisher={Bernoulli Society for Mathematical Statistics and Probability}
}

@article{wu2005bahadur,
  title={On the Bahadur representation of sample quantiles for dependent sequences},
  author={Wu, Wei Biao},
  journal={The Annals of Statistics},
  volume={33},
  number={4},
  pages={1934},
  year={2005},
  publisher={Institute of Mathematical Statistics}
}

@article{dehling2009new,
  title={New techniques for empirical processes of dependent data},
  author={Dehling, Herold and Durieu, Olivier and Volny, Dalibor},
  journal={Stochastic Processes and their Applications},
  volume={119},
  number={10},
  pages={3699--3718},
  year={2009},
  publisher={Elsevier}
}

@article{doukhan2007probability,
  title={Probability and moment inequalities for sums of weakly dependent random variables, with applications},
  author={Doukhan, Paul and Neumann, Michael H},
  journal={Stochastic Processes and their Applications},
  volume={117},
  number={7},
  pages={878--903},
  year={2007},
  publisher={Elsevier}
}

@article{ball1993reverse,
  title={The reverse isoperimetric problem for Gaussian measure},
  author={Ball, Keith},
  journal={Discrete \& Computational Geometry},
  volume={10},
  number={4},
  pages={411--420},
  year={1993},
  publisher={Springer-Verlag Berlin, Heidelberg}
}

@article{chen2011multivariate,
  title={Multivariate normal approximation by Stein's method: The concentration inequality approach},
  author={Chen, Louis HY and Fang, Xiao},
  journal={arXiv preprint arXiv:1111.4073},
  year={2011}
}

@article{lee2019martingale,
  title={Martingale decomposition and approximations for nonlinearly dependent processes},
  author={Lee, Ji Hyung},
  journal={Statistics \& Probability Letters},
  volume={152},
  pages={35--42},
  year={2019},
  publisher={Elsevier}
}

@article{zhang2021high,
  title={High-quantile regression for tail-dependent time series},
  author={Zhang, Ting},
  journal={Biometrika},
  volume={108},
  number={1},
  pages={113--126},
  year={2021},
  publisher={Oxford University Press}
}

@techreport{cattaneo2025uniform,
  title={Uniform Estimation and Inference for Nonparametric Partitioning-Based M-Estimators},
  author={Cattaneo, Matias and Feng, Yingjie and Shigida, Boris},
  year={2025},
  institution={arXiv. org}
}

@book{boucheron2013concentration,
  title={Concentration Inequalities: A Nonasymptotic Theory of Independence},
  author={Boucheron, S. and Lugosi, G. and Massart, P.},
  isbn={9780199535255},
  lccn={2012277339},
  year={2013},
  publisher={OUP Oxford}
}

@book{zygmund2002trigonometric,
  title={Trigonometric Series},
  author={Zygmund, A.},
  number={j. 1},
  isbn={9780521890533},
  lccn={2002067363},
  series={Cambridge Mathematical Library},
  year={2002},
  publisher={Cambridge University Press}
}
\end{document}